\documentclass{amsart}
\usepackage{amsmath,amssymb,amsthm,amscd}
\usepackage{pgfplots}
\usepackage{mathrsfs}
\newtheorem{formula}{}[section]
\newtheorem{propos}[formula]{Proposition}
\newtheorem{corollary}[formula]{Corollary}
\newtheorem{lemma}[formula]{Lemma}
\newtheorem{theorem}[formula]{Theorem}
\theoremstyle{definition}
\newtheorem{definition}[formula]{Definition}
\newtheorem{example}[formula]{Example}
\theoremstyle{remark}
\newtheorem*{remark}{Remark}

\title{Naturally graded Lie algebras (Carnot algebras) of slow growth}
\author{Dmitry V. Millionshchikov}
\thanks{This work is supported by the Russian Science Foundation under grant 14-11-00414.}
\address{Steklov Institute of Mathematics of RAS, 8 Gubkina St. Moscow,
119991, Russia and  Department of Mechanics and Mathematics, Moscow
State University, 1 Leninskie gory, 119992 Moscow, Russia}
\email{mitia\_m@hotmail.com}

\date{}

\begin{document}

\maketitle

\begin{abstract}
A nilpotent Lie algebra ${\mathfrak g}$ is said to be naturally graded if it is isomorphic to its associated graded Lie algebra ${\rm gr} \mathfrak{g}$ with respect to filtration by ideals of the lower central series. This concept is equivalent to the concept of the Carnot algebra arising in sub-Riemannian geometry and the geometric control theory.

We classify finite-dimensional and infinite-dimensional naturally graded Lie algebras (Carnot algebras)
${\mathfrak g}=\oplus_{i=1}^{{+}\infty}{\mathfrak g}_i$ with properties 
$$
[{\mathfrak g}_1, {\mathfrak g}_i]={\mathfrak g}_{i{+}1},  \; \dim{{\mathfrak g}_i}+\dim{{\mathfrak g}_{i{+}1}} \le 3, \; i \ge 1.
$$
For growth functions of such Lie algebras, we have the estimate $F(n) \le \frac{3}{2}n{+}1$.
\end{abstract}
\section*{Introduction}\label{s0}

Pro-nilpotent Lie algebras are an important and interesting class of Lie algebras, which generalizes the class of nilpotent Lie algebras. Recall that for the ideals $ {\mathfrak g}^1={\mathfrak g}, \dots, {\mathfrak g}^{i+1}=[{\mathfrak g}^1, {\mathfrak g}^i],\dots$ of the lower central series of an arbitrary pro-nilpotent Lie algebra ${\mathfrak g}$, the residual nilpotency condition $\cap_{i=1}^{+\infty}{\mathfrak g}^i=\{0\}$. One can call "tends to zero" such a behavior of a ideal ${\mathfrak g}^i$ only with a fair amount of fantasy,  for instance, a free finite-generated Lie algebra is also pro-nilpotent. Nevertheless, the condition of residual nilpotency is such that the ordinary, finite-dimensional, nilpotent Lie algebras are pro-nilpotent too. The second part of the definition of a pro-nilpotent Lie algebra is the condition that all its quotient Lie algebras of the form ${\mathfrak g}/{\mathfrak g}^i,i=1,2,\dots$ must have finite dimension. This condition allows us to endow an infinite-dimensional pro-nilpotent Lie algebra ${\mathfrak g}$ by the topology of the inverse limit of finite-dimensional spaces.

Let ${\mathfrak g}$ be a pro-nilpotent Lie algebra. The case when the Lie algebra ${\mathfrak g}$ is isomorphic to its associated graded Lie algebra ${\rm gr}{\mathfrak g}$ will be the main one for us. We shall call the Lie algebra ${\mathfrak g} $ for which ${\mathfrak g}\cong{\rm gr}{\mathfrak g}$, {\it a naturally graded Lie algebra}. In the finite-dimensional case this definition is equivalent to the definition of the so-called Carnot algebra. Carnot algebras, as a class of finite-dimensional positively graded Lie algebras, arise in the study of a number of problems of sub-Riemannian geometry and geometric control theory. Recall that a Carnot algebra is a positive-graded Lie algebra ${\mathfrak g}=\oplus_{i=1}^n{\mathfrak g}_i$ such that
$$
[{\mathfrak g}_1, {\mathfrak g}_i]={\mathfrak g}_{i{+}1},\; i=1,\dots,n-1,\quad
[{\mathfrak g}_1, {\mathfrak g}_i]=0, i > n. 
$$
Nilpotent naturally graded Lie algebras (Carnot algebras) are the first important class among nilpotent Lie algebras for one simple reason: an arbitrary nilpotent Lie algebra ${\mathfrak g}$ is a filtered deformation of its associated graded Lie algebra ${\rm gr}{\mathfrak g}$.
Another important feature of naturally graded Lie algebras is their role in studies related to the growth of Lie algebras.
It is easy to see that in the case of a pro-nilpotent Lie algebra ${\mathfrak g}$, its associated graded Lie algebra ${\rm gr}{\mathfrak g}$ grows exactly at the same rate as ${\mathfrak g}$.

When studying the growth of infinite-dimensional Lie algebras, we are between two extreme situations: at one pole there are free Lie algebras with exponential growth, on the other -- the so-called Lie algebras of maximal class (filiform Lie algebras) that grow slower than all the others. Over the years since Vergne's publication \cite{V}, filiform Lie algebras (Lie algebras of maximal class) have become a very popular object among algebraists and not only algebraists. Such attention can be explained by a number of their properties and applications. Filiform Lie algebras in a certain sense are generic nilpotent Lie algebras. Also, the Lie algebras of this class had important applications. In this subsection it is natural to recall Benois's paper \cite{Benoist}, in which, using a special $11$-dimensional graded filiform Lie algebra, he constructed an example of a nilmanifold that does not admit of any complete left-invariant affine structure (connection), thereby presenting a counterexample to one hypothesis Milnor \cite{Milnor}.

Attempts were made to generalize, in some reasonable sense, the class of filiform Lie algebras. Such generalizations were usually based on the concept of the length of the lower central series, since for the finite-dimensional filiform Lie algebra ${\mathfrak g}$ the length $s({\mathfrak g})$ of its lower central series is maximal for a given dimension of Lie algebras: $s({\mathfrak g})=\dim{\mathfrak g}-1$. The class of so-called quasifiliform Lie algebras ($s({\mathfrak g})=\dim{\mathfrak g}-2$) has not greatly extended the supply of examples of Lie algebras close in properties to filiform \cite{Go,Ga}. In this paper we propose a generalization of the class of filiform Lie algebras from the point of view of the growth of Lie algebras. As we already mentioned above, filiform Lie algebras are the most slowly growing Lie algebras. We propose to study such Lie algebras that grow "slightly faster" than filiform Lie algebras, and are naturally graded.

Zelmanov and Shalev introduced the concept of narrow Lie algebras in \cite{ShZ1, ShZ2}, or in other words, the class of positively graded Lie algebras of small width $d$. A positively graded Lie algebra ${\mathfrak g}=\oplus_i {\mathfrak g}_i$ is called a Lie algebra of width $d$ if there is a $d$ (minimal with such a property) such that $\dim{{\mathfrak g }_i} \le d, \forall i \in{\mathbb N}$. The problem of classifying graded Lie algebras of finite width was outlined by Zel'manov and Shalev as an important and difficult problem (even "a formidable challenge" \cite{ShZ2}).
Naturally graded Lie algebras can not have width one. The minimum possible value for their width $d({\mathfrak g})$ is two. But also the classification of naturally graded Lie algebras of width two seems at the moment an immense task. One can distinguish among them a subclass of Lie algebras of "width $\frac {3}{2}$", that is, Lie algebras satisfying the following condition of narrowness: 
\begin{equation}
\label{width_3/2}
\dim{{\mathfrak g}_i}+\dim{{\mathfrak g}_{i{+}1}} \le 3, \; i \ge 1.
\end{equation}
The present paper is devoted to the classification of naturally graded Lie algebras of this class.

In conclusion, it should be noted that problems related to narrow and slowly growing Lie algebras were intensively studied in the case of a field of positive characteristic \cite{CMNS_96, CMN_97}. Petrogradsky constructed a continuum family of nil restricted Lie algebras of slow growth \cite{Petr}. Another example: let $ G $ be the Grigorchuk group and $G=G_1\supset G_2 \supset \dots \supset G_n \supset \dots$ its decreasing central series. The corresponding Lie algebra  (Lie ring) $L_{{\mathbb Z}_2}(G)$ над ${\mathbb Z}_2$ defined as
$$
L_{{\mathbb Z}_2}(G)=\oplus_{i=1}^{\infty}\left( G_i/G_{i+1}\right)\otimes_{\mathbb Z} {\mathbb Z}_2,
[a_iG_{i+1}, b_jG_{j+1}]=a_i^{-1}b_j^{-1}a_ib_jG_{i+j+1}.
$$ 
has width two \cite{Rozh} (except the very first component $G_1/G_{2}\otimes_{\mathbb Z} {\mathbb Z}_2$, which is three dimensional).

The paper is organized as follows. In the first section we give an overview of all the necessary facts concerning ${\mathbb N}$-graded Lie algebras of width one, i.e. positively graded Lie algebras ${\mathfrak g}=\oplus_{i=1}^{+\infty}{\mathfrak g}_i$ with one-dimensional homogeneous components $\dim{\mathfrak g}_i=1$.
The most interesting examples of such Lie algebras are the positive part $W^+$ of the Witt algebra (Virasoro) and the maximal nilpotent subalgebras ${\mathfrak n}_1, {\mathfrak n}_2 $ of the affine Kac-Moody algebras $A_1^{(1)}$ and $A_2^{(2)}$, respectively. The simplest example of a Lie algebra of width one is the Vergne algebra ${\mathfrak m}_0$  given by its infinite basis $e_1, e_2,\dots, e_n,\dots$  and the structure relations
$$
[e_1, e_i]=e_{i+1}, \; i \ge 2.
$$ 
The Lie algebra ${\mathfrak m}_0$ has a natural grading, and ${\mathfrak m}_0$ is the unique (up to isomorphism) infinite-dimensional naturally graded Lie algebra of maximal class \cite{V}.

All three Lie algebras ${\mathfrak m} _0 $, ${\mathfrak n}_1$ and ${\mathfrak n}_2 $ are naturally graded, but this is no longer true for the positive part $ W^+$ of the Witt algebra (Virasoro).
We show in the second section that in the real case two different real forms of the complex simple Lie algebra ${\mathfrak sl}(2,{\mathbb C})$ generate two nonisomorphic algebras of polynomial loops (currents) ${\mathfrak n} _1^{\pm}$. These two naturally graded Lie algebras will play an important role in the main classification theorem of the paper. 

In the section \ref{s2}, naturally graded Lie algebras and their growth are discussed. The paragraph \ref{s3} is devoted to cohomological computations. We discuss central extensions of naturally graded Lie algebras (Carnot algebras) in the paragraph \ref{s3}. We introduce the notion of Carnot expansion, a central extension of a special kind. It is with the help of Carnot extensions that all Carnot algebras are constructed inductively. The key lemma of the paragraph \ref{s3} is
{\bf The sentence \ref {main_propos}} that reduces the classification
Carnot extensions of the Carnot algebra ${\mathfrak g}$ to the description of the orbit space of the linear action of the group of graded automorphisms ${\rm Aut} _ {gr}({\mathfrak g})$ on the Grassmannians defined by the subspace of the second cohomology $H^2_{(n{+}1)}({\mathfrak g})$ with grading $n{+}1$.  

We study the groups of graded automorphisms of naturally graded Lie algebras (Carnot algebras) in the paragraph \ref{s4}. The paragraph \ref{s5} is devoted to the classification of Carnot extensions for several important series of Lie algebras. In the section \ref{s6} we classify Carnot extensions of finite-dimensional Lie factor-algebras obtained from ${\ mathfrak n}^{\pm}_1$ and $ {\mathfrak n} _2 $. All these three sections are preliminary preparations for the proof of two main theorems. These two main theorems of this article are:

{\bf Theorem \ref{osnovn}} - the classification of finite-dimensional naturally graded Lie algebras (Carnot algebras) of width $\frac{3}{2}$, i.e. satisfying the condition (\ref{width_3/2});
 
{\bf Theorem \ref{second_osnovn}}
{\it Let $\mathfrak {g}=\bigoplus_{i=1}^{+\infty}\mathfrak{g}_i$ --
an infinite-dimensional naturally graded Lie algebra over ${\mathbb K}={\mathbb R},{\mathbb C}$ satisfying the condition (\ref{width_3/2}).
Then $\mathfrak{g}$ is isomorphic to one and only one Lie algebra from the following list:}
$$
\mathfrak {n} _1^{\pm}, \mathfrak{n}_2, \mathfrak {n}_2^3, \mathfrak{m}_{0}, \{\mathfrak{m}_ {0}^{S}, S \subset \{3,5,7,\dots, 2m{+}1, \dots \}.
$$
where $S$ denotes a subset (possibly infinite) of the set $ \{3,5,7, \dots \}$ of odd positive integers of strictly large unity, and $\mathfrak{m}_{0}^{S}$ denotes the central extension of the algebra Vern $\mathfrak{m}_{0}$, defined by a set of two-dimensional cocycles whose gradings coincide with elements of $S$. The Lie algebra $\mathfrak {n}_2^3$ denotes the one-dimensional central extension of the Lie algebra $\mathfrak {n}_2$ defined by the cocycle of natural grading three in the second cohomology $H^2(\mathfrak {n}_2)$.

In Appendix 1 we present tables with bases and structure constants of all finite-dimensional naturally graded Lie algebras (Carnot algebras) of width $\frac{3}{2}$, i.e. naturally graded Lie algebras satisfying the property (\ref{width_3/2}).
Appendix 2 is devoted to quasifiliform Lie algebras. In particular, we compare
our more general results and notation with classification results on naturally graded quasifiliform Lie algebras from \cite{Go}.

\section{Infinite-dimensional ${\mathbb N}$-graded Lie algebras}\label{s1}

\begin{definition}
A Lie algebra ${\mathfrak g}$ is said to be ${\mathbb N}$-graded if there exists a direct sum decomposition of linear subspaces ${\mathfrak g}_i$ such that
$$
{\mathfrak g}{=}\bigoplus_{i \in{\mathbb N}}{\mathfrak g}_i, \; [{\mathfrak g}_i, {\mathfrak g}_j] \subset  {\mathfrak g}_{i{+}j}, \;  i,j \in  {\mathbb N}.
$$
\end{definition}

For the next three infinite-dimensional ${\mathbb N}$-graded Lie algebras ${\mathfrak g}$ all homogeneous subspaces ${\mathfrak g}_i$ are one-dimensional. In each homogeneous subspace ${\mathfrak g}_i, i \in {\mathbb N}$ the basic vector $e_i$ is fixed.

\begin{example}
The Lie algebra $\mathfrak{m}_0$ is given by commutation relations
$$ 
\label{m_0}
[e_1,e_i]=e_{i+1}, \; \forall \; i \ge 2.
$$
The remaining commutators of the basis elements vanish 
$
[e_i,e_j]=0, \; i,j \ne 1.
$
\end{example}
\begin{remark}
We shall henceforth omit relations of the form
$[e_i,e_j]=0$ in the definitions of Lie algebras.
\end{remark}

\begin{example}
The Lie algebra $\mathfrak{m}_2$ is given by relations
$$
[e_1, e_i ]=e_{i+1}, \quad \forall \; i \ge 2; \quad \quad
[e_2, e_j ]=e_{j+2}, \quad \forall \; j \ge 3.
$$
\end{example}

\begin{example}
The positive part $W^+$ of the Witt algebra can also be given by means of its infinite basis and relations
$$
[e_i,e_j]= (j-i)e_{i{+}j}, \; \forall \;i,j \in {\mathbb N}.
$$
\end{example}

All three infinite-dimensional Lie algebras $\mathfrak{m}_0, \mathfrak {m}_2, W^+$ are generated by two generators $e_1,e_2$ with gradings one and two, respectively.

A.~Fialowski presented in the \cite{Fial} classification
${\mathbb N}$-graded Lie algebras ${\mathfrak g}$ with one-dimensional homogeneous components ${\mathfrak g}_i$ that are multiplicatively generated by ${\mathfrak g}_1$ and ${\mathfrak g}_2$. To formulate Fialowski's theorem, we need to give two more examples of ${\mathbb N}$-graded Lie algebras.

\begin{example}
The Lie algebra ${\mathfrak n}_1$.
We consider an infinite collection of polynomial matrices of the following form. For each integer $k$, we define three matrices
$$
e_{3k{+}1}{=}\frac{1}{2}\begin{pmatrix} 
0 & t^{k}\\
0&0\\
\end{pmatrix}, e_{3k{-}1}{=}
\begin{pmatrix} 
0 & 0\\
t^{k}&0\\
\end{pmatrix}, 
e_{3k}{=}
\frac{1}{2}\begin{pmatrix} 
t^{k} & 0\\
0&{-}t^{k}\\
\end{pmatrix}. 
$$
The linear span of matrices with positive subscripts $\langle e_1, e_2, e_3, \dots, e_n, \dots \rangle$ Defines ${\mathbb N}$-graded Lie subalgebra ${\mathfrak n}_1$ in the algebra of polynomial loops
${\mathfrak sl}(2, {\mathbb K})\otimes {\mathbb K}[t]$. 
$$
{\mathfrak n}_1=\oplus_{i=1}^{+\infty}\langle e_i \rangle \subset {\mathfrak sl}(2, {\mathbb K})\otimes {\mathbb K}[t],
$$
where the following relations  hold
$$
[e_i,e_j]= c_{i,j}e_{i{+}j}, \; c_{i,j}=\left\{ \begin{array}{c} 1, {\rm if} \; j{-}i \equiv 1 \; {\rm mod} 3;\\
0, {\rm if} \; j{-}i \equiv 0  \; {\rm mod} 3;\\
{-}1, {\rm if} \; j{-}i \equiv{-}1 \; {\rm mod} 3.
\end{array}\right.
$$
\end{example}

\begin{example} 
The Lie algebra ${\mathfrak n}_2$.
We define an infinite set of polynomial $3\times 3$-matrices
(see also \cite{Kac}, Exercise 8.12) for non-negative integers
$k$
$$
\begin{array}{c}
f_{8k{+}1}{=}\begin{pmatrix} 
0 & t^{2k}&0\\
0&0 & t^{2k}\\
0&0&0\\
\end{pmatrix}, 
f_{8k{+}2}{=}\begin{pmatrix} 
0 & 0 &0\\
0&0 & 0 \\
t^{2k{+}1}&0&0\\
\end{pmatrix},
f_{8k{+}3}{=}\begin{pmatrix} 
0 & 0 &0\\
t^{2k{+}1}&0 & 0 \\
0&{-}t^{2k{+}1}&0\\
\end{pmatrix},\\
f_{8k{+}4}{=}\begin{pmatrix} 
t^{2k{+}1} & 0 &0\\
0&{-}2t^{2k{+}1} & 0 \\
0&0&t^{2k{+}1}\\
\end{pmatrix},
f_{8k{+}5}{=}\begin{pmatrix} 
0 & t^{2k{+}1} &0\\
0&0 & {-}t^{2k{+}1} \\
0& 0 &0\\
\end{pmatrix},\\
f_{8k{+}6}{=}\begin{pmatrix} 
0 & 0 &t^{2k{+}1}\\
0&0 &0 \\
0& 0 &0\\
\end{pmatrix},
f_{8k{+}7}{=}\begin{pmatrix} 
0 & 0 &0\\
t^{2k{+}2}&0 &0 \\
0& t^{2k{+}2} &0\\
\end{pmatrix},
f_{8k{+}8}{=}\begin{pmatrix} 
t^{2k{+}2} & 0 &0\\
0&0 &0 \\
0& 0 &{-}t^{2k{+}2}
\end{pmatrix},
\\
\end{array}
$$
It is easy to compute the pairwise commutators $[f_q, f_l]$ of these matrices, they are connected in the following way
$$
[f_q,f_l]= d_{q,l}f_{q{+}l}, \; q,l \in {\mathbb N}.
$$
Structur constants $d_{q, l}$ are presented in the Table  \ref{structure_const_n_2}.

\begin{table}
\caption{Structur constants of the Lie algebra ${\mathfrak n}_2$.}
\label{structure_const_n_2}
\begin{center}
\begin{tabular}{|c|c|c|c|c|c|c|c|c|}
\hline
&&&&&&&&\\[-10pt]
 &$f_{8j}$  &$f_{8j{+}1}$ &$f_{8j{+}2}$&$f_{8j{+}3}$&$f_{8j{+}4}$&$f_{8j{+}5}$&$f_{8j{+}6}$&$f_{8j{+}7}$\\
&&&&&&&&\\[-10pt]
\hline
&&&&&&&&\\[-10pt]
$f_{8i}$ & $0$ &$1$ & ${-}2$ & ${-}1$ &$0$ &$1$ &$2$ &${-}1$\\
&&&&&&&&\\[-10pt]
\hline
&&&&&&&&\\[-10pt]
$f_{8i{+}1}$ & ${-}1$ & $0$ & $1$ & $1$ & ${-}3$ & ${-}2$ & $0$ & $1$\\
&&&&&&&&\\[-10pt]
\hline
&&&&&&&&\\[-10pt]
$f_{8i{+}2}$ & $2$ & ${-}1$ &$0$ &$0$&$0$&$1$&${-}1$&$0$ \\
&&&&&&&&\\[-10pt]
\hline
&&&&&&&&\\[-10pt]
$f_{8i{+}3}$ & $1$ &${-}1$&$0$&$0$&$3$&${-}1$&$1$&${-}2$ \\
&&&&&&&&\\[-10pt]
\hline
&&&&&&&&\\[-10pt]
$f_{8i{+}4}$ & $0$ &$3$&$0$&${-}3$&$0$&$3$&$0$&${-}3$ \\
&&&&&&&&\\[-10pt]
\hline
&&&&&&&&\\[-10pt]
$f_{8i{+}5}$ & ${-}1$ &$2$&${-}1$&$1$&${-}3$&$0$&$0$&${-}1$ \\
&&&&&&&&\\[-10pt]
\hline
&&&&&&&&\\[-10pt]
$f_{8i{+}6}$ & ${-}2$ &$0$&$1$&${-}1$&$0$&$0$&$0$&$1$ \\
&&&&&&&&\\[-10pt]
\hline
&&&&&&&&\\[-10pt]
$f_{8i{+}7}$ & $1$ &${-}1$&$0$&$2$&$3$&$1$&${-}1$&$0$ \\[2pt]
\hline
\end{tabular}
\end{center}
\end{table}

The matrix $(d_{q, l})$ is skew-symmetric, and its elements $d_{q, l}$
depend only on the residues of the numbers $q$ and $l$ modulo $8$. In addition, the matrix elements $(d_{q, l})$ satisfy the following relations (see \cite{Kac}):
$$
d_{i,j}+d_{q,l}=0, \; {\rm if}\; i{+}q\equiv 0 \; {\rm mod} \; 8, j{+}l \equiv 0 \; {\rm mod} \; 8.
$$

Thus, the infinite-dimensional linear hull $\langle f_1, f_2, f_3, \dots \rangle$ defines a positively graded subalgebra ${\mathfrak n}_2$ in the algebra of polynomial loops ${\mathfrak sl}(3, {\mathbb K})\otimes {\mathbb K}[t]$:
$$
{\mathfrak n}_2=\oplus_{i=1}^{+\infty}\langle f_i \rangle \subset {\mathfrak sl}(3, {\mathbb K})\otimes {\mathbb K}[t],
$$
\end{example}

\begin{theorem}[A.~Fialowski, \cite{Fial}] 
Let ${\mathfrak g} =\oplus_{i{=}1}^{+\infty}{\mathfrak g} _i$ -- ${\mathbb N} $ be a graded Lie algebra with one-dimensional homogeneous components of ${\mathfrak g} _i, i \in {\mathbb N}$, and let the Lie algebra ${\mathfrak g} =\oplus_{i{=}1}^{+\infty}{\mathfrak g}_i$ be generated by  components  ${\mathfrak g}_1$ and ${\mathfrak g}_2$. 
We choose an infinite basis $e_1,e_2,e_3, \dots$ of the Lie algebra ${\mathfrak g}$ such that ${\mathfrak g}_i=\langle e_i \rangle$ for all $i \in {\mathbb N}$.

a) Suppose that  $[e_1,e_4]\neq0$ and $[e_2,e_3]\neq0$. If $[e_3, e_4] \neq 0$, then ${\mathfrak g} \cong W^+$, and if the equality $[e_3, e_4] = 0$ holds, then ${\mathfrak g} \cong {\mathfrak m} _2$.

b) Let  $[e_2,e_3]=0$. If $[e_3, e_4] \neq 0$, then ${\mathfrak g} \cong {\mathfrak n}_2$, while the relation $[e_3, e_4]=0$ implies an isomorphism  ${\mathfrak g} \cong {\mathfrak m}_0$.

c) Let  $[e_1,e_4]=0$. If $[e_3, e_4] \neq 0$, then  ${\mathfrak g} \cong {\mathfrak n}_1$, in the case $[e_3,e_4]=0$, then ${\mathfrak g}$ is isomorphic to a Lie algebra from a multiparameter family 
${\mathbb N}$-graded Lie algebras  ${\mathfrak g}(\lambda_8, \lambda_{12},{\dots},\lambda_{4k},\dots )$ with parameters $\lambda_{4k} \in {\mathbb P}{\mathbb K}^1, k \in {\mathbb N}$.
A Lie algebra  ${\mathfrak g}(\lambda_8, \lambda_{12},{\dots},\lambda_{4k},\dots )$ of this family is given by commutation relations
\begin{equation}
\begin{split}
[e_1, e_4]=0, [e_3,e_4]=0, [e_i,e_j]=0, \; i=2m \ge 4; \\
[e_1, e_{4k{-}1}]=\alpha_{4k}e_{4k}, \; [e_2, e_{4k{-}2}]=\beta_{4k}e_{4k}, \; 
(\alpha_{4k}:\beta_{4k})=\lambda_{4k} \in {\mathbb P}{\mathbb K}^1, \;  k\in {\mathbb N},
\end{split}
\end{equation}
where $\alpha_{4k}$ and $\beta_{4k}$ are homogeneous coordinates of the point $\lambda_{4k}$ of the projective line ${\mathbb K}{\mathbb P}^1$. The remaining commutators can be uniquely reconstructed from the above formulas. Structural constants are homogeneous polynomials in the parameters $\alpha_{4k}$ and $\beta_{4k}$. 
\end{theorem}
\begin{example}
The Lie algebra ${\mathfrak g} (1,1,\dots,1,\dots)$ (see \cite{Fial}) is determined by the following commutation relations (recall that we do not give a relation of the form $[e_i, e_j]=0$):
\begin{equation}
\begin{split}
[e_1,e_2]=e_3, \; [e_1,e_3]=e_4,\\ [e_1,e_{2k{+}1}]=[e_2,e_{2k}]=e_{2k{+}2},\; k \ge 2,\\
[e_2,e_{2k{-}1}]=e_{2k{+}1}, \; k \ge 2.
\end{split}
\end{equation}
\end{example}

Unfortunately, the work \cite{Fial} contains only a sketch of the proof, and the point c) of Fialowski's theorem is the most unclear in this sketch.

In 1992, Benoit gave a negative answer to Milnor's question \cite{Milnor}, who asked whether there always exists a left-invariant affine structure on a simply-connected nilpotent group. Benois presented an example of a compact $11$-dimensional nilmanifold that does not admit of any such complete affine structure. To substantiate the properties of his example, he constructed $11$-dimensional nilpotent Lie algebras that do not have exact linear representations of dimension $12$. In addition, Benoit \cite{Benoist} classified
${\mathbb N}$-graded Lie algebras ${\mathfrak a} _r$ given by the two generators $f_1$ and $f_2$ of degrees $1$ and $2$, respectively,
and two relations $[f_2, f_3] = f_5$ and $[f_2, f_5] = rf_7$, where $r$ denotes an arbitrary scalar, and $f_i, i \ge 3$, are given inductively
by the relation $f_ {i {+} 1} = [f_1, f_i]$.
\begin{lemma}[Benoist, \cite{Benoist}]
If $r {\ne} \frac{9}{10},1$, then ${\mathfrak a}_r$ is a finite-dimensional Lie algebra.

1) Let $r=\frac{9}{10}$, then ${\mathfrak a}_r \cong W^+$, 

2) Let $r=1$, then ${\mathfrak a}_r \cong {\mathfrak m}_2$.

3) Let $r \ne 0,\frac{9}{10},1,2,3$, then ${\mathfrak a}_r$ is $11$-dimensional filiform Lie algebra.
\end{lemma}
\begin{remark}
Later, in \cite{Mill2} it was shown that the case $r = 3$ leads to the $8$-dimensional filiform Lie algebra (see the definition in the next section), and the cases $r = 0,2$ lead
to two different $10$-dimensional filiform Lie algebras, respectively ($r=\frac{1+ \alpha}{2+\alpha}$ in the notation of \cite{Mill2}).
\end{remark}

A few years later Zelmanov and Shalev proved their theorem on the characterization of the Witt algebra \cite{ShZ1}, but their theorem was a consequence of the Fialowski classification \cite{Fial}.

\section{Naturally graded Lie algebras and the growth}\label{s2}

The sequence of ideals of the Lie algebra $\mathfrak{g}$
$$
\mathfrak{g}^1=\mathfrak{g} \; \supset \;
\mathfrak{g}^2=[\mathfrak{g},\mathfrak{g}] \; \supset \; \dots
\; \supset \;
\mathfrak{g}^k=[\mathfrak{g},\mathfrak{g}^{k-1}] \; \supset
\; \dots
$$
is called a decreasing (lower) central series of the Lie algebra $\mathfrak{g}$.

A Lie algebra $\mathfrak {g}$ is said to be nilpotent if
there exists a natural number $s$ such that
$$
\mathfrak{g}^{s+1}=[\mathfrak{g},\mathfrak{g}^s]=0,
\quad \mathfrak{g}^s \: \ne 0.
$$
The number $s$ is called the nil-index of a nilpotent Lie algebra
$\mathfrak {g}$, and the Lie algebra itself
$\mathfrak {g}$ is called a nilpotent Lie algebra of index $s$ or a nilpotent Lie algebra of degree $s$.

Let ${\mathfrak g}=\oplus_{i=1}^{+\infty}{\mathfrak g}_i$ be a finite-dimensional ${\mathbb N}$-graded Lie algebra, i.e.
there exists $N \in {\mathbb N}$ such that ${\mathfrak g}_i=0, i > N$. It is obvious that the Lie algebra ${\mathfrak g}$ is nilpotent.

\begin{example}
The Lie algebra $\mathfrak {m}_0(n)$, defined by its basis
$e_1, e_2, \dots, e_n$ with commutation relations
$$ 
[e_1, e_i] = e_ {i + 1}, \; \ forall \; 2 \ le i \ le n {-} 1, 
$$
is finite-dimensional ${\mathbb N}$-graded Lie algebra  ${\mathfrak g}=\oplus_{i=1}^{+\infty}{\mathfrak g}_i$,
where
$$
{\mathfrak g}_i=\langle e_i \rangle, i \le n, {\mathfrak g}_i=0, i > n.
$$
The Lie algebra
$\mathfrak{m}_0(n)$ is nilpotent with nil-index $s(\mathfrak{m}_0(n))=n{-}1$. 
\end{example}
\begin{propos}[\cite{V}]
Let $\mathfrak {g}$ be an $n$-dimensional nilpotent Lie algebra.
Then for its nil-index $s(\mathfrak {g})$ the estimate
 $s(\mathfrak{g}) \le n-1$.
\end{propos}
\begin{definition}
A nilpotent $n$-dimensional Lie algebra $\mathfrak{g}$ is said to be filiform if its nil-index is one less than the dimension $s=n-1$. 
\end{definition}
\begin{remark}
In the theory of $p$-groups, the nil-index $s({\mathfrak g})$ of the Lie algebra ${\mathfrak g} $ is called the class of the Lie algebra
${\mathfrak g}$. In this same terminology, filiform Lie algebras are called Lie algebras of maximal class \cite{ShZ1, ShZ2,CMN_97}.
\end{remark}

The Lie algebra ${\mathfrak m}_0 (n)$, considered above, is filiform.

\begin{definition}
A Lie algebra ${\mathfrak g}$ is said to be pro-nilpotent if for
of the ideals ${\mathfrak g}^k$ of its lower central series is satisfied
$$
\cap_{i{=}1}^{\infty}\mathfrak{g}^i=\{0\}, \; \dim{{\mathfrak g}/{\mathfrak g}^{i}} <+\infty,\; \forall i \in {\mathbb N}.
$$
\end{definition}
In this case all its Lie quotients of the form ${\mathfrak g} / {\mathfrak g}^i$ are finite-dimensional nilpotent Lie algebras. One can consider the inverse spectrum of finite-dimensional nilpotent Lie algebras
$$
 \dots \stackrel{p_{k{+}2,k{+}1}}{\longrightarrow}{\mathfrak g}/{\mathfrak g}^{k{+}1} \stackrel{p_{k{+}1,k}}{\longrightarrow}{\mathfrak g}/
{\mathfrak g}^k \stackrel{p_{k,k{-}1}}{\longrightarrow}
\dots \stackrel{p_{3,2}}{\longrightarrow} {\mathfrak g}/{\mathfrak g}^2 \stackrel{p_{2,1}}{\longrightarrow}{\mathfrak g}/{\mathfrak g}^1,
$$
Denote by $\widehat{\mathfrak g}$ its projective (inverse) limit $\widehat{\mathfrak g}=\varprojlim \limits_{k} {\mathfrak g}/
{\mathfrak g}^k$. 

We shall call the pro-nilpotent Lie algebra ${\mathfrak g}$ complete if ${\mathfrak g}\cong \widehat{\mathfrak g}$. 

We have a set of projections of the pro-nilpotent Lie algebra ${\mathfrak g}$ onto its finite-dimensional quotients
$$
p_m: {\mathfrak g} \to  {\mathfrak g}/{\mathfrak g}^m, \; m \in {\mathbb N}.
$$
These projections define the topology of the inverse limit of finite-dimensional spaces on our Lie algebra ${\mathfrak g}$, i.e. the weakest topology in which all maps $p_m$ are continuous.

${\mathbb N}$-graded Lie algebras $\mathfrak {m}_0, \mathfrak{m}_2, W^+,$ considered above are not complete. Their completions $\hat {\mathfrak m}_0, \hat {\mathfrak m}_2, \hat W^+,$ are ${\mathbb N}$-graded Lie algebras of the type ${\mathfrak g}=\prod_{i=1}^{+\infty}{\mathfrak g}_i$, spaces of formal series $\sum_{i=1}^{+\infty}\alpha_i e_i$. The completions $\hat{\mathfrak m} _0, \hat {\mathfrak m}_2, \hat W^+$ are the inverse limits of the corresponding finite-dimensional ${\mathbb N}$-graded Lie algebras ${\mathfrak m}_0(n), {\mathfrak m}_2(n), W^+(n)$.

The ideals $\mathfrak{g}^k, k \ge 1,$ of the lower central series of the pro-nilpotent Lie algebra ${\mathfrak g}$ determine a decreasing filtration
$$
\mathfrak{g}=\mathfrak{g}^1 \supset 
\mathfrak{g}^2 \supset \dots
\supset  \mathfrak{g}^k \supset \mathfrak{g}^{k+1} \supset
\dots , \;\; [{\mathfrak g}^i, {\mathfrak g}^j] \subset {\mathfrak g}^{i+j}, i,j \in {\mathbb N}.
$$
Pro-nilpotency is important here in terms of numbering the filter subspaces, in this case we start it from unity. In the case of a solvable Lie algebra, we would have to start with a zero index. 

Consider the associated graded Lie algebra
${\rm gr} \mathfrak{g}=\oplus_{i=1}^{+\infty}  \left(\mathfrak{g}^i / \mathfrak{g}^{i{+}1}\right)$ with respect to this filtration. Its Lie bracket is given by
$$
 \left[x{+}\mathfrak{g}^{i+1}, y{+}\mathfrak{g}^{j+1}\right]=[x,y]{+}\mathfrak{g}^{i{+}j{+}1}, x\in \mathfrak{g}^i , y\in \mathfrak{g}^j.
$$
In view of the foregoing ${\rm gr} \mathfrak{g}$ will be a ${\mathbb N}$-graded Lie algebra.
\begin{definition}
A Lie algebra $\mathfrak {g}$ is said to be naturally graded if it is isomorphic to its associated graded Lie algebra
${\rm gr} \mathfrak{g}$.
The grading $\mathfrak {g}=\oplus_{i=1}^{+\infty}{\mathfrak g}_i$ of a naturally graded Lie algebra $\mathfrak{g}$ is called a natural gradation if there exists a graded isomorphism
$$
\varphi: {\rm gr}  \mathfrak{g} \to \mathfrak{g}, \; \varphi(({\rm gr} \mathfrak{g})_i)={\mathfrak g}_i, \; 
i \in {\mathbb N}.
$$
\end{definition}

In what follows, by a naturally graded Lie algebra ${\mathfrak g}=\oplus_{i=1}^{+\infty} {\mathfrak g}_i$ we mean a naturally graded Lie algebra equipped with a natural grading.

The class of pro-nilpotent Lie algebras essentially extends the world of nilpotent Lie algebras. It is easy to see that the free Lie algebra ${\mathcal L}(k)$ generated by $ k $ by the generators $a_1, \dots, a_k$ is also a pro-nilpotent Lie algebra. Its quotient Lie algebra ${\mathcal L} (k, m)={\mathcal L}(k)/{\mathcal L}^{m + 1}(k)$ is a finite-dimensional nilpotent Lie algebra, sometimes called in the literature a free nilpotent Lie algebra of degree nilpotency $m$. 

The Lie algebra $\mathfrak {m}_0(n)$, considered above, is naturally graded. However, its natural grading differs in properties from its grading, which we examined at the very beginning.
$$
\left( {\rm gr} m_0(n)\right)_1=\langle e_1, e_2\rangle, 
\left( {\rm gr} m_0(n)\right)_i=\langle e_{i{+}1}\rangle, i=2,\dots,n{-}1.
$$
\begin{propos}
Let $\mathfrak{g}$ be a filiform Lie algebra and
${\rm gr}\mathfrak{g}=\oplus_i ({\rm gr} \mathfrak {g})_i$ is the corresponding associated (relative to filtering ideals of the lower central series)
graded Lie algebra. Then
$$
\dim ({\rm gr} \mathfrak{g})_1=2, \quad
\dim ({\rm gr} \mathfrak{g})_2=\dots=
\dim ({\rm gr} \mathfrak{g})_{n{-}1}=1.
$$
\end{propos}
\begin{theorem}[Vergne \cite{V}]
\label{V_ne1}
Let $\mathfrak{g}=\oplus_{i=1}^{n{-}1}\mathfrak{g}_{i}$ be a naturally graded
$n$-dimensional filiform Lie algebra. Then  

1) if $n=2k+1$, then $\mathfrak{g} \cong \mathfrak{m}_0(2k+1)$; 

2) if $n=2k$, then $\mathfrak{g}$ is isomorphic either to
$\mathfrak{m}_0(2k)$, or to the Lie algebra  $\mathfrak{m}_1(2k)$, 
given by its basis
$e_1, \dots, e_{2k}$ 
and commutation relations
$$
[e_1, e_i ]=e_{i+1}, \; i=2, \dots, 2k{-}1; \quad \quad
[e_j, e_{2k{+}1{-}j} ]=(-1)^{j{+}1}e_{2k}, \quad j=2, \dots, k.
$$
\end{theorem}

An infinite-dimensional ${\mathbb N}$-graded Lie algebra $\mathfrak{m}_0$ is also naturally graded,
but the positive part $W^+$ of the Witt algebra and the Lie algebra $\mathfrak{m}_2$ are not naturally graded.
It is easy to verify the existence of the following isomorphisms
$$ {\rm gr} \mathfrak{m}_2 \cong
{\rm gr} W^+ \cong {\rm gr} \mathfrak{m}_0 \cong
\mathfrak{m}_0.$$

From the properties of a descending (lower) central series one can derive one very important property of natural grading  ${\mathfrak g}=\oplus_{i=1}^{+\infty}{\mathfrak g}_i$:
$$
[{\mathfrak g}_1,{\mathfrak g}_i]={\mathfrak g}_{i{+}1}, i \in {\mathbb N}.
$$
In particular, a naturally graded Lie algebra ${\mathfrak g}=\oplus_{i=1}^{+\infty}{\mathfrak g}_i$ is generated
its first homogeneous component ${\mathfrak g}_1$.

In sub-Riemannian geometry, control theory, and geometric group theory, finite-dimensional naturally graded Lie algebras and the corresponding nilpotent Lie groups play a significant role and are well known as Carnot algebras (Carnot groups) \cite{Gro96, Mon02}.

\begin{definition}
A finite-dimensional Lie algebra ${\mathfrak g}$ is called a Carnot algebra if it has ${\mathbb N}$-grading
${\mathfrak g}=\oplus_{i=1}^n{\mathfrak g}_i$ such that
\begin{equation}
\label{carnot}
[{\mathfrak g}_1,{\mathfrak g}_i]={\mathfrak g}_{i{+}1}, i=1,2,\dots,n-1, \; [{\mathfrak g}_1,{\mathfrak g}_n]=0.
\end{equation}
\end{definition}
\begin{remark}
In some sources, grading with properties (\ref{carnot}) is called the stratification of the Lie algebra ${\mathfrak g}$.
\end{remark}
For ideals of the lower central series of the Carnot algebra ${\mathfrak g}$,
$$
{\mathfrak g}^k=\oplus_{i=k}^n{\mathfrak g}_i, \; k=1,\dots, n.
$$
This implies the uniqueness in the natural sense of Carnot's grading (stratification).

For an arbitrary automorphism $f$ of the Lie algebra ${\mathfrak g}$, we have
$$
f({\mathfrak g}^i)={\mathfrak g}^i, \; \forall i \in {\mathbb N}.
$$
Thus, for an arbitrary automorphism $f$ of a naturally graded Lie algebra (Carnot algebra) ${\mathfrak g}=\oplus_{i=1}^{+\infty}{\mathfrak g}_i$
the following property holds
$$
f(\oplus_{i=m}^{+\infty}{\mathfrak g}_i)=\oplus_{i=m}^{+\infty}{\mathfrak g}_i, \; \forall m \in {\mathbb N}.
$$
This means that, generally speaking, only the filtration $\{\oplus_{i = m}^{+\infty} {\mathfrak g}_i, m \in{\mathbb N} \}$, constructed from the natural grading ( Carnot graduation) is invariant under an arbitrary automorphism $f$.
\begin{definition}
An automorphism $f$ of a naturally graded Lie algebra (Carnot algebra) ${\mathfrak g}=\oplus_{i=1}^{+\infty}{\mathfrak g}_i$ is called {\it graded} if it is compatible with natural graduation
$$
f({\mathfrak g}_i)={\mathfrak g}_i, \; \forall i \in {\mathbb N}.
$$
We denote $Aut_{gr}({\mathfrak g})$ by the subgroup of graded automorphisms in the group
$Aut({\mathfrak g})$ of all automorphisms of the Lie algebra${\mathfrak g}$.
\end{definition}

Let $f$ be an automorphism of a naturally graded Lie algebra (the Carnot algebra) ${\mathfrak g}=\oplus_{i=1}^{+\infty}{\mathfrak g}_i$. One can define the corresponding associated graded automorphism $f_{gr}$:
$$
f_{gr}: \oplus_{i=m}^{+\infty}{\mathfrak g}_i/\oplus_{i=m+1}^{+\infty}{\mathfrak g}_i \to \oplus_{i=m}^{+\infty}{\mathfrak g}_i/\oplus_{i=m+1}^{+\infty}{\mathfrak g}_i.
$$
\begin{example}
Let $f$ be an automorphism of a naturally graded Lie algebra ${\mathfrak m}_0(n)$. It suffices to specify it on the generators $e_1, e_2$. 
$$
f(e_1)=\alpha_1e_1+\alpha_2e_2+\alpha_3e_3+\dots+\alpha_ne_n, \; f(e_2)=\beta_2e_2+\beta_3e_3+\dots+\beta_ne_n, 
$$
with some constants $\alpha_i, i=1,\dots, n, \; \beta_j, j=2,\dots, n$.  For the associated graded automorphism
$f_{gr}: {\mathfrak m}_0(n) \to {\mathfrak m}_0(n)$ we have
$$
f_{gr}(e_1)=\alpha_1e_1+\alpha_2e_2, \; f_{gr}(e_2)=\beta_1e_1+\beta_2e_2, \; \det{\begin{pmatrix}\alpha_1 & \beta_1\\ \alpha_2 & \beta_2 \end{pmatrix}} \ne 0.
$$
Continuing this inductive procedure, we obtain the general formula
 $$
f_{gr}(e_i)=\alpha_1^{i{-}2}\beta_2e_i, \; i\ge 2, i \in {\mathbb N}.
$$
\end{example}

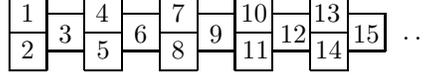
\begin{figure}[tb]
\begin{picture}(100,20)(-35,0)
  \multiput(0,0)(5,0){2}{\line(0,1){10}}
  \multiput(0,0)(0,5){3}{\line(1,0){5}}
  \multiput(10,2.5)(5,0){1}{\line(0,1){5}}
  \multiput(5,2.5)(0,5){2}{\line(1,0){5}}
 \multiput(10,0)(5,0){2}{\line(0,1){10}}
  \multiput(10,0)(0,5){3}{\line(1,0){5}}
\multiput(10,2.5)(5,0){1}{\line(0,1){5}}
  \multiput(15,2.5)(0,5){2}{\line(1,0){5}}
\multiput(20,0)(5,0){2}{\line(0,1){10}}
  \multiput(20,0)(0,5){3}{\line(1,0){5}}
 \multiput(30,2.5)(5,0){2}{\line(0,1){5}}
  \multiput(25,2.5)(0,5){2}{\line(1,0){5}}
  \multiput(30,0)(5,0){2}{\line(0,1){10}}
  \multiput(30,0)(0,5){3}{\line(1,0){5}}
 \multiput(35,2.5)(5,0){2}{\line(0,1){5}}
  \multiput(35,2,5)(0,5){2}{\line(1,0){5}}
\multiput(40,0)(5,0){2}{\line(0,1){10}}
  \multiput(40,0)(0,5){3}{\line(1,0){5}}
\multiput(45,2.5)(5,0){2}{\line(0,1){5}}
  \multiput(45,2,5)(0,5){2}{\line(1,0){5}}
\put(52.3, 4){$\dots$}
%\put(60.3, 4){$\dots$}
\put(1.6,1.5){$2$}
 \put(1.6, 6.4){$1$}
 \put(6.6, 3.7){$3$}
 \put(16.6, 3.7){$6$}
\put(11.6,1.5){$5$}
  \put(11.5, 6.4){$4$}
\put(21.6,1.5){$8$}
  \put(21.6, 6.4){$7$}
\put(26.6, 3.7){$9$}
 \put(31,1.5){$11$}
  \put(30.8, 6.4){$10$}
 \put(36, 3.7){$12$}
 \put(45.7, 3.7){$15$}
 \put(40.7,1.5){$14$}
  \put(40.5, 6.4){$13$}
%\put(2, 16.5){$\bullet$}
%\put(2, 25.5){$\bullet$}
%\put(11.5, 16.5){$\bullet$}
%\put(11.5, 25.5){$\bullet$}
%\put(2.5, 17){\line(1,1){10}}
%\put(2.5, 27){\line(1,-1){10}}
%\put(6.5, 21){$\bullet$}

%\put(21.5, 16.5){$\bullet$}
%\put(21.5, 25.5){$\bullet$}
%\put(12.5, 17){\line(1,1){10}}
%\put(12.5, 27){\line(1,-1){10}}
%\put(16.5, 21){$\bullet$}

%\put(31.5, 16.5){$\bullet$}
%\put(31.5, 25.5){$\bullet$}
%\put(22.5, 17){\line(1,1){10}}
%\put(22.5, 27){\line(1,-1){10}}
%\put(26.5, 21){$\bullet$}

%\put(41.5, 16.5){$\bullet$}
%\put(41.5, 25.5){$\bullet$}
%\put(32.5, 17){\line(1,1){10}}
%\put(32.5, 27){\line(1,-1){10}}
%\put(36.5, 21){$\bullet$}

%\put(42.5, 17){\line(1,1){5}}
%\put(42.5, 27){\line(1,-1){5}}
%\put(46.5, 21){$\bullet$}

%\put(52.3, 21){$\dots$}
%\put(60.3, 21){$\dots$}
\end{picture}
\caption{The natural grading of  ${\mathfrak n}_1$}
\label{firstfig}
\end{figure}

\begin{propos}
The Lie algebras ${\mathfrak n}_1, {\mathfrak n}_2$ and finite-dimensional quotient Lie algebras ${\mathfrak n}_1/({\mathfrak n}_1)^k$ and ${\mathfrak n}_2/({\mathfrak n}_2)^k$, over $k \in {\mathbb N}$,
are naturally graded Lie algebras.
\end{propos}
\begin{proof}
For the proof we enumerate the vectors of the bases for both Lie algebras. New bases will be consistent with the natural graduation.

For the Lie algebra ${\mathfrak n}_1$ we define new basis vectors $a_{2k{+}1}, b_{2k{+}1}, c_{2k}$ 
$$
a_{2k{+}1}=e_{3k{+}1}, \; b_{2k{+}1}=e_{3k{+}2}, \; c_{2k}=e_{3k}, \; \forall k \in {\mathbb Z}_{\ge 0}.
$$
In this basis, the commutation relations are rewritten as follows
\begin{equation}
\label{n_1^-}
[a_{2k{+}1}, b_{2l{+}1}]=c_{2(k{+}l{+}1)},\; [c_{2k}, a_{2l{+}1}]=a_{2(k{+}l){+}1}, \;
[c_{2k}, b_{2l{+}1}]={-}b_{2(k{+}l){+}1},
\end{equation}
It is easy to consistently find the ideals of the lower central series
$$
{\mathfrak n}_1^{2m{+}1}=\langle a_{2m{+}1}, b_{2m{+}1}, c_{2m{+}2},\dots,\rangle, \;{\mathfrak n}_1^{2m}=\langle c_{2m}, a_{2m{+}1}, b_{2m{+}1},\dots,\rangle.
$$
So, the natural grading is defined as follows
$$
{\mathfrak n}_1=\oplus_{i=0}^{+\infty}{\mathfrak n}_{1,i}, \; {\mathfrak n}_{1,2m{+}1}=
 \langle a_{2m{+}1}, b_{2m{+}1} \rangle, {\mathfrak n}_{1,2m}=\langle c_{2m} \rangle
$$
i.e. in this grading homogeneous components with even superscripts are one-dimensional, and components with odd superscripts are two-dimensional.

For the second Lie algebra ${\mathfrak n}_2$ we define new basis vectors $a_i, b_{6q{+}1}, b_{6q{+}5}$
$$
\begin{array}{c}
a_{6q{+}1}=e_{8k{+}1},\\
a_{6q{+}2}=e_{8k{+}3},\\
\end{array}
\begin{array}{c}
a_{6q{+}3}=e_{8k{+}4},\\ 
a_{6q{+}4}=e_{8k{+}5},\\
\end{array}
\begin{array}{c}
a_{6q{+}5}=e_{8k{+}6},\\ 
a_{6q{+}6}=e_{8k{+}8},\\
\end{array}
\begin{array}{c}
 b_{6q{+}1}=e_{8k{+}2}, \\
b_{6q{+}5}=e_{8k{+}7}, 
\end{array}
$$
Then the natural grading of the Lie algebra ${\mathfrak n}_2$ is written as
$$
{\mathfrak n}_2=\oplus_{i{=}0}^{+\infty} {\mathfrak n}_{2,i}, \;{\mathfrak n}_{2,i}=
 \langle a_{i}, b_{i}\rangle, \;  i=6q{+}2, 6q{+}5, \; {\mathfrak n}_{2,i}=\langle a_{i} \rangle,
\;   i \ne 6q{+}2, 6q{+}5. 
$$
\end{proof}
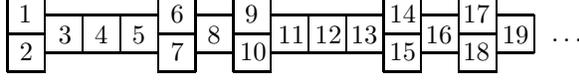
\begin{figure}[tb]
\begin{picture}(112,25)(-30,0)
  \multiput(0,0)(5,0){2}{\line(0,1){10}}
  \multiput(0,0)(0,5){3}{\line(1,0){5}}
  \multiput(10,2.5)(5,0){2}{\line(0,1){5}}
  \multiput(5,2.5)(0,5){2}{\line(1,0){15}}
 \multiput(20,0)(5,0){2}{\line(0,1){10}}
  \multiput(20,0)(0,5){3}{\line(1,0){5}}
\multiput(30,2.5)(5,0){1}{\line(0,1){5}}
  \multiput(25,2.5)(0,5){2}{\line(1,0){5}}
\multiput(30,0)(5,0){2}{\line(0,1){10}}
  \multiput(30,0)(0,5){3}{\line(1,0){5}}
 \multiput(40,2.5)(5,0){2}{\line(0,1){5}}
  \multiput(35,2.5)(0,5){2}{\line(1,0){15}}
  \multiput(50,0)(5,0){2}{\line(0,1){10}}
  \multiput(50,0)(0,5){3}{\line(1,0){5}}
 \multiput(55,2.5)(5,0){2}{\line(0,1){5}}
  \multiput(55,2,5)(0,5){2}{\line(1,0){5}}
\multiput(60,0)(5,0){2}{\line(0,1){10}}
  \multiput(60,0)(0,5){3}{\line(1,0){5}}
 \multiput(70,2.5)(5,0){1}{\line(0,1){5}}
  \multiput(65,2.5)(0,5){2}{\line(1,0){5}}
\put(72.3, 4){$\dots$}
%\put(80.3, 4){$\dots$}
\put(1.6,1.5){$2$}
  \put(1.5, 6.5){$1$}
 \put(6.8, 3.7){$3$}
 \put(11.6, 3.7){$4$}
 \put(16.6, 3.7){$5$}
\put(21.7,1.5){$7$}
  \put(21.7, 6.5){$6$}
\put(26.6, 3.7){$8$}
 \put(30.9,1.5){$10$}
  \put(31.6, 6.5){$9$}
 \put(36, 3.7){$11$}
 \put(40.9, 3.7){$12$}
 \put(45.7, 3.7){$13$}
 \put(50.8,1.5){$15$}
  \put(50.8, 6.5){$14$}
 \put(55.6, 3.7){$16$}
\put(60.6,1.5){$18$}
  \put(60.6, 6.5){$17$}
\put(65.8, 3.7){$19$}
\end{picture}
\caption{Natural grading of ${\mathfrak n}_2$.}
\label{firstfig}
\end{figure}

We define the set of matrices for each natural number $k$
$$
u_{2k{-}1}=\begin{pmatrix} 0 & t^{2k{-}1} &0\\
{-}t^{2k{-}1} & 0 &0\\
0&0&0
\end{pmatrix},
v_{2k{-}1}^{\pm}=\begin{pmatrix} 0 & 0&0\\
 0 & 0& t^{2k{-}1}\\
0& {\mp}t^{2k{-}1} & 0 \\
\end{pmatrix}, 
w_{2k}^{\pm}=\begin{pmatrix} 0 & 0& t^{2k}\\
 0 & 0&\\
 {\mp}t^{2k}& 0 & 0 \\
\end{pmatrix}.
$$ 
One can easy verify the following commutation relations
$$
[u_{2k{-}1}, v_{2l{-}1}^{\pm}]= w_{2(k{+}l){-}2}^{\pm}, \;
[v_{2k{-}1}^{\pm}, w_{2l}^{\pm}]
=\pm u_{2(k{+}l){-}1}, \; 
[w_{2k}^{\pm},u_{2l{+}1}]
= v_{2(k{+}l){-}1}^{\pm}.
$$

The linear span ${\mathfrak n}_1^+=\langle u_1, v_1^+, w_2^+, \dots, u_{2k {+} 1}, v_{2k {+} 1}^+, w_{2k {+} 2}^+, \dots \rangle$ defines a naturally graded Lie subalgebra in ${\mathfrak so}(3, {\mathbb K})\otimes {\mathbb K}[t]$. In its turn, other linear span
${\mathfrak n}_1^-{=}\langle u_1, v_1^-, w_2^-, \dots, u_{2k{+}1}, v_{2k{+}1}^-, w_{2k{+}2}^-, \dots \rangle$ is a naturally graded Lie subalgebra in the Lie algebra of polynomial loops ${\mathfrak so}(2,1)\otimes {\mathbb K}[t]$.             
\begin{lemma}
The Lie algebra ${\mathfrak n}_1^{\pm}$ and all finite-dimensional quotients  ${\mathfrak n}_1^{\pm}/({\mathfrak n}_1^{\pm})^k, k \ge 5,$ are isomorphic over
${\mathbb C}$ and non isomorphic over ${\mathbb R}$.
\end{lemma}
\begin{proof}
We make a coordinate change in the quotient Lie algebra
${\mathfrak n}_1^{-}/({\mathfrak n}_1^{-})^5$ 
$$
e_1=u_1, e_2=v_1^-, e_3=w_2^-, e_4={-}v_3^-, e_5={-}u_3^-, e_6={-}w_4^-.
$$ 
The commutation relations of the Lie quotient algebra ${\mathfrak n}_1^{-}/({\mathfrak n}_1^{-})^5$ will have the form
$$
[e_1, e_2]=e_3,\; [e_1,e_3]=e_4,\; [e_2, e_3]=e_5,\; [e_1,e_4]={-}[e_2,e_5]=e_6.
$$

Consider a basis change in the Lie algebra ${\mathfrak n}_1^{+}/({\mathfrak n}_1^{+})^5$.
$$
e_1=u_1, e_2=v_1^+, e_3=w_2^+, e_4={-}v_3^+, e_5=u_3, e_6={-}w_4^+.
$$ 
The commutation relations of the Lie algebra  ${\mathfrak n}_1^{+}/({\mathfrak n}_1^{+})^5$ will look like this
$$
[e_1, e_2]=e_3,  \; [e_1, e_3]=e_4, \; [e_2, e_3]=e_5, \; [e_1, e_4]=[e_2 ,e_5]=e_6.
$$
Thus, we see that the structure constants of the Lie algebras ${\mathfrak n}_1^{\pm}/({\mathfrak n}_1^{\pm})^5$ differ only by the sign before $e_6$ in the relation $[e_2, e_5] = \pm e_6$.
This circumstance explains the choice of notation for these Lie algebras.

On the other hand, according to Morozov's classification of \cite{Mor} $6$-dimensional nilpotent Lie algebras, the last two Lie algebras are isomorphic over the complex field ${\mathbb C}$, but not isomorphic over ${\mathbb R}$.
The latter circumstance is not surprising, if we recall that
The Lie algebras ${\mathfrak so}(3, {\mathbb R})$ and ${\mathfrak sl}(2,{\mathbb R})$ are distinct real forms of the complex Lie algebra ${\mathfrak sl}(2,{\mathbb C})$.
\begin{remark}
$
{\mathfrak n}_1^{\pm}/({\mathfrak n}_1^{\pm}),3 \cong {\mathfrak m}_0(3), {\mathfrak n}_1^{\pm}/({\mathfrak n}_1^{\pm}),4\cong {\mathcal L}(2,3),
$
where  ${\mathcal L}(2,3)$ denotes the free three-step nilpotent Lie algebra generated by two generators. 
\end{remark}
\end{proof}

\begin{definition}
\label{n_1_quotients}
We denote finite-dimensional Lie quotient-algebras  ${\mathfrak n}_1$ for $m \ge 0$ by
$$
\begin{array}{c}
{\mathfrak n}_1(n){=}{\mathfrak n}_1/({\mathfrak n}_1)^{n{+}1}, \;\;
{\mathfrak n}_{1,1}(2m{+}1){=}{\mathfrak n}_1/K_{2m{+}1},
\end{array}
$$
where the ideal $K_{2m {+}1}$ is given as a linear span
$\langle u_{2m{+}1}, w_{2m{+}2}, u_{2m{+}3}, v_{2m{+}3},\dots, \rangle$.
\end{definition}
\begin{propos}
Lie quotient-algebras
${\mathfrak n}_1(n)$, ${\mathfrak n}_{1,1}(2m{+}1)$, given above
are finite-dimensional Carnot algebras with nil-indices
$n,2m{+}1$.
\end{propos}

The Lie algebras ${\mathfrak n}_1$ and ${\mathfrak n}_2$ are maximal nilpotent Lie subalgebras of affine Kac-Moody algebras $A_1^{(1)}$ and $A_2^{(2)}$ \cite{Kac}. 

Let $A$ be a generalized $n\times n$ Cartan matrix  and ${\mathfrak g}(A)$ denotes the corresponding Kac-Moody algebra (see \cite{Kac}). The Kac-Moody algebra ${\mathfrak g}(A)$ has the maximal nilpotent Lie subalgebra ${\mathfrak n}(A) \subset {\mathfrak g}(A)$, which, in turn, can be given by its generators $e_1, e_2, \dots, e_n$ and a set of defining relations
$$
ad e_i^{{-}a_{ij}{+}1}(e_j)=0, 1 \le i \ne j \le n,
$$ 
where $a_{ij}$ denote the elements of the Cartan matrix $A$.

The Lie algebra ${\mathfrak n}(A)$ is ${\mathbb N}\oplus \dots \oplus {\mathbb N}$-graded
\begin{equation}
\label{canon_grad}
{\mathfrak n}(A)=\oplus_{k_1{>}0,k_2{>}0,{\dots},k_n{>}0}^{{+}\infty} {\mathfrak n}(A)_{(k_1,k_2,\dots,k_n)},
\end{equation}
where a homogeneous subspace ${\mathfrak n}(A)_{(k_1,k_2,\dots,k_n)}$ is a linear span of commutators  that include exactly
$k_i$ of generators $e_i$, $i=1,\dots,n$.
\begin{propos}
The Lie algebra ${\mathfrak n}(A)$ is naturally graded. Two gradings, natural and canonical (\ref{canon_grad}), are obviously connected with each other by the rule
$$
{\mathfrak n}(A)=\bigoplus_{N{=}1}^{{+}\infty}{\mathfrak n}(A)_{(N)}, \; {\mathfrak n}(A)_{(N)}=\bigoplus_{k_1{+}{\dots}{+}k_n=N} {\mathfrak n}(A)_{(k_1,\dots,k_n)}
$$
\end{propos}The Lie algebra $A_1^{(1)}$ corresponds to the generalized Cartan matrix $A=\begin{pmatrix} 2 & {-}2\\ {-}2 & 2\end{pmatrix}$.
Thus, the Lie algebra ${\mathfrak n}_1=N(A_1^{(1)})$ is generated by two elements $e_1, e_2$ that satisfy two relations
$$
ad^3 e_2 (e_1)=\left[e_2,[e_2,[e_2, e_1] ]\right]{=}0, \; ad^{3} e_1(e_2){=}\left[e_1,[e_1,[e_1,e_2]] \right] =0. 
$$

In turn, the Lie algebra $A_2^{(2)}$ corresponds to another generalized Cartan matrix $\begin{pmatrix} 2 & {-}4\\ {-}1 & 2\end{pmatrix}$. Its maximal nilpotent subalgebra ${\mathfrak n}_2{=}N(A_2^{(2)})$ is generated by $e_1, e_2$,  also satisfying two relations
$$
ad^2 e_2 (e_1)=\left[e_2, [e_2, e_1] \right]{=}0, \; ad^{5} e_1(e_2){=}\left[e_1,[e_1,[e_1,[e_1,[e_1,e_2]]]] \right] =0. 
$$
Both Lie algebras Ли $A_1^{(1)}$ and $A_2^{(2)}$ are canonically ${\mathbb Z}\oplus {\mathbb Z}$-graded as affine Kac-Moody algebras, their maximal nilpotent Lie subalgebras ${\mathfrak n}_1$ and  ${\mathfrak n}_2$ are ${\mathbb N}\oplus {\mathbb N}$-graded
$$
{\mathfrak n}_i=\oplus_{p,q{=}1}^{{+}\infty} ({\mathfrak n}_i)_{(p,q)}, i=1,2,
$$
where the subspace $({\mathfrak n}_i)_{p,q}, i{=}1,2,$ denotes the linear span of all commutators that contain exactly
$p$ of generators $e_1$ and $q$ of generators $e_2$. The generators $e_1, e_2$ have gradings $(1,0)$ and $(0,1)$ respectively. 

One can verify that the canonical grading of the basis element $f_{8m{+}s}$ in ${\mathfrak n}_2$ is defined (see \cite{FuWill}) for ${-}1 \le s \le 6,$ by the rule 
\begin{equation}
\label{A^2_2_grad}
deg(f_{8m{+}s})=\left\{ \begin{array}{l} (4m{+}s, 2m), \; {\rm if}\; s \le 1;\\(4m{+}s{-}2, 2m{+}1), \;{\rm if}\; s \ge 2. \end{array} \right.
\end{equation}
For example, $b_{6m{+}5}=f_{8m{+}7}$ has canonical bigrading equal to
$(4m+3,2m+2)$. In turn, the bigrading of $a_ {6m{+}5}=f_{8m{+}6}$ equals $(4m+4,2m+1)$. Hence both of these elements have a natural grading equal to $6m+5$.

Suppose that an infinite-dimensional Lie algebra ${\mathfrak g}$ is generated by some finite-dimensional subspace $V_1({\mathfrak g})$. For $n>1$, let $V_n({\mathfrak g})$ be the linear span of all the commutators of the elements of
$V_1({\mathfrak g})$ of length at most $n$ with an arbitrary arrangement of brackets. An increasing chain of finite-dimensional subspaces of the Lie algebra ${\mathfrak g}$ is given
$$
V_1({\mathfrak g}) \subset V_2({\mathfrak g}) \subset\dots \subset  V_n({\mathfrak g}) \subset \dots, \; \cup_{i=1}^{+\infty} V_i({\mathfrak g}) ={\mathfrak g}.
$$ 

The Gelfand-Kirillov dimension of the Lie algebra ${\mathfrak g}$
 is called the upper limit
$$
GK dim {\mathfrak g}= \limsup_{n \to +\infty} \frac{\log{\dim{V_n({\mathfrak g})}}}{\log{n}}.
$$
The finite Gelfand-Kirillov dimension means that there exists a polynomial $P(x)$ with which we can estimate all dimensions of the subspaces $V_n$, namely
$\dim{V_n({\mathfrak g})} < P(n)$ for all positive integers  $n$. 
For finite-dimensional Lie algebras, the Gelfand-Kirillov dimension is zero.

The growth function $F_{\mathfrak g}(n)$ of a pro-nilpotent Lie algebra ${\mathfrak g}$ can be expressed in terms of the codimension of the ideals of its lower central series 
$$
F_{\mathfrak g}(n)=\dim{V_n({\mathfrak g})}=\dim{({\mathfrak g}/{\mathfrak g}^{n{+}1})}.
$$ 
Let
${\mathfrak g}=\oplus_{i=1}^{+\infty}{\mathfrak g}_i$ be a naturally graded Lie algebra. Then its growth function can be expressed in terms of the dimensions of its homogeneous components  
$$
F_{\mathfrak g}(n)=\dim{V_n({\mathfrak g})}=\sum_{i=1}^n\dim{{\mathfrak g}_i}.
$$

For the Lie algebras
${\mathfrak m}_0, {\mathfrak m}_2$ и $W^+$, considered above, we have  
$$
F_{W^+}(n)=F_{{\mathfrak m}_0}(n)=F_{{\mathfrak m}_2}(n)=n{+}1,
$$ 
and this is the slowest (of all possible) growth. 

For an arbitrary naturally graded Lie algebras  ${\mathfrak g}=\oplus_{i=1}^{+\infty}{\mathfrak g}_i$ of width $d$ the growth function $F_{\mathfrak g}(n)$ grows no faster than $dn$
$
F_{\mathfrak g}(n) \le dn.
$

Consider the growth functions of our two Lie algebras ${\mathfrak n}_1$ and ${\mathfrak n}_2$. We have estimates
$$
\frac{3n}{2} \le F_{{\mathfrak n}_1}(n) \le \frac{3n{+}1}{2},
\frac{4n}{3} \le F_{{\mathfrak n}_2}(n) \le \frac{4n{+}2}{3}, \; \forall n \in {\mathbb N}.
$$
Consequently, piecewise-linear functions $F_{{\mathfrak n}_1}(n)$ and $F_{{\mathfrak n} _1}(n)$ grow linearly with average rates $\frac {3}{2}$ and $\frac{4}{3}$, respectively. It is interesting, but the slow growth of the so-called. characteristic Lie algebras of nonlinear hyperbolic partial differential equations is related to their  integrability. In particular, it turn out that the algebras ${\mathfrak n}_1$ and ${\mathfrak n}_1$ are directly related to the sine-Gordon and Tzitizeika equations \cite{Mill3}, and also with polynomial Lie algebras \cite{Bu}.

\section{Cohomology of ${\mathbb N}$-graded Lie algebras and Carnot extensions}\label{s3}

Consider the standard cochain complex of an $n$ -dimensional Lie algebra $\mathfrak{g}$:
$$
\begin{CD}
\mathbb K @>{d_0{=}0}>> \mathfrak{g}^* @>{d_1}>> \Lambda^2 (\mathfrak{g}^*) @>{d_2}>>
\dots @>{d_{n-1}}>>\Lambda^{n} (\mathfrak{g}^*) @>>> 0
\end{CD}
$$
where the symbol $d_1: \mathfrak{g}^* \rightarrow \Lambda^2 (\mathfrak{g}^*)$
denotes the dual map to the Lie bracket
$[ \, , ]: \Lambda^2 \mathfrak{g} \to \mathfrak{g}$, 
and the differential $d$ (in fact, it is a collection of mappings of $d_p$)
is a differentiation of the exterior algebra $\Lambda^*(\mathfrak{g}^*)$,
which continues $d_1$:

$$
d(\rho \wedge \eta)=d\rho \wedge \eta+(-1)^{deg\rho} \rho \wedge d\eta,
\; \forall \rho, \eta \in \Lambda^{*} (\mathfrak{g}^*).
$$
The relation $d^2=0$ is equivalent to the Jacobi identity in the Lie algebra $\mathfrak{g}$.

The cohomology of the complex $(\Lambda^{*}(\mathfrak{g}^*),d)$ is called
cohomology (with trivial coefficients) of the Lie algebra
$\mathfrak {g}$ and are denoted by $H^*(\mathfrak{g})$.

If the Lie algebra ${\mathfrak g}$ is a topological infinite-dimensional Lie algebra, then we need to consider the linear spaces $\Lambda^k(\mathfrak{g}^*)$ of continuous skew-symmetric $k$-linear functions. In this case, the cochain complex of the Lie algebra $\mathfrak{g}$ is infinite \cite{Fu}. 

The exterior algebra $\Lambda^* \mathfrak{g}$ of ${\mathbb N}$-graded Lie algebra алгебры Ли  
$\mathfrak{g}=\oplus_{\alpha=1}^{+\infty}\mathfrak{g}_{\alpha}$
can be provided with the second grading
$\Lambda^* \mathfrak{g} =
\bigoplus_{i=1}^{+\infty} \Lambda^*_{i} \mathfrak{g}$, where
$\Lambda^p_{i} \mathfrak{g}$ is a linear span of monomials 
$\xi_1 {\wedge} \xi_2 {\wedge} \dots {\wedge}\xi_p$ such that
$$\xi_1 \in \mathfrak{g}_{\alpha_1}, \xi_2 \in \mathfrak{g}_{\alpha_2},
\dots,  \xi_p \in \mathfrak{g}_{\alpha_p}, \;
\alpha_1{+}\alpha_2{+}\dots{+}\alpha_p=i.$$
The space of continuous skew-symmetric $p$-functions
$\Lambda^p (\mathfrak{g}^*)$ is also endowed with a second graduation
$\Lambda^p (\mathfrak{g}^*) =
\bigoplus_{\lambda} \Lambda^p_{(\lambda)} (\mathfrak{g}^*)$,
where the subspace
$\Lambda^p_{(\lambda)} (\mathfrak{g}^*)$ is given as
$$ \Lambda^p_{(\lambda)} (\mathfrak{g}^*)= \left\{ \omega
\in \Lambda^p (\mathfrak{g}^*)
 \; | \; \omega(v)=0, \: \forall v \in
 \Lambda^p_{(\mu)} (\mathfrak{g}), \: \mu \ne \lambda \right\}.$$
In the sequel, as a rule, we will consider the cohomology of finite-dimensional ${\mathbb N}$-bunded Lie algebras, and therefore
the direct sum symbol in the preceding formulas denotes the usual finite direct sum of subspaces. In the case of the infinite-dimensional Lie algebra ${\mathfrak g}$, this symbol needs to be clarified.

The second graduation is compatible with the differential
$d$ and with an external product
$$d \Lambda^p_{(\lambda)} (\mathfrak{g}^*)
\subset \Lambda^{p{+}1}_{(\lambda)} (\mathfrak{g}^*), \qquad
\Lambda^{p}_{(\lambda)} (\mathfrak{g}^*) \cdot
\Lambda^{q}_{(\mu)} (\mathfrak{g}^*) \subset
\Lambda^{p{+}q}_{(\lambda{+}\mu)} (\mathfrak{g}^*)
$$
The exterior product in  $\Lambda^*(\mathfrak{g}^*)$ 
induces the structure of a bigraded algebra in cohomology $H^*(\mathfrak{g})$
$$
H^{p}_{(\lambda)} (\mathfrak{g}) \otimes
H^{q}_{(\mu)} (\mathfrak{g}) \to
H^{p{+}q}_{(\lambda{+}\mu)} (\mathfrak{g}).
$$

\begin{example}
The cochain complex $(\Lambda^*({\mathfrak m}_0(n)), d)$ is generated by $a^1,b^1,a^2,\dots, a^n$ with differential
$$
da^1=db^1=0, da^2=a^1\wedge b^1, da^{i}=a^1\wedge a^{i-1}, 3 \le i \le n.
$$
\end{example}
\begin{example}
$(\Lambda^*({\mathfrak n}_1^{\pm}), d)$ generated by $a^1,b^1,a^2,\dots, a^{2k+1}, b^{2k+1}, a^{2k{+}2},\dots$ with differential $d$:
$$
\begin{array}{c}
da^1=db^1=0, da^2=a^1\wedge b^1, da^{2k+2}=\sum_{i+j=k, i<j} a^{2i+1}\wedge b^{2j+1},\\ da^{2k+1}=\pm\sum_{i+j=k} b^{2i+1}\wedge a^{2j},  db^{2k+1}=\sum_{i+j=k}a^{2i} \wedge a^{2j+1}, k\ge 1.
\end{array}
$$
\end{example}

Consider a linear space $V$ as an abelian Lie algebra.
 The central extension of the Lie algebra $\mathfrak{g}$ is an exact sequence
\begin{equation}
\label{exactseq}
\begin{CD}0 @>>> V@>{i}>>\tilde {\mathfrak g} @>\pi>>{\mathfrak g}@>>>0
\end{CD}
\end{equation}
of Lie algebras and their homomorphisms, where the image of homomorphism 
$i: V \to \tilde {\mathfrak{g}}$ 
is in the center $Z(\tilde{\mathfrak{g}})$ of the Lie algebra $\tilde{\mathfrak{g}}$.

Two extensions are said to be equivalent if there exists an isomorphism of Lie algebras $f: \tilde{\mathfrak g}_2 \to \tilde {\mathfrak g}_1$,
such that the following diagram is commutative
\begin{equation}
\label{equival}
\begin{CD}0 @>>> V@>{i_1}>>\tilde {\mathfrak g}_1 @>\pi_1>>{\mathfrak g}@>>>0\\
    @AAA @AA{Id}A @AA{f}A  @AA{Id}A @AAA\\
    0 @>>> V@>{i_2}>>\tilde {\mathfrak g}_2 @>\pi_2>>{\mathfrak g}@>>>0
\end{CD}
\end{equation}

As a vector space, the central extension of $\tilde{\mathfrak{g}}$ is a direct sum
$W \oplus \mathfrak{g}$ with standard inclusions $i$ and the projection $\pi$.
The Lie bracket in the vector space $V \oplus \mathfrak{g}$ is defined by the formula
$$
[(v,g), (w, h)]_{\tilde {\mathfrak g}}=(c(g,h), [g,h]_{{\mathfrak g}}), \; \; \; g, h \in {\mathfrak g},
$$
where $c$ is a bilinear function on ${\mathfrak g}$, which takes its values in the space $V$.
One can verify directly that the Jacobi identity for this bracket is equivalent to the fact that the bilinear function $c$
is a cocycle; the equality
$$
c([g,h]_{{\mathfrak g}},e)+c([h,e]_{{\mathfrak g}},g)+c([e,g]_{{\mathfrak g}},h)=0, \;\; \forall g,h,e \in {\mathfrak g},
$$
we assume that the Jacobi identity holds for the original Lie bracket $[g,h]_{{\mathfrak g}}$. We do not define a cochain complex of the Lie algebra $ {\mathfrak g}$ with values in the space (trivial ${\mathfrak g}$-module) $V$, referring the reader for details to \cite{Fu}. 

Let $c$ and $c'$ be the cohomology cocycles on the Lie algebra ${\mathfrak g}$, i.e. $c'=c+d\mu$, where $\mu: {\mathfrak g} \to V$ is some linear mapping, and by definition $d\mu(x, y)=\mu([x,y]_{\mathfrak g})$. Then the corresponding central extensions are equivalent. Indeed, one can verify that the linear map
$$
f=Id + \mu: V \oplus \mathfrak{g} \to V \oplus \mathfrak{g}, \; f(v,g)=(v{+}\mu(g),g),
$$
is an isomorphism of Lie algebras in the corresponding diagram (\ref{equival}). The converse statement is also true \cite{Fu}. 

Now let us consider the Carnot algebra
$
\tilde{\mathfrak{g}}=\oplus_{i{=}1}^k{\mathfrak g}_i.
$ 
It is obvious that the last homogeneous summand ${\mathfrak g}_k$ of our Lie algebra $\tilde{\mathfrak g} $ belongs to its center
$
Z(\tilde {\mathfrak g}).
$
Consider the quotient-algebra ${\mathfrak g}=\tilde{\mathfrak{g}}/{\mathfrak g}_k$. It is easy to see that it is also a Carnot algebra, and, as a vector space, it coincides with direct sum ${\mathfrak g}=\oplus_{i{=}1}^{k{-}1}{\mathfrak g}_i$. 
Thus, we have a central extension
$$
0 \to{\mathfrak g}_k \to  \tilde{\mathfrak g}  \to {\mathfrak g} \to 0, 
$$
which corresponds to some cocycle $\tilde c$  in
$H^2({\mathfrak g}, {\mathfrak g}_k)$. 
We fix now the basis $e_1, \dots, e_{j_k}$ of the subspace ${\mathfrak g}_k$. We can write our cocycle $\tilde c$ in the corresponding coordinates
$
\tilde c=(\tilde c_1,\dots,\tilde c_{j_k}).
$
The component $\tilde c_l$ of the cocycle $\tilde c$ has a simple meaning, it is the differential of the linear functional $e^l$ from the dual basis of the space ${\mathfrak g}_k^*$
$$
de^l=\tilde c_l, \; l=1,\dots, j_k.
$$

It is obvious that $\tilde c_l \in H^2_{(k)}({\mathfrak g})$, $l=1,\dots,j_k$.

\begin{propos}
Let ${\mathfrak g}{=}\oplus_{i{=}1}^{k{-}1}{\mathfrak g}_i$  be a Carnot algebra and  let $\tilde c_1,\dots,\tilde c_{j_k}$ be a set of $j_k$ cocycles in
$H^2_{(k)}({\mathfrak g}, {\mathbb K})$. The Lie algebra  $\tilde {\mathfrak g}{=}\oplus_{i{=}1}^{k} {\mathfrak g}_i$, defined by means of the corresponding central extension, is a Carnot algebra of nil-index $k$ and $\dim{ {\mathfrak g}_k}=j_k$, if and only if the cocycles $\tilde c_1, \dots, \tilde c_ {j_k}$ are linearly independent in the subspace of two-dimensional cohomology $H^2_{(k)}({\mathfrak g}, {\mathbb K})\subset H^2({\mathfrak g}, {\mathbb K})$ of grading $k$.
\end{propos}
\begin{proof}
It remains for us to verify the condition for the linear independence of cocycles $\tilde c_1,\dots,\tilde c_{j_k}$. Suppose that they are still linearly dependent
$
\alpha_1\tilde c_1+\dots+\alpha_{j_k}\tilde c_{j_k}=0
$.
Then
$
d(\alpha_1e^1+\dots+\alpha_{j_k}e^{j_k})=0, 
$
that means ${\mathfrak g}_k^* \cap {\mathfrak g}_1^* \ne 0$.
\end{proof}
\begin{remark}
The coboundaries of grading $k$ are absent in $Z^2_{(k)}({\mathfrak g}, {\mathbb K})$, hence
$$H^2_{(k)}({\mathfrak g}, {\mathbb K})=Z^2_{(k)}({\mathfrak g}, {\mathbb K}).$$
\end{remark}

We consider two central extensions $\tilde {\mathfrak g}$ and $\tilde{\mathfrak g} '$ of the Carnot algebra ${\mathfrak g}=\oplus_{i=1}^{k{-}1}{\mathfrak g}_i$. Suppose that they are both Carnot algebras of nil-index $k$.

Suppose also that there exists an isomorphism $f: \tilde {\mathfrak g}' \to \tilde {\mathfrak g}$. An equality holds on $f\left(\tilde {\mathfrak g}'^k\right)=\tilde {\mathfrak g}^k$ for the ideals of the lower central series. Whence follows that
$$
f\left( \tilde {\mathfrak g}_k'\right)=\tilde {\mathfrak g}_k,
$$
and we have a commutative diagram
\begin{equation}
\label{isom}
\begin{CD}0 @>>> \tilde {\mathfrak g}_k@>{i_1}>>\tilde {\mathfrak g} @>\pi_1>>{\mathfrak g}@>>>0\\
    @AAA @AA{\Psi}A @AA{f}A  @AA{\Phi}A @AAA\\
    0 @>>> \tilde {\mathfrak g}_k'@>{i_2}>>\tilde {\mathfrak g}' @>\pi_2>>{\mathfrak g}@>>>0,
\end{CD}
\end{equation}
where we denote by the symbol $\Psi$ an isomorphism between vector spaces $\tilde {\mathfrak g}_k$ and $\tilde {\mathfrak g}_k'$. The sybmol $\Phi$ denotes some automorphism of the Lie algebra  ${\mathfrak g}$.
\begin{propos}
\label{main_propos}
Let $\{ \tilde c_1,{\dots}, \tilde c_{j_k} \}$ и $\{ \tilde c_1',{\dots}, \tilde c_{j_k}'\}$ be two sets of linearly independent cocycles in $H^2_{(k)}({\mathfrak g}, {\mathbb K})$. 
They define isomorphic central extensions $\tilde {\mathfrak g}$ and $\tilde {\mathfrak g}'$ if and only if the linear spans $\langle \tilde c_1,{\dots}, \tilde c_{j_k} \rangle$ and $\langle \tilde c_1',{\dots}, \tilde c_{j_k}'\rangle$ lie in one orbit of the linear action of the group of graded automorphisms ${\rm Aut}_{gr}({\mathfrak g})$  on the subspace  $H^2_{(k)}({\mathfrak g}, {\mathbb K})$ градуировки $k$ in two-cohomology.   
\end{propos}
\begin{proof}
We prove this proposition in one direction, leaving the converse statement as an elementary exercise to the reader. Let $L$ и $L'$ be two subspaces in the cohomology  $H^2_{(k)}({\mathfrak g})$ such that  $\Phi(L')=L$, where $\Phi \in Aut_{gr}({\mathfrak g})$ denotes a graded automorphism of the Lie algebra ${\mathfrak g}$. We assume that both subspaces are given in the form of linear hulls of two sets of homogeneous cocycles $\langle \tilde c_1,{\dots}, \tilde c_{j_k} \rangle$ и $\langle \tilde c_1',{\dots}, \tilde c_{j_k}'\rangle$. We apply the standard formulas for the $\Phi$-action on bilinear forms
$$
(\Phi \cdot \tilde c_l')(x,y)= \tilde c_l'(\Phi^{-1} x, \Phi^{-1} y), \forall x,y \in {\mathfrak g}, \; l=1,\dots, j_k, \; \Phi \in Aut_{gr}({\mathfrak g}).
$$
Bilinear functions $\Phi \cdot \tilde c_1', \dots, \Phi \cdot \tilde c_{j_k}'$ form a basis in $L$, as well as$ \tilde c_1,{\dots}, \tilde c_{j_k} $.Thus there exists an isomorphism $\psi: L \to L$ such that
$$
\psi \left( \Phi \cdot \tilde c_1'\right)=\tilde c_1, \dots, \psi \left( \Phi \cdot \tilde c_{j_k}'\right)=\tilde c_{j_k}.
$$ 
Define $\tilde {\mathfrak g}=L\oplus{\mathfrak g}$ and $\tilde {\mathfrak g}'=L'\oplus{\mathfrak g}$ as vector spaces. Define an isomorphism $f: \tilde {\mathfrak g}' \to \tilde {\mathfrak g}$ by the formula $f((a,g))=(\Psi(a),\Phi(g))$, where $\Psi(a)=\psi \cdot (\Phi \cdot a)$ and $\Phi$ is an automorphism of the Lie algebra ${\mathfrak g}$. 

It remains only to verify the compatibility of the linear isomorphism $f$ with the Lie brackets of the Lie algebras $\tilde {\mathfrak g}'$ and  $\tilde {\mathfrak g}$.
$$
\begin{array}{c}
f([(a,g),(b,h)]_{\tilde {\mathfrak g}'})=f((\tilde c'(g,h),[g,h]_{\mathfrak g}))
=(\Psi\tilde c'(g,h),\Phi([g,h]_{\mathfrak g}))=\\
=(\tilde c(\Phi g,\Phi h),[\Phi g, \Phi h ]_{\mathfrak g})=[(\Psi a,\Phi g),(\Psi b,\Phi h)]_{\tilde {\mathfrak g}}=[f(a,g),f(b,h)]_{\tilde {\mathfrak g}}
\end{array}
$$
\end{proof}
\begin{remark}There is an internal restriction on the dimension $j_k$ of the space by which we extend the Carnot algebra
$$
j_k \le \dim{H^2_{(k)}({\mathfrak g}, {\mathbb K})}.
$$
\end{remark}

Suppose
$
\mathfrak{g}=\oplus_{i{=}1}^n{\mathfrak g}_i
$ 
is a Carnot algebra.  We illustrate the process of building it using the child constructor in Figure \ref{lego_towers}. First we take cubes, each of which corresponds to the generator of the Carnot algebra. In our examples -- these are the two lower cubes, they make up the first floor of our tower. Next we attach one cube of the second floor, it corresponds to the switch of two generators. And so on, we build a floor by floor, adding each time the number of cubes, equal to the dimension of this extension of Carnot. The number of cubes of the $k$-th floor is the dimension of the $k$-th homogeneous component of our Carnot algebra.
The total number of cubes in the tower is equal to the dimension $\dim{\mathfrak{g}}$ of the Carnot algebra, the height of the tower is equal to its nil-index $s({\mathfrak{g}})$. The width of the tower is equal to the width of the Carnot algebra.
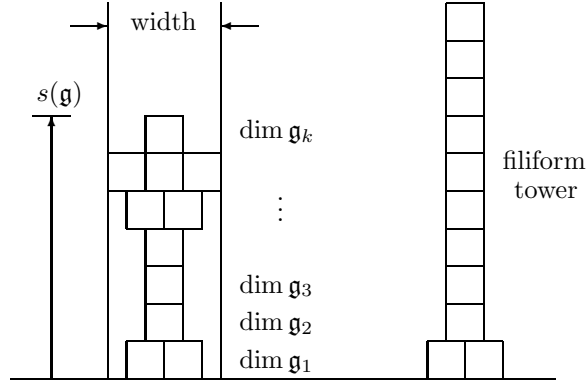
\begin{figure}[tb]
\begin{picture}(112,50)(-40,-3)
  \multiput(0,0)(5,0){3}{\line(0,1){5}}
  \multiput(0,0)(0,5){2}{\line(1,0){10}}
  \multiput(2.5,5)(5,0){2}{\line(0,1){15}}
  \multiput(2.5,10)(0,5){2}{\line(1,0){5}}
 \multiput(0,20)(5,0){3}{\line(0,1){5}}
  \multiput(0,20)(0,5){2}{\line(1,0){10}}
 \multiput(-2.5,25)(5,0){4}{\line(0,1){5}}
  \multiput(-2.5,25)(0,5){2}{\line(1,0){15}}
\multiput(2.5,30)(5,0){2}{\line(0,1){5}}
  \multiput(2.5,35)(0,5){1}{\line(1,0){5}}

\put(-10,0){\vector(0,1){35}}
\put(-15.5,0){\line(1,0){72.5}}
\put(-12.5,35){\line(1,0){5}}
\put(-12,37){$s(\mathfrak{g})$}
\put(15,1.5){$\dim{\mathfrak{g}_1}$}
\put(15,6.5){$\dim{\mathfrak{g}_2}$}
\put(15,11.5){$\dim{\mathfrak{g}_3}$}
\put(20,21){$\vdots$}
\put(15,32){$\dim{\mathfrak{g}_k}$}
\put(-2.5,0){\line(0,1){50}}
\put(12.5,0){\line(0,1){50}}
\put(-7.5,47){\vector(1,0){5}}
\put(17.5,47){\vector(-1,0){5}}
\put(0.5,46.5){${\rm width}$}

 \multiput(40,0)(5,0){3}{\line(0,1){5}}
  \multiput(40,0)(0,5){2}{\line(1,0){10}}
  \multiput(42.5,5)(5,0){2}{\line(0,1){45}}
  \multiput(42.5,10)(0,5){9}{\line(1,0){5}}
 \put(50,28){${\rm filiform}$}
 \put(51.5,24){${\rm tower}$}

\end{picture}
\caption{Carnot algebras and towers of bricks}
\label{lego_towers}
\end{figure}

\section{Graded automorphisms of Carnot algebras}\label{s4}

Let ${\mathfrak g}=\oplus_{i{=}1}^n{\mathfrak g}_i$ be a Carnot algebra and $\varphi$ be its graded automorphism. Consider its restriction
$\varphi_1=\varphi |_{{\mathfrak g}_1}$ onto a homogeneous subspace ${\mathfrak g}_1$. Fixing some basis in ${\mathfrak g}_1$, we can define the matrix $A$ of the map  $\varphi_1$. Because the elements of the basis ${\mathfrak g}_1$ are generators of our Carnot algebra ${\mathfrak g}=\oplus_{i{=}1}^n{\mathfrak g}_i$, then the automorphism $\varphi$ is completely determined by the matrix $A$ and we have, therefore, a homomorphism
$$
\Phi: Aut_{gr}({\mathfrak g}) \to GL(q,{\mathbb K}), \; q=\dim{{\mathfrak g}_1}.
$$
\begin{remark}
The group $Aut_{gr}({\mathfrak g})$ is always non-trivial, because it contains a one-dimensional algebraic torus ${\mathbb K}^*$ of homotheties
$$
\phi(v)=\alpha^k v, v \in {\mathfrak g}_k, k=1,\dots,n, \alpha \in {\mathbb K}^*.
$$
\end{remark}

\begin{example}
\label{Aut_L_2_3}
${\mathfrak m}_0^3(3)\cong {\mathcal L}(2,3)$.  We choose a basis in this Lie algebra
$$
a_1, b_1, a_2=[a_1, b_1], a_3=[a_1,[a_1,b_1]], b_3=[b_1,[a_1,b_1]].
$$
In this case an arbitrary linear operator
$\varphi_1: \langle a_1, b_1 \rangle \to \langle a_1, b_1 \rangle$ continues to the automorphism $\varphi$ of the Lie algebra ${\mathcal L}(2,3)$. Hence 
the group $Aut_{gr}({\mathcal L}(2,3))$ is isomorphic to $GL(2,{\mathbb K})$ and for the automorphism $\varphi$ we have these formulas
$$
\begin{array}{c}
\varphi(a_1)=\alpha a_1 + \beta b_1, \varphi(b_1)=\rho a_1+\mu b_1, \varphi(a_2)=\det{A} a_2, \\
\varphi(a_3)= \det{A}(\alpha a_3+\beta b_3), \varphi(b_3)= \det{A} (\rho a_3+\mu b_3), \det{A}=\alpha\mu- \rho \beta.
\end{array}
$$
\end{example}
\begin{example}
\label{m_0(4)}
${\mathfrak m}_0(4)$  is given by a basis $a_1, b_1, a_2, a_3,$ and
relations
$$
[a_1, b_1]=a_2, [a_1,a_2]=a_3.
$$
Consider a linear operator
$\varphi_1: \langle a_1, b_1 \rangle \to \langle a_1, b_1 \rangle$ defined by the matrix $A{=}\begin{pmatrix} \alpha & \rho\\ \beta & \mu \end{pmatrix}$. It continues to an automorphism $\varphi$ of the Lie algebra ${\mathfrak m}_0(4)$ if and only if
$
\rho=0.
$
In fact, we have
$$
\varphi(a_2)=[\varphi (a_1), \varphi(b_1)]=\det{A} a_2, \varphi(a_3)=[\varphi (a_1), \varphi(a_2)]=\alpha \det{A} a_3.
$$
In the same time
$
0=[\varphi (b_1), \varphi(a_2)]=[\rho a_1 +\mu b_1, \det{A}a_2]=\rho \det{A}a_3.
$
Hence $Aut_{gr}({\mathfrak m}_0(4))$ is isomorphic to the subgroup $LT(2,{\mathbb K})$ lower-triangular matrices in $GL(2,{\mathbb K})$.
An arbitrary $\varphi \in Aut_{gr}({\mathfrak m}_0(4))$ is given by the formulas
$$
\varphi(a_1)=\alpha a_1 + \beta b_1, \varphi(b_1)=\mu b_1, \varphi(a_2)=\alpha \mu a_2, \varphi(a_3)=\alpha^2\mu a_3.
$$
\end{example}
\begin{lemma}
\label{m_0_auto}
$Aut_{gr}( {\mathfrak m}_0^{r_1,\dots,r_k}(n) ) \cong LT(2,{\mathbb K})$, group of lower-triangular matrices of the second order. For an arbitrary graded automorphism
$\varphi: {\mathfrak m}_0^{r_1,\dots,r_k}(n) \to  {\mathfrak m}_0^{r_1,\dots,r_k}(n)$  for $n\ge 4$ the following formulas are valid ($\alpha \ne 0, \mu \ne 0, \beta \in {\mathbb K}$):
\begin{equation}
\label{m_0^S_automorphism}
\begin{array}{c}
\varphi(a_{1})=\alpha a_{1} + \beta b_{1},
\varphi(b_1)= \mu b_1,\\
\varphi(a_{m})=\alpha^{m{-}1} \mu a_{m},  1 \le m \le n, m \ne r_1,\dots, r_k,\\
\left\{ \begin{array}{r}
\varphi(a_{r_i})=\alpha^{r_i{-}2} \mu(\alpha a_{r_i}+\beta b_{r_i}), \\
\varphi(b_{r_i})=\alpha^{r_i{-}2} \mu^2 b_{r_i}, 
\end{array}\right.
i=1,\dots, k.\\
\end{array}
\end{equation}
\end{lemma}
\begin{proof}
We introduce the operator
$\varphi_1: \langle a_1, b_1 \rangle \to \langle a_1, b_1 \rangle$. Let it be given by the matrix $A{=}\begin{pmatrix} \alpha & \rho\\ \beta & \mu \end{pmatrix}$.
If $r_1 >3$, then our Lie algebra ${\mathfrak m}_0^{r_1,\dots,r_k}(n)$ 
is obtained by successive central Carnot extensions from the Lie algebra Ли $m_0(4)$ and, therefore,
from the arguments of Example \ref{m_0(4)} it follows that $\rho=0$.

In the case of $r_1=3$, our Lie algebra ${\mathfrak m}_0^{r_1,\dots,r_k}(n)$ is obtained by Carnot extensions from ${\mathfrak m}_0^3(4)$ and
then the equalities hold
$$
0=[\varphi(b_1),\varphi(a_3)]=[\rho a_1+ \mu b_1, \alpha \det{A} a_3]= -\rho \alpha \det{A} a_4.
$$
Hence, in this case also we have $\rho=0$.  The formulas (\ref{m_0^S_automorphism}) are easily verified by successive calculations.
\end{proof}

\begin{corollary}
Groups of graded automorphisms of Carnot algebras  $${\mathfrak m}_1^{r_1,\dots,r_k}(2m{-}1),  {\mathfrak m}_{0,2}^{r_1,\dots,r_k}(2m), {\mathfrak m}_{0,3}^{r_1,\dots,r_k}(2m{+}1),$$
are isomorphic as $m \ge 3$ to a two-dimensional algebraic torus
${\mathbb K}^*\times {\mathbb K}^*$, i.e.
 $$
Aut_{gr}({\mathfrak m}_1^{r_1,\dots,r_k}(2m{-}1) ) \cong Aut_{gr}({\mathfrak m}_{0,2}^{r_1,\dots,r_k}(2m) ) \cong Aut_{gr}({\mathfrak m}_{0,3}^{r_1,\dots,r_k}(2m{+}1) ) \cong{\mathbb K}^*\times {\mathbb K}^*
$$
Operator
$\varphi_1$ of the automorphism restriction $\varphi$ to the subspace ${\mathfrak g}_1$ for each of the three kinds of Carnot algebras will have a diagonal matrix $A{=}\begin{pmatrix} \alpha & 0 \\ 0 & \mu \end{pmatrix}$.
And in the formulas (\ref{m_0^S_automorphism}) for $\varphi$ it is necessary to put $\beta=0$ and also to determine successively $\varphi$ on additional elements of the basis of these Carnot algebras
$$
\varphi(b_{2m{-}1})=\alpha^{2m{-}3} \mu^2 b_{2m{-}1}, 
\varphi(a_{2m})=\alpha^{2m{-}2} \mu^2 a_{2m}, 
\varphi(a_{2m{+}1})=\alpha^{2m{-}1} \mu^2 a_{2m{+}1}.
$$
\end{corollary}
\begin{proof}
In the case ${\mathfrak m}_1^{r_1,\dots,r_k}(2m{-}1)$ we consider the equalities
$$
0=[\varphi(a_1),\varphi(a_{2m{-}2})]=[\alpha a_1+\beta b_1, \alpha^{2m{-}3}\mu a_{2m{-}2}]=-\beta \alpha^{2m{-}3}\mu b_{2m{-}1}.
$$
Whence $\beta=0$. 
The remaining cases are treated similarly.
\end{proof}

\begin{lemma}
The group of graded automorphisms of the Carnot algebra ${\mathfrak n}_1^+(n)$ for $n \ge 4$
is isomorphic to the direct product $O(2, {\mathbb R})\times {\mathbb R}_{>0}$ orthogonal group and a multiplicative group of positive real scalars 
 $$
 Aut_{gr}({\mathfrak n}_1^+(n)) \cong O(2, {\mathbb R})\times {\mathbb R}_{>0}
$$
for $n \ge 4$ we have 
if $1 \le 2q, 2k+1 \le n$ with some $c >0,  t \in {\mathbb R}$
\begin{equation}
\label{n_1+_automorphism_b}
\begin{array}{c}
\varphi(u_{2k{+}1})=c^{2k{+}1}( \cos{t} u_{2k{+}1}+ \sin{t} v_{2k{+}1}),\\
\varphi(v_{2k{+}1})= \pm c^{2k{+}1}({-}\sin{t} u_{2k{+}1}+\cos{t} v_{2k{+}1}),\\
\varphi(w_{2q})=\pm c^{2q} w_{2q}.
\end{array}
\end{equation}
\end{lemma}
\begin{proof}
Let $\varphi$ be a graded automorphism of the Carnot algebra ${\mathfrak n}_1^+(n)$. One can verify the formulas by direct calculations
$$
\begin{array}{c}
\varphi(u_{1})=\alpha u_{1} + \beta v_{1},
\varphi(v_1)= \rho u_1+ \mu v_1,\\
\varphi(w_{2})=\det{A} w_{2},  \det{A}=\alpha\rho-\mu\beta, \\
\varphi(u_{3})=\det{A}(\mu  u_{3} -\rho v_3), \varphi(v_{3})=\det{A}(-\beta  u_{3} +\alpha v_3).\\
\end{array}
$$
Further from
$[\varphi(u_1), \varphi(u_3)]=[\varphi(v_1), \varphi(v_3)]=0$ follows that $\alpha \rho + \mu \beta=0$. On the other hand, the equality
$[\varphi(u^1),\varphi(v^3)]=[\varphi(u^3),\varphi(v^1)]$ is equivalent to
$$
\alpha^2+\beta^2=\mu^2+\rho^2.
$$
Taking $c>0$ and $t$ such that $\alpha =c \cos{t}, \beta=c\sin{t}$ we obtain expressions for $\varphi(u_3)$ and $\varphi(v_3)$. Now it remains, as an elementary exercise, to verify all the remaining formulas from (\ref{n_1+_automorphism_b}). 
\end{proof}
\begin{lemma}
The group of graded automorphisms of the Carnot algebra ${\mathfrak n}_1^-(n)$ for $n \ge 4$
is isomorphic to the direct product $O(1,1)\times {\mathbb R}_{>0}$ Lorentz group and the multiplicative group of positive real scalars 
 $$
 Aut_{gr}({\mathfrak n}_1^-(n)) \cong O(1,1)\times {\mathbb R}_{>0}
$$
for $1 \le 2q, 2k+1 \le n$ with some $c >0,  t \in {\mathbb R}$
\begin{equation}
\label{n_1_automorphism_b}
\begin{array}{c}
\varphi(u_{2k{+}1})=c^{2k{+}1}( \cosh{t} u_{2k{+}1} + \sinh{t} v_{2k{+}1}),\\
\varphi(v_{2k{+}1})= \pm c^{2k{+}1}(\sinh{t} u_{2k{+}1}+\cosh{t} v_{2k{+}1}),\\
\varphi(w_{2q})=\pm c^{2q} w_{2q}.
\end{array}
\end{equation}
\end{lemma}
\begin{proof}
Completely analogous to the proof of the preceding lemma. We also remark that there are isomorphisms ${\mathfrak n}_1^-(2) \cong {\mathfrak n}_1^+(2), {\mathfrak n}_1^-(3) \cong {\mathfrak n}_1^+(3)$.
\end{proof}
  
We now pass to the infinite-dimensional Lie algebra ${\mathfrak n}_2$ given by an infinite basis $f_1,f_2,f_3, \dots$.  It is easy to verify that the following linear spans define ideals in ${\mathfrak n}_2$:
$$
\begin{array}{c}
I_{6m{+}1}=\langle f_{8m{+}2},f_{8m{+}3},f_{8m{+}4}, \dots\rangle,\\
J_{6m{+}1}=\langle f_{8m{+}1},f_{8m{+}3},f_{8m{+}4}, \dots\rangle,\\
K_{6m{+}1}=\langle f_{8m{+}1}{-}f_{8m{+}2},f_{8m{+}3},\dots  \rangle,
\end{array}
\begin{array}{c}
I_{6m{+}5}=\langle f_{8m{+}6},f_{8m{+}7},f_{8m{+}8}, \dots\rangle,\\
J_{6m{+}5}=\langle f_{8m{+}5},f_{8m{+}7},f_{8m{+}8}, \dots \rangle,\\
 K_{6m{+}5}=\langle f_{8m{+}5}{-}f_{8m{+}6},f_{8m{+}7},\dots \rangle, 
\end{array}
$$

\begin{definition}
\label{n_2_quotients}
Let $m \ge 1$. We define finite-dimensional quotient Lie algebras 
${\mathfrak n}_{2,s}(6m+r), s=1,2,3,\; r=1,5,$ in the following way
$$
{\mathfrak n}_{2,1}(6m{+}r){=}{\mathfrak n}_2/I_{8m{+}r}, 
{\mathfrak n}_{2,2}(6m{+}r){=}{\mathfrak n}_2/J_{8m{+}r}, 
{\mathfrak n}_{2,3}(6m{+}r){=}{\mathfrak n}_2/K_{8m{+}r}, \; r=1,5. 
$$
\end{definition}
It is easy to see that the quotient Lie algebras
${\mathfrak n}_{2,s}(6m{+}1),  {\mathfrak n}_{2,s}(6m{+}5), s{=}1,2,3,$ 
are Carnot algebras of nil-indices  
$6m{+}1, 6m{+}5$ respectively.
\begin{lemma}
\label{n_2_automorphism_2}
Groups of graded automorphisms of Carnot algebras
\begin{equation}
\label{n_2_quotient}
{\mathfrak n}_2(n)={\mathfrak n}_2/({\mathfrak n}_2))^{n+1}, {\mathfrak n}_{2,s}(6m{+}1),  {\mathfrak n}_{2,s}(6m{+}5), s{=}1,2,3,
\end{equation}
are isomorphic for $n\ge 7, m \ge 1,$ to two-dimensional algebraic torus
${\mathbb K}^*\times {\mathbb K}^*$.
 
Let $n \ge 7$ и $\varphi$ be a graded automorphism of the Carnot algebra from the list (\ref{n_2_quotient}). Its value on the basis vector $f_{8m{+}s}$ is defined by the formula
\begin{equation}
\label{avtom_n_2}
\varphi(f_{8m{+}s})=\left\{ \begin{array}{l} \alpha^{4m{+}s}\mu^{2m}f_{8m{+}s}, \;  {-}1 \le s \le 1;\\\alpha^{4m{+}s{-}2}\mu^{2m{+}1}f_{8m{+}s}, \; 2 \le s \le 6. \end{array} \right., \;  (\alpha,\mu) \in {\mathbb K}^*{\times} {\mathbb K}^*.
\end{equation}
\end{lemma}
\begin{proof}
By direct computations we obtain the values of the automorphism $\varphi$ on the first basis vectors  $f_1,f_2,\dots, f_6,$ for a nondegenerate matrix $A$
$$
\begin{array}{c}
\varphi(f_{1})=\alpha f_{1} + \beta f_{2},
\varphi(f_{2})= \rho f_1+\mu f_{2},
\varphi(f_{3})=\det{A} f_{3}, \\
\varphi(f_{4})=\alpha \det{A} f_{4}, 
\varphi(f_{5})=\alpha^{2}\det{A} f_{5},
\varphi(f_{6})=\alpha^{3}\det{A} f_{6},
\end{array}, \; A=\begin{pmatrix}\alpha & \rho \\ \beta & \mu \end{pmatrix}.
$$
We note that $3[f_2,f_5]=[f_3, f_4]=3f_7$, then we have the equality
$$
3\alpha\det{A}^2f_7={-}6\rho\alpha^2\det{A}f_6+3\mu\alpha^2\det{A}f_7.
$$
From where $\rho=0$ and $\det{A}=\alpha\mu$. On the other hand, $[f_1,f_6]=0$, which implies 
$[\varphi(f_1),\varphi(f_6)]=-\beta\alpha^3\det{A}f_8=0$, i.e. $\beta=0$. Further formulas from (\ref {avtom_n_2}) follow from the bigraded structure ${\mathfrak n}_2$ or are checked by induction directly.
\end{proof}
\begin{corollary}
\label{n_2_automorphism_2^3}
Groups of graded automorphisms of Carnot algebras  
\begin{equation}
\label{n_2^3_quotient}
{\mathfrak n}_2^3(n)={\mathfrak n}_2^3/({\mathfrak n}_2^3))^{n+1}, {\mathfrak n}_{2,s}^3(6m{+}1),  {\mathfrak n}_{2,s}^3(6m{+}5), s{=}1,2,3,
\end{equation}
are isomorphic for $n \ge 7, m \ge 1,$ to the two-dimensional algebraic torus 
${\mathbb K}^*\times {\mathbb K}^*$.
 
Let $n \ge 7$ и $\varphi$ be a graded automorphism of a Carnot algebra from the list (\ref{n_2^3_quotient}). Its values on the basis vectors $ f_j $ are determined by the formula (\ref{avtom_n_2}), to which it is necessary to add the formula
$
\varphi(z)=\alpha^2\mu^2z.
$
\end{corollary}
\section{Filiform Carnot algebras and their central extensions}\label{s5}

In papers \cite{Mill1, Mill2} it was shown that the problem of classifying $\mathbb N$-graded filiform Lie algebras of width one can be solved by inductive process of successive one-dimensional central extensions. Later, this idea was applied to the classification of a similar class
$\mathbb N$-graded Lie Algebras \cite{BKT}. What is this recursive procedure? In particular, it includes the computation of the space of two-cohomology $H^2(\mathfrak{g})$ of the Lie algebra ${\mathfrak g}$, that Lie algebra is what we are going to expand. Sometimes the subspace $H^2_{(n+1)}(\mathfrak {g})$ is trivial, which means the absence of Carnot extensions of the Carnot algebra under the study $\mathfrak{g}$ of length $n$.
\begin{propos}
\label{cohom_m_1}
Rhe Lie algebra ${\mathfrak m}_1(2m{+}1)$ cannot be extended to aCarnot algebra   
${\mathfrak g}=\oplus_{i=1}^{2m{+}2}{\mathfrak g}_i$ of length $2m+2$.
\end{propos}
\begin{proof} The assertion follows from the triviality of the cohomology subspace of interest to us
$H^2_{(2m{+}2)}({\mathfrak m}_1(2m{+}1)$ of grading $2m+2$. Indeed,
an arbitrary homogeneous $2$-form $\omega$ of the grading $2m{+}2$ can be written as the sum
$$
\omega= (x_1 a^1 + y_1 b^1) \wedge b^{2m{+}1}+\sum_{k{=}2}x_k a^k{\wedge} a^{2m{+}2{-}k}.
$$
Comparing the coefficients in the expansion of $d\omega$ with respect to the standard basis of $3$-forms, we obtain the condition $y_1=0$ and in addition to it the following system of equations
$$
\left\{
\begin{array}{l}
x_1+x_2=0,\\
 {-}x_1+x_2+x_3=0,\\
 \dots,\\ ({-}1)^mx_1+x_{m{-}1}+x_m=0,\\
 ({-}1)^{m{+}1}x_1+x_m=0,
\end{array}\right.
$$
which is equivalent to the closure condition $d\omega=0$. Thus, we obtain the chain of equalities $x_1=x_2=\dots=x_m=0$.
\end{proof}
\begin{example}
A Lie algebra ${\mathfrak m}_{0}^{2m{+}1}(2m{+}1)$
is given by its basis
$a_1, a_2,\dots, a_{2m{+}1}$, $b_1, b_{2m{+}1}$ and structure relations
\begin{equation}
\begin{split}
[a_1, b_1]=a_2, [a_1,a_i]=a_{i+1}, \quad i=2,\dots,2m; \\
[b_1,a_{2m}]={-}b_{2m{+}1}, [a_l,a_{2m-l+1}]=({-}1)^{l+1}b_{2m+1}, \quad l=2,\dots,m.
\end{split}
\end{equation}
\end{example}
\begin{definition}
\label{m_0_S}
Let $S_n=\{ r_1,\dots, r_k\}$ be an ordered set of $k$ odd natural numbers such that  
$3 \le r_1 <\dots < r_k \le n,$ где $n \ge 3, k \ge 1$.
Define a Lie algebra ${\mathfrak m}_0^{S_n}(n)={\mathfrak m}_0^{r_1,\dots,r_k}(n)$ as a $k$-dimensional extension of the filiform Lie algebra ${\mathfrak m}_0(n)$. It is given by the following set of $k$ cocycles of gradings $r_1, \dots, r_k$:
\begin{equation}
 \omega_{r_j}={-}b^1{\wedge}a^{r_j{-}1}{+}\sum_{l{=}1}^{\frac{1}{2}(r_j{-}1)}({-}1)^l a^l{\wedge}a^{r_j{-}l}, \; j=1,\dots, k. 
\end{equation}
\end{definition}
This central extension is not a Carnot extension, because the cocycles $\omega_ {r_j}$ are not taken from the subspace
$H^2_{(n{+}1)}( {\mathfrak m}_0(n))$. The resulting Lie algebra ${\mathfrak m}_0^{S_n}(n)={\mathfrak m}_0^{r_1,\dots,r_k}(n)$ has the same nil-index $n$ as the Lie algebra  ${\mathfrak m}_0(n)$.
\begin{definition}
\label{m_1_S}
Let $r_1,\dots, r_k,$ be an ordered set of $k$ odd natural numbers such that  
$3 \le r_1 <\dots < r_k \le 2m{-}3,$ where $m \ge 3, k \ge 1$.
Define a Lie algebra ${\mathfrak m}_1^{r_1,\dots,r_k}(2m{-}1)$ as the $(k{+}1)$-dimensional central extension of  the filiform Lie algebra  ${\mathfrak m}_0(2m{-}2)$ which is defined by the set of $k+1$ cocycles of gradings $2m{-}1, r_1, \dots, r_k$
\begin{equation}
\begin{array}{c}
\omega_{2m{-}1}={-}b^1{\wedge}a^{r_j{-}1}{+}\sum_{l{=}1}^{\frac{1}{2}(r_j{-}1)}({-}1)^l a^l{\wedge}a^{r_j{-}l},\\
\omega_{r_j}={-}b^1{\wedge}a^{r_j{-}1}{+}\sum_{l{=}1}^{\frac{1}{2}(r_j{-}1)}({-}1)^l a^l{\wedge}a^{r_j{-}l}, \; j=1,\dots, k.
\end{array}
\end{equation}
\end{definition}

\begin{definition}
\label{m_0,2_S}
Let $r_1,\dots, r_k$ be an ordered set of $k$ odd natural numbers  
$3 \le r_1 <\dots < r_k \le 2m{-}3,$ где $m \ge 3, k \ge 1$.
Define the Lie algebra  ${\mathfrak m}_{0,2}^{r_1,\dots,r_k}(2m)$ as the central $(k{+}1)$-dimensional Lie algebra extension  ${\mathfrak m}_0^{2m{-}1}(2m{-}1)$, which is defined by the set of $2$-cocycles of gradings  $2m, r_1, \dots, r_k$:
\begin{equation}
\begin{array}{c}
\tilde \omega_{2m}=a^1{\wedge}b^{2m{-}1}{-}(m{-}1)b^1{\wedge}a^{2m{-}1}{+}\sum_{l{=}1}^{\frac{1}{2}(r_j{-}1)}({-}1)^l a^l{\wedge}a^{r_j{-}l},\\
\omega_{r_j}={-}b^1{\wedge}a^{r_j{-}1}{+}\sum_{l{=}1}^{\frac{1}{2}(r_j{-}1)}({-}1)^l(m{-}l) a^l{\wedge}a^{r_j{-}l}, \; j=1,\dots, k.
\end{array}
\end{equation}
\end{definition}

The following lemma describes (up to isomorphism) all Carnot algebras of length $n+1$ that can be obtained as Carnot extensions of the Lie algebra ${\mathfrak m}_0^{r_1,\dots,r_k}(n)$ of lenfgth $n$.
\begin{lemma}
\label{lemma_m_0}
Let $n \ge 4$ и ${\mathfrak g}=\oplus_{i=1}^{n+1}{\mathfrak g}_i$ be a Carnot algebra of length $n+1$, such that 
$$
{\mathfrak g}/{\mathfrak g}_{n+1} \cong {\mathfrak m}_0^{r_1,\dots,r_k}(n), 
\dim{{\mathfrak g}_n}+\dim{{\mathfrak g}_{n+1}}\le 3.
$$
Then ${\mathfrak g}=\oplus_{i=1}^{n+1}{\mathfrak g}_i$ is isomorphic to one and only one Carnot algebra from the following list:
\begin{equation}
\begin{split}
{\mathfrak m}_0^{r_1,\dots,r_k}(n{+}1),
{\mathfrak m}_1^{r_1,\dots,r_k}(2m{+}1),
{\mathfrak m}_0^{r_1,\dots,r_k, 2m{+}1}(2m{+}1),
{\mathfrak m}_{0,2}^{r_1,\dots,r_{k{-}1}}(2m{+}2).
\end{split}
\end{equation}
\end{lemma}
\begin{proof}

We first consider the case of even $n =2m \ge 4$.  The subspace $H^2_{(2m{+}1)}\left({\mathfrak m}_0^{r_1,\dots,r_k}(2m) \right)$ is two-dimensional and the following basis of cocycles can be chosen
in it 
$$
a^1{\wedge}a^{2m}, \; {-} b^1{\wedge}a^{2m}{+}\sum_{i{=}2}^{m} ({-}1)^ia^i {\wedge} a^{2m{+}1{-}i}.
$$

According to the Lemma \ref{m_0_auto} the group $Aut_{gr}({\mathfrak m}_0^{r_1,\dots,r_k}(2m))$ acts in this basis as a group of linear operators with upper-triangular matrices of the following form
$$
\alpha^{2m{-}1}\mu\begin{pmatrix} \alpha & -\rho \\ 0 & \mu\end{pmatrix}, 
\; \alpha, \mu \ne 0.
$$
The corresponding projective action on the projective line has exactly two orbits represented by points $(1:0)$ and $(0:1)$. These points of the projective line correspond to one-dimensional central extensions ${\mathfrak m}_0^{r_1,\dots,r_k}(2m{+}1)$ and ${\mathfrak m}_1^{r_1,\dots,r_k}(2m{+}1)$ respectively.
In this case we also have a unique two-dimensional central extension   ${\mathfrak m}_0^{r_1,\dots,r_k, 2m{+}1}(2m{+}1)$.

Let $n=2m-1\ge 5$ be odd and $r_k< 2m-1$. Then the subspace $H^2_{(2m)}\left({\mathfrak m}_0^{r_1,\dots,r_k}(2m{-}1) \right)$ is the one-dimensional span  $\langle a^1 {\wedge}a^{2m{-}1}\rangle$. Obviously, in this case we will have a unique nontrivial extension ${\mathfrak m}_0^{r_1,\dots,r_k}(2m)$.

If $n=2m-1 \ge 5$, and $r_k=2m-1$, then our subspace $H^2_{(2m)}\left({\mathfrak m}_0^{r_1,\dots,r_k}(2m{-}1) \right)$ is two-dimensional  and we can fix a basis of cocycles in it
$$
 a^1 {\wedge}a^{2m{-}1},\; a^1{\wedge}b^{2m{-}1}-(m{-}1)b^1{\wedge}a^{2m{-}1}{+}\sum_{i{=}2}^{m{-}1} ({-}1)^i(m{-}i)a^i {\wedge} a^{2m{-}i}.
$$

The group of graded automorphisms $Aut_{gr}({\mathfrak m}_0^{r_1,\dots,r_k}(2m{-}1))$ acts on one-dimensional projectivization ${\mathbb P}H^2_{(2m)}\left({\mathfrak m}_0^{r_1,\dots,r_k}(2m{-}1)\right)$ operators with lower triangular matrices, and we again have only two orbits corresponding to the Lie algebras ${\mathfrak m}_0^{r_1,\dots,r_k}(2m)$ and ${\mathfrak m}_{0,2}^{r_1,\dots,r_{k{-}1}}(2m)$. The two-dimensional extension $ {\mathfrak g}$ is excluded from consideration, because it will have two successive two-dimensional homogeneous components ${\mathfrak g}_{2m-1}$ and ${\mathfrak g}_{2m}$.

How new Carnot algebras are generated from the algebra ${\mathfrak m} _0 ^ {r_1, \dots, r_k} (n)$, depending on the parity of $n$, is shown in Fig. \ref{m_0^R}.
\end{proof}
\begin{figure}[tbb]
\begin{picture}(130,50)(-20,15)

\put(10, 57){${\mathfrak m}_0^{r_1,\dots,r_k}(2m)$}
\put(-5, 23){${\mathfrak m}_0^{r_1,\dots,r_k}(2m{+}1)$}
\put(-5, 18){${\mathfrak m}_1^{r_1,\dots,r_k}(2m{+}1)$}
\put(57, 58){${\mathfrak m}_0^{r_1,\dots,2m{-}1}(2m{-}1)$}
\put(63, 18){${\mathfrak m}_0^{r_1,\dots,2m{-}1}(2m)$}
\put(63, 23){${\mathfrak m}_{0,2}^{r_1,\dots,2m{-}1}(2m)$}
\put(25, 20){${\mathfrak m}_0^{r_1,\dots,r_k,2m{+}1}(2m{+}1)$}
%\put(82, 21){${\mathfrak g}{=}\oplus_{i{=}1}^{2m{+}1}{\mathfrak g}_i$}

%\put(79,60){\vector(1,-1){22}}
%\put(77,60){\vector(1,-1){22}}

\put(31,43){\vector(1,-1){5}}
\put(33,43){\vector(1,-1){5}}

\put(10,43){\vector(-1,-1){5}}

\put(72,43){\vector(0,-1){5}}

\multiput(58,45)(5,0){2}{\line(0,1){10}}
  \multiput(58,45)(0,5){3}{\line(1,0){5}}
  \multiput(63,47.5)(0,5){2}{\line(1,0){5}}
\put(68,47.5){\line(0,1){5}}
\put(69.5,48.5){$\dots$}
\put(75,47.5){\line(0,1){5}}
\multiput(75,47.5)(0,5){2}{\line(1,0){5}}
\multiput(80,45)(5,0){2}{\line(0,1){10}}
  \multiput(80,45)(0,5){3}{\line(1,0){5}}

 \multiput(8,45)(5,0){2}{\line(0,1){10}}
  \multiput(8,45)(0,5){3}{\line(1,0){5}}
  \multiput(18,47.5)(5,0){1}{\line(0,1){5}}
  \multiput(13,47.5)(0,5){2}{\line(1,0){5}}
\multiput(25,47.5)(5,0){1}{\line(0,1){5}}
\put(19.5, 49){$\dots$}
\multiput(30,47.5)(5,0){1}{\line(0,1){5}}
  \multiput(25,47.5)(0,5){2}{\line(1,0){5}}

 \multiput(-7,27)(5,0){2}{\line(0,1){10}}
  \multiput(-7,27)(0,5){3}{\line(1,0){5}}
  \multiput(3,29.5)(5,0){1}{\line(0,1){5}}
  \multiput(-2,29.5)(0,5){2}{\line(1,0){5}}
\multiput(10,29.5)(5,0){1}{\line(0,1){5}}
\put(5,31){$\dots$}
\multiput(15,29.5)(5,0){1}{\line(0,1){5}}
  \multiput(10,29.5)(0,5){2}{\line(1,0){5}}
\multiput(20,29.5)(5,0){1}{\line(0,1){5}}
  \multiput(15,29.5)(0,5){2}{\line(1,0){5}}

\multiput(26,27)(5,0){2}{\line(0,1){10}}
  \multiput(26,27)(0,5){3}{\line(1,0){5}}
  \multiput(36,29.5)(5,0){1}{\line(0,1){5}}
  \multiput(31,29.5)(0,5){2}{\line(1,0){5}}
\multiput(43,29.5)(5,0){1}{\line(0,1){5}}
\put(37.5,31){$\dots$}
  \multiput(43,29.5)(0,5){2}{\line(1,0){5}}
\multiput(48,27)(5,0){2}{\line(0,1){10}}
  \multiput(48,27)(0,5){3}{\line(1,0){5}}

\multiput(58,27)(5,0){2}{\line(0,1){10}}
  \multiput(58,27)(0,5){3}{\line(1,0){5}}
  \multiput(63,29.5)(0,5){2}{\line(1,0){5}}
\put(68,29.5){\line(0,1){5}}
\put(69.5,30.5){$\dots$}
\put(75,29.5){\line(0,1){5}}
 \multiput(75,29.5)(0,5){2}{\line(1,0){5}}

\multiput(80,27)(5,0){2}{\line(0,1){10}}
  \multiput(80,27)(0,5){3}{\line(1,0){5}}
\multiput(85,29.5)(0,5){2}{\line(1,0){5}}
\multiput(90,29.5)(0,5){1}{\line(0,1){5}}

%\multiput(78,25)(5,0){2}{\line(0,1){10}}
  %\multiput(78,25)(0,5){3}{\line(1,0){5}}
%\multiput(83,27.5)(0,5){2}{\line(1,0){5}}
%\put(88,27.5){\line(0,1){5}}
%\put(95,27.5){\line(0,1){5}}
%\put(89.5,29){$\dots$}

%  \multiput(95,27.5)(0,5){2}{\line(1,0){5}}
%\multiput(100,25)(5,0){2}{\line(0,1){10}}
  %\multiput(100,25)(0,5){3}{\line(1,0){5}}
%\multiput(105,25)(5,0){2}{\line(0,1){10}}
  %\multiput(105,25)(0,5){3}{\line(1,0){5}}
\end{picture}
\caption{Central extensions ${\mathfrak m}_0^{r_1,\dots,r_k}(n)$.}
\label{m_0^R}
\end{figure}
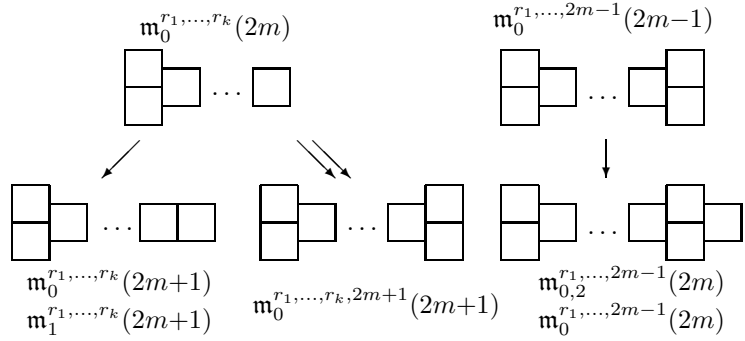
\begin{propos}
Алгебра Карно ${\mathfrak m}_1^{r_1,\dots,r_k}(2m{-}1)$ не продолжается до алгебры Карно длины $2m$.
\label{prop_m_1}
\end{propos}
\begin{proof} This follows from the triviality of the cohomology
$$H^2_{(2m)}({\mathfrak m}_1^{r_1,\dots,r_k}(2m{-}1))=0.$$
In fact, an arbitrary homogeneous $2$-form $\omega$ of grading $2m$ can be written in the form
$
\omega=\omega_1+\omega_2+\omega_3,
$
where the first term $\omega_1$ has the form
$$
\omega_1= (x_1 a^1 + y_1 b^1) \wedge b^{2m{-}1}+\sum_{k{=}2}x_k a^k{\wedge} a^{2m{-}k},
$$
The second and third terms $\omega_2$ and $\omega_3$ are equal, respectively
$$
\omega_2=\sum_{k{+}r_j{=}2m} y_{k,r_j} a^{k}{\wedge}b^{r_j}, \; \omega_3=\sum_{r_i{+}r_j{=}2m,\; r_i{<} r_j} z_{r_i,r_j} b^{r_i}{\wedge}b^{r_j}.
$$
Consider the differential  $d\omega_3$
$$
d\omega_3={-}\sum_{r_i{+}r_j{=}2m,\;r_i{<} r_j} z_{r_i,r_j} b^1 {\wedge}\left(a^{r_i{-}1}{\wedge}b^{r_j}+ b^{r_i}{\wedge}a^{r_j{-}1}\right)+\dots
$$
In the formula for $d\omega_3$ we put the dots instead of a linear combination of monomials of the form $a^{i_1}{\wedge} a^{i_2}{\wedge} b^{r_s}$. It is obvious that the differentials
$d\omega_1$ and $d\omega_2$ do not contain terms of the form $\gamma a^l{\wedge}b^1{\wedge}b^{r_s}$.
Taking the maximum value $r_j$ such that $\exists r_i, r_i{+}r_j=2m$ we see that the corresponding coefficient  $z_{r_i,r_j}=0$. Continuing this inductive procedure, we see that all the coefficients $z_{r_i, r_j}$ must be zero. In the same way, we can prove that all the coefficients $y_{k, r_j}$ are trivial. Hence  $\omega_2=\omega_3=0$. The triviality of $\omega_1$ now follows from the proof of Proposition \ref{cohom_m_1}.
\end{proof}
\begin{propos}
\label{prop_m_0,2}
Let $m{\ge 4}$, then ${\mathfrak m}_{0,2}^{r_1, \dots, r_k}(2m)$ has a unique Carnot extension. We denote it as
${\mathfrak m}_{0,3}^{r_1,\dots,r_k}(2m{+}1)$.
\label{prop_m_0_3}
\end{propos}
\begin{proof}
An arbitrary homogeneous $2$-form $\omega$ of grading $2m{+}1$ can be represented as a sum $\omega=\omega_1+\omega_2$, where
$$
\begin{array}{c}
\omega_1= (x_1 a^1 {+} y_1 b^1) {\wedge} a^{2m}+a^{2} {\wedge}(x_2 a^{2m{-}1} {+} y_2 b^{2m{-}1})+ \sum_{k{=}3}^m x_k a^k{\wedge} a^{2m{-}k{+}1},\\
\omega_2=\sum_{j=1}^k y_{j} a^{2m{+}1{-}r_j}{\wedge}b^{r_j}. 
\end{array}
$$
We calculate the differentials for both terms $\omega_1$ and $\omega_2$
$$
\begin{array}{c}
d\omega_1=a^1{\wedge}b^1{\wedge}\left((x_1(m{-}1){+}x_2)a^{2m{-}1}{+}(y_1{+}y_2)b^{2m{-}1} \right){+}\\
{+}\left( x_2a^1{-} y_2b^1\right){\wedge}a^2{\wedge}a^{2m{-}2}{+}\sum_{l{=}2}^{m{-}1}({-}1)^{l{+}1}(m{-}l)(x_1a^1{+}y_1b^1){\wedge}a^l{\wedge}a^{2m{-}l}{+}\\
{+}y_2\sum_{s=2}^{m{-}1}({-}1)^sa^2{\wedge}a^s{\wedge}a^{2m{-}1{-}s}
{+}\sum_{k{=}3}^m x_k a^1{\wedge}\left( a^{k{-}1}{\wedge} a^{2m{-}k{+}1}{+}a^k{\wedge} a^{2m{-}k}\right).
\end{array}
$$
$$
d\omega_2=\sum_{j=1}^k y_{j}a^1{\wedge} a^{2m{-}r_j}{\wedge}b^{r_j}{+}\sum_{j=1}^k y_{j} a^{2m{+}1{-}r_j}{\wedge} \left(b^1{\wedge}a^{r_j{-}1}{-}\sum_{l=2}^{\frac{1}{2}(r_j{-}1)}({-}1)^la^l{\wedge}a^{r_j{-}l}\right),
$$
There are no monomials of the form $\gamma_j a^{1}{\wedge} a^{2m{-}r_j}{\wedge} b^{r_j}$ in the decomposition $d\omega_1$. Thus, from the closedness of the form
$\omega$ follows $\omega_2=0$.

If $m\ge 4$, then the coefficient facing $a^2{\wedge}a^{m{-}1}{\wedge}{2m{-}1}$ in the decomposition of $d\omega_1$ equals $({-}1)^{m{-}1}y_2$. It follows from $d\omega_1=0$ that $y_2=0$. You can also calculate the coefficient in front of the monomial $a^1{\wedge} b^1 {\wedge}b^{2m{-}1}$. Он равен $y_1+y_2$. Hence $y_1=y_2=0$.

Comparing the other coefficients in the expansion of $ d\omega_1$, we obtain the following system of linear equations on the coefficients $x_1,x_2,\dots, x_m$
$$
\left\{
\begin{array}{c}
x_1(m{-}1)+x_2=0,\\
({-}1)^{i{+}1}(m{-}i)x_1+x_i+x_{i{+}1}=0, \; i=2,\dots, m{-}1,\\
({-}1)^mx_1+x_m=0
\end{array}
\right.
$$We solve this system and conclude that when
  $2m\ge 8$ the subspace of interest to us $H^2_{(2m{+}1)}\left({\mathfrak m}_{0,2}^{r_1,\dots,r_k}(2m) \right)$ has dimension one.  As its basic cocycle, we choose a
$$
a^1{\wedge} a^{2m}{+}\frac{1}{2}\sum_{i{=}2}^{m} ({-}1)^{i{+}1}(i{-}1)\left(2m{-}i\right)a^i {\wedge} a^{2m{+}1{-}i}.
$$
The corresponding central extension is isomorphic to
 ${\mathfrak m}_{0,3}^{r_1,\dots,r_k}(2m{+}1)$.

If $m=3$ the the subspace $H^2_{(7)}\left({\mathfrak m}_{0,2}^{r_1}(6) \right)$ is the two-dimmensional span of the following cocycles
$$
a^1{\wedge} a^{6}{-}2a^2 {\wedge} a^5{+}
 3a^3 {\wedge} a^4, \; b^1{\wedge}a^6{-}a^2{\wedge}b^5.
$$ 
\end{proof}
\begin{definition}
The Lie algebra ${\mathfrak m}_{0,3}^{r_1,\dots,r_k}(2m{+}1)$, 
where $3 \le r_1<\dots<r_k \le 2m{-}3$ denotes an ordered set of odd natural numbers  if $m \ge 3$,
can be given by its basis
$
a_1,{\dots},a_{2m}, a_{2m{+}1}, \; b_1,
b_{r_1},\dots, b_{r_k}, b_{2m{-}1},
$
and commutation relations.
\begin{equation}
\label{m_0_3_S}
\begin{split}
[a_1, b_1]=a_2, [a_1, a_i]=a_{i{+}1}, \quad i{=}2, \dots, 2m{-}2, 2m,\\
[b_1, a_{{2m}{-}2}]{=}{-}b_{2m{-}1},
[a_q,a_{2m{-}1{-}q}]=({-}1)^{q}b_{2m{-}1}, q{=}2,{\dots}, m{-}1;\\
[a_1,b_{2m{-}1}]{=}a_{2m}, [b_1, a_{{2m}{-}1}]{=}{-}(m{-}1)a_{2m},\\
[a_q,a_{2m{-}q}]=({-}1)^{q}(m{-}q)a_{2m}, \; q{=}2,
\dots,m{-}1;\\
[a_p,a_{2m{-}p{+}1}]=
({-}1)^{p{+}1}(p{-}1)\left(m{-}\frac{p}{2}\right)a_{2m{+}1},
p{=}2,\dots,m,\\   
[b_1, a_{{r_j}{-}1}]{=}{-}b_{r_j} \; [a_q,a_{r_j{-}q}]=({-}1)^{q}b_{r_j}, j{=}1,{\dots},k; q{=}2,{\dots}, \frac{r_j{-}1}{2}.
\end{split}
\end{equation}
\end{definition}
\begin{propos}
\label{prop_m_0,3}
The Lie algebra ${\mathfrak m}_{0,3}^{r_1,\dots,r_{k}}(2m{+}1)$ does not admit a Carnot extension.
\end{propos}
\begin{proof} The proof follows from the triviality of 
the corresponding subspace of two-dimensional cohomology
$H^2_{(2m{+}2)}({\mathfrak m}_{0,3}^{r_1,\dots,r_{k}}(2m{+}1))=0$. $H^2_ {(2m {+} 2)}({\mathfrak m}_{0,3}^{r_1,\dots,r_{k}}(2m{+}1))=0$. The cohomological calculations themselves completely repeat similar arguments from the proofs of the Propositions
\ref{prop_m_1} и \ref{prop_m_0_3}.
\end{proof} 

\section{Finite-dimensional factor algebras of Lie algebras  ${\mathfrak n}_1$ and ${\mathfrak n}_2$}\label{s6}

The Lie algebras cohomology of ${\mathfrak n}_1$ and ${\mathfrak n}_2$ are well-known \cite{Garl, LepM, Lep, Kumar}.

In particular, both spaces of second cohomology  $H^2({\mathfrak n}_1)$ and $H^2({\mathfrak n}_2)$ are two-dimensional, more precisely:
\begin{itemize}
\item
$H^2({\mathfrak n}_2)$ is the linear span of the cohomology classes
$$
f^2{\wedge}f^3=b^1{\wedge}a^2, f^1{\wedge}f^6=a^1{\wedge}a^5.
$$
with natural gradings $3$ and $6$ respectively;
\item
$H^2({\mathfrak n}_1)$ is a linear span of two cohomology classes
$$
e^1{\wedge}e^4=a^1{\wedge}a^3, \; e^2{\wedge}e^5=b^1{\wedge}b^3,
$$ 
with a common natural grading $4$.
\end{itemize}

Define ${\mathfrak n}_2^3$ as an one-dimensional central extension of the Lie algebra ${\mathfrak n}_2$, which corresponds to a cocycle $f^2{\wedge}f^3=b^1{\wedge}a^2$.

To obtain a basis and structural relations of the Lie algebra ${\mathfrak n}_2^3$, you simply need to add a new central element $z$ to a basis of ${\mathfrak n}_2$, satisfying the additional commutation relations
$$
[f_2,f_3]=z, [z,f_i]=0,  i \in {\mathbb N}.
$$

We can consider other central extensions of Lie algebras   ${\mathfrak n}_1$ and ${\mathfrak n}_2$, but if we consider a central extension ${\mathfrak n}_2^6$ of the Lie algebra ${\mathfrak n}_2$, which corresponds to a cocycle $f^1{\wedge}f^6$, then it turns out that the sum of the dimensions of two neighboring homogeneous components is greater than three
$$
\dim{({\mathfrak n}_2^6)_{(5)}}+\dim{({\mathfrak n}_2^6)_{(6)}}=4.
$$
For the same reasons, we do not consider central extensions ${\mathfrak n}_1$.

\begin{lemma}
\label{n_1_ext}
Let $n \ge 4$. We consider an $n$-dimensional Lie quotient algebra ${\mathfrak n}_1^{\pm}(n)={\mathfrak n}_1^{\pm}/({\mathfrak n}_1^{\pm})^{n{+}1}$. It is a Carnot algebra of length $n$.  The list of its nonisomorphic Carnot extensions of height $n{+}1$ consists of two Carnot algebras
$$ 
{\mathfrak n}_1^{\pm}(n{+}1)={\mathfrak n}_1^{\pm}/({\mathfrak n}_1^{\pm})^{n{+}2},\quad  {\mathfrak n}_{1,1}^{\pm}(2m{+}1).
$$
Moreover the Carnot algebras  ${\mathfrak n}_{1,1}^{\pm}(2m{+}1)$ do not have extensions of length $2m+2$ for all $m\ge2$.
\end{lemma} 
\begin{proof}
The subspace $H^2_{(n{+}1)}({\mathfrak n}_1^{\pm}(n))$ is two-dimensional for even $n=2m$ and is one-dimensional if $n=2m+1$ is odd.  For $n=2m$ the subspace 
$H^2_{(2m{+}1)}({\mathfrak n}_1^{\pm}(2m))$ spanned by the differentials $du^{2m{+}1}=\Omega_1(2m{+}1)$ and $dv^{2m{+}1}=\Omega_2(2m{+}1)$ that do not exist after taking the quotient 
in $\Lambda^1_{(2m+1)}({\mathfrak n}_{1,1}^{\pm}(2m))$ of one-forms $u^{2m{+}1}$ and  $v^{2m{+}1}$
$$
\Omega_1(2m{+}1)=\sum_{l{+}k{=}m}v^{2l{+}1}{\wedge}w^{2k}, \Omega_2(2m{+}1)= \sum_{q{+}k{=}m}w^{2k}{\wedge}u^{2q{+}1}.
$$
Suppose that there exist cocycles of grading $2m+1$ that are linearly independent with $\Omega_1(2m{+}1)$ and $\Omega_2(2m{+}1)$ in the subspace $H^2_{(2m{+}1)}({\mathfrak n}_1^{\pm}(2m))$. Then we will have nontrivial cohomology classes in the subspaces 
 $H^2_{(n{+}1)}({\mathfrak n}_1)$  for $n\ge4$ now for the infinite-dimensional Lie algebra ${\mathfrak n}_1$, which contradicts the two-dimensionality of the second cohomology $H^2({\mathfrak n}_1)$ of the Lie algebra ${\mathfrak n}_1$. Recall that $H^2({\mathfrak n}_1^{+})=H^2({\mathfrak n}_1^{-})$, because ${\mathfrak n}_1^{+}$ and ${\mathfrak n}_1^{-}$ are isomorphic over ${\mathbb C}$. 

The group of graded automorphisms $Aut_{gr}({\mathfrak n}_1(2m))$ act on the subspace  $H^2_{(2m{+}1)}({\mathfrak n}_1^{\pm}(2m))$ according to the formulas (\ref{n_1+_automorphism_b})
$$
\begin{array}{c}
\varphi(\Omega_1(2m{+}1))=c^{2k{+}1}( \cos{t} \Omega_1(2m{+}1)+ \sin{t} \Omega_2(2m{+}1)),\\
\varphi(\Omega_2(2m{+}1))= \pm c^{2k{+}1}({-}\sin{t} \Omega_1(2m{+}1)+\cos{t} \Omega_2(2m{+}1)).
\end{array}
$$
It induces a transitive action on projectivization ${\mathbb P}H^2_{(2m{+}1)}({\mathfrak n}_1^{+}(2m))$, in which, in turn,
there is only one orbit represented by a point $(1:0)$.  This point of the projective line corresponds to a unique (up to isomorphism) nontrivial one-dimensional central extension, which we denote by ${\mathfrak n}_{1,1}^{+}(2m{+}1)$. The case of an odd $n=2m+1$ is completely analogous.

Two-dimensional central extension of the Lie algebra Ли ${\mathfrak n}_1^{\pm}(2m)$ which is isomorphic to ${\mathfrak n}_1^{\pm}(2m+1)$, is unique up to isomorphism, just as the one-dimensional central extension of the Lie algebra ${\mathfrak n}_1^{\pm}(2m{+}1)$, which is isomorphic to the Lie algebra ${\mathfrak n}_1^{\pm}(2m{+}2)$. The last extension corresponds to a cocycle
$$
\Omega(2m{+}2)=\sum_{l{+}k{=}m}u^{2l{+}1}{\wedge}v^{2k{+}1}.
$$
\end{proof}

\begin{lemma}
\label{n_2_ext}
Let $n \ge 7$. Consider the quotient Lie algebra  ${\mathfrak n}_2(n)={\mathfrak n}_2/{\mathfrak n}_2^{n{+}1}$. It is a Carnot algebra of length $ n $. There are two possible cases:

1) if $n\ne 6m, n\ne 6m+4$, then ${\mathfrak n}_2(n)$ uniquely extends to a Carnot algebra of length $n+1$,  its extension will be ${\mathfrak n}_2(n+1)={\mathfrak n}_2/{\mathfrak n}_2^{n{+}2}$.  

2) if  $n=6m$ or  $n=6m+4$, then there will be several extensions. Let's present their list.
$$
{\mathfrak n}_2(n{+}1)={\mathfrak n}_2/{\mathfrak n}_2^{n{+}2}, \;
{\mathfrak n}_{2,1}(n{+}1), {\mathfrak n}_{2,2}(n{+}1), {\mathfrak n}_{2,3}(n{+}1).
$$
Moreover, the Carnot extensions of the form ${\mathfrak n}_{2,s}(n{+}1), s=1,2,3,$ for $n=6m$  or $n=6m+4$ themselves can not be extended to a Carnot algebra of greater length.
\end{lemma} 
\begin{proof}
Let $n \ge 7$. If $n=6m$ or $n=6m{+}4$ then the subspace $H^2_{(n{+}1)}({\mathfrak n}_2(n))$ is two-dimensional, in all other cases  $H^2_{(n{+}1)}({\mathfrak n}_2(n))$ is one-dimensional. The proof repeats the arguments of the previous sentence. Indeed,
cochain subcomplexes of grading $n+1$ of two Lie algebras ${\mathfrak n}_2$ and ${\mathfrak n}_2(n)$ almost completely coincide.  The only difference is that the subspace of one-form $\Lambda^1_{(n+1)}({\mathfrak n}_2)(n)$ is trivial and $\Lambda^1_{(n+1)}({\mathfrak n}_2) \ne 0$. In this siruation two-cohomology $H^2_{(n+1)}({\mathfrak n}_2)$  and $H^2_{(n+1)}({\mathfrak n}_2(n))$ are related in a very simple way
$$
H^2_{(n+1)}({\mathfrak n}_2(n))=H^2_{(n+1)}({\mathfrak n}_2)\oplus d\Lambda^1_{(n+1)}({\mathfrak n}_2).
$$
The cohomology $H^2_{(n+1)}({\mathfrak n}_2)$, as we have already remarked above, they are known from classical papers \cite{KacK, GLep, LepM, Lep}.
For $n=6m, 6m+4$ the subspace $d\Lambda^1_{(n+1)}({\mathfrak n}_2)$ is two-dimensional and it is one-dimensional in alявляется двумерным, and it is one-dimensional in all other cases. We denote the basis cocycles from $d\Lambda^1_{(6m{+}r{+}1)}({\mathfrak n}_2(6m{+}r))$ as  $\Omega_1(6m{+}r{+}1))$ and $\Omega_1(6m{+}r{+}1))$ respectively. They are equal $df^{8m+1}$ and $df^{8m+2}$ for $r=0$, and also  $df^{8m+6}$ and $df^{8m+7}$ if $r=4$.

We recall that the group of graded automorphisms $Aut_{gr}({\mathfrak n}_2(n))$ according to (\ref{n_2_automorphism_2})
is the two-dimensional algebraic torus ${\mathbb K}^*\times {\mathbb K}^*$ of order two diagonal matrices which acts on the basic cocycles $\Omega_1(6m{+}r{+}1)$ and $\Omega_2(6m{+}r{+}1)$ of the subspace $H^2_{(6m{+}r{+}1)}({\mathfrak n}_2(6m{+}r)), r=0,4,$  diagonally according to the following formulas $\mu\alpha \ne 0$:
$$
\begin{array}{c}
\varphi(\Omega_1(6m{+}1))=\alpha^{4m{+}1}\mu^{2m} \Omega_1(6m{+}1),
\varphi(\Omega_2(6m{+}1))=\alpha^{4m}\mu^{2m{+}1}\Omega_2(6m{+}1);\\
\varphi(\Omega_1(6m{+}5))=\alpha^{4m{+}4}\mu^{2m{+}1} \Omega_1(6m{+}5),
\varphi(\Omega_2(6m{+}5))=\alpha^{4m{+}3}\mu^{2m{+}2}\Omega_2(6m{+}5);
\end{array}
$$
The projectivization of this action has three orbits, given by their representatives $ (1: 0), (0: 1), (1: 1) $. To these three points of the projective line there correspond three Carnot algebras ${\mathfrak n}_{2,s}(6m{+}r{+}1), s=1,2,3,$
respectively $(r=0,4)$ (see the Definition \ref{n_2_quotients}).
\end{proof}
\begin{corollary}
Let $n \ge 7$. Consider the quotient Lie algebra ${\mathfrak n}_2^3(n)={\mathfrak n}_2^3/({\mathfrak n}_2^3)^{n{+}1}$. It is a Carnot algebra of length $n$. There are two possible cases:

1) if $n\ne 6m, n\ne 6m+4$ then ${\mathfrak n}_2^3(n)$ Uniquely extends to a Carnot algebra of length $n+1$, its extension is ${\mathfrak n}_2^3(n{+}1)={\mathfrak n}_2^3/({\mathfrak n}_2^3)^{n{+}2}$.  

2) if  $n=6m$ or  $n=6m+4$, then there are four admissible extensions up to isomorphism. Let's list them.
$$
{\mathfrak n}_2(n{+}1)={\mathfrak n}_2^3/({\mathfrak n}_2^3)^{n{+}2}, \;
{\mathfrak n}_{2,1}^3(n{+}1), {\mathfrak n}_{2,2}^3(n{+}1), {\mathfrak n}_{2,3}^3(n{+}1).
$$
Moreover, the Carnot algebras ${\mathfrak n}_{2,s}^3(n{+}1), s=1,2,3,$ for $n=6m$ or $n=6m+4$ do not have Carnot extensions.
\end{corollary} 
\section{The main theorem}\label{s7}
\begin{theorem}
\label{osnovn}
Let $\mathfrak{g}=\bigoplus_{i=1}^n\mathfrak{g}_i$ be a real positively graded Lie algebra for  $n \ge 2$ which satisfies the properties
\begin{equation}
\label{3/2}
[{\mathfrak g}_1, {\mathfrak g}_i]={\mathfrak g}_{i+1},\;
\dim{\mathfrak{g}_i}+\dim{\mathfrak{g}_{i{+}1}}\le 3, \; i=1,\dots,n{-}1, 
{\mathfrak{g}_n}\ne 0.
\end{equation}
Then $\mathfrak{g}$ is a Carnot algebra and it is isomorphic to one and only one Lie algebra from the following list:

A. Filiform Lie algebras and their central extensions:
$$
\begin{array}{c}
\mathfrak{m}_{0}^{S_n}(n), n \ge 2, \quad\quad \quad\quad
\mathfrak{m}_{1}^{S_{n-2}}(n),  n=2m{-}1 \ge 5,\\  
\mathfrak{m}_{0,2}^{S_{n-3}}(n), n=2m \ge 8, \quad
\mathfrak{m}_{0,3}^{S_{n-4}}(n), n=2m{+}1 \ge 9, 
\end{array}
$$
where  $S_p$ denotes an ordered collection (possibly empty) of odd natural numbers 
$$
\begin{array}{c}
S_p= \{ r_1,\dots,r_k | r_q=2m_q{+}1, 1\le q \le k, 3 \le r_1 <\dots < r_k \le p \},\\
S_p =  \varnothing, \{3\}, \{5\},\dots, \{2m{-}1\},\{3,5\}\dots, \{ 3,5,7,\dots,2m{-}1\}, \; p=2m, 2m-1.
\end{array}
$$ 

B. Finite-dimensional quotient-Lie algebras for $\mathfrak{n}_1^{\pm}$:
$$ 
\mathfrak{n}_1^{\pm}(n), n \ge 4; \quad 
\mathfrak{n}_{1,1}^{\pm}(n), n=2m{+}1 \ge 5;
%\item
%$\mathfrak{n}_{1,2}^{\pm}(n), n{=}2m{+}2 \ge 6$; 
$$

C. Finite-dimensional quotient-Lie algebras for  $\mathfrak{n}_2$:
$$
\mathfrak{n}_2(n), n \ge 7, \quad 
\mathfrak{n}_{2,s}(n),  s=1,2,3,  \; n=2m{+}1, 2m+5, m\ge 3,
$$

D. Finite-dimensional quotient-Lie algebras for  $\mathfrak{n}_2^3$:
$$
\mathfrak{n}_2^3(n), n \ge 7; \quad
\mathfrak{n}_{2,s}^3(n),  s=1,2,3,  \; n=2m+1, 2m+5, m\ge 3;
$$
\end{theorem}
\begin{remark}
All the Lie algebras appearing in the statement of the theorem were defined earlier in the text. For the convenience of the reader, we present tables
with their defining bases and structural constants in Appendix 1.
\end{remark}
\begin{proof} We will prove the theorem by induction on the length of Carnot algebras. We recall that the length of the Carnot algebra ${\mathfrak g}$ coincides with its index of nilpotency $s({\mathfrak g})$.

The first stage of our proof: the classification of all Carnot algebras ${\mathfrak g}$ of length $n \le 7$ satisfying the condition (\ref{3/2}).

We start with the two-dimensional Abelian Lie algebra ${\mathfrak m}_0(1)$ given as a linear span $\langle a_1, b_1 \rangle$. The space of its second cohomology   $H^2({\mathfrak m}_0(1))$  is one-dimensional and its basic cocycle $a^1 \wedge b^1$. Thus, up to isomorphism, this Lie algebra has a unique central extension ${\mathfrak m}_0(2)$, which is a Carnot algebra of length two. Extension ${\mathfrak m}_0(2)$  is isomorphic to the classical Heisenberg algebra ${\mathfrak h}_3$. In turn, the Lie algebra ${\mathfrak m}_0(2)$ has two nonisomorphic central extensions: ${\mathfrak m}_0(3)$ and a free nilpotent Lie algebra ${\mathcal L}(2,3)\cong {\mathfrak m}^3_0(3)$ of degree three. They will compile a list of Carnot algebras of length three. 
$$
{\mathfrak m}_0(3), {\mathcal L}(2,3)\cong {\mathfrak m}^3_0(3).
$$
The following lemma is very important, for it is in it that the difference between the real and complex cases in our classification will manifest itself.
\begin{lemma}
\label{vazhnaya_lemma}
For a real free nilpotent Lie algebra ${\mathcal L}(2,3)$ of degree of nilpotency $3$ there are three nonisomorphic one-dimensional central Carnot extensions:
$$
{\mathfrak m}_0^3(4), {\mathfrak n}_1^+(4), {\mathfrak n}_1^-(4).
$$
If the ground field is ${\mathbb C}$, then there is an isomorphism ${\mathfrak n}_1^+(4)\cong {\mathfrak n}_1^-(4)$.
\end{lemma}
\begin{proof}
The cochain complex of the Lie algebra ${\mathcal L}(2,3)$ is given by
\begin{equation}
\begin{split}
da^1=db^1=0, \;\;  da^2=a^1\wedge b^1, \\
 da^3=a^1 \wedge a^2, \;\;  db^3=b^1 \wedge a^2;
\end{split}
\end{equation}
We fix the following basis of a three-dimensional space $H^2_{(4)}({\mathcal L}(2,3),{\mathbb R})$:
$$
a^1\wedge a^3, \; a^1\wedge b^3+b^1\wedge a^3, \; b^1\wedge b^3 .
$$
The group of graded automorphisms ${\rm Aut}_{gr}({\mathcal L}(2,3))$ is isomorphic to $GL(2,{\mathbb R})$  (see Example \ref{Aut_L_2_3}) and acts in three-dimensional space $H^2_{(4)}({\mathcal L}(2,3),{\mathbb R})$ in the following way
$$
A=
\begin{pmatrix}
\alpha & \beta\\
\rho & \mu
\end{pmatrix} \in GL(2,{\mathbb R}) \to \det{A}\cdot 
\begin{pmatrix}
\alpha^2 & 2\rho\alpha &\rho^2\\
\alpha\beta & \rho\beta{+}\alpha \mu& \mu\rho\\
\beta^2 & 2\mu\beta& \mu^2
\end{pmatrix}.
$$
We now consider the action of the group $GL(2,{\mathbb R})$ on one-dimensional subspaces of the space $H^2_{(4)}({\mathcal L}(2,3),{\mathbb R})$, i.e. on the projective plane ${\mathbb R}P^2$. The orbit of the point $ (0:0:1)$ is an oval $\{(\rho^2{:}\mu\rho{:}\mu^2)\}\subset {\mathbb R}P^2$, which is represented by the parabola $y^2=x$ in the standard affine map $x_3 \ne 0$ with the coordinates  $x=\frac{x_1}{x_3}, y=\frac{x_2}{x_3}$. 

The action of the group $GL(2,{\mathbb R})$ on the projective plane ${\mathbb R}P^2$ there are two more orbits, they coincide with the inner and outer regions of the parabola $y^2=x$ on an affine chart $ x_3 \ne 0$. These two orbits can be represented by the points
$({-}1:0:1)$ and $(1:0:1)$. Note that if the ground field is complex, then the points $({-}1:0:1)$ and $(1:0:1)$ are in the same orbit of action
$$
({-}1:0:1) \in \{ (\alpha^2+\rho^2, \alpha\beta+\mu\rho, \beta^2+\mu^2)\}, \;  \mu=1, \beta=\rho=0, \alpha^2= {-}1.
$$

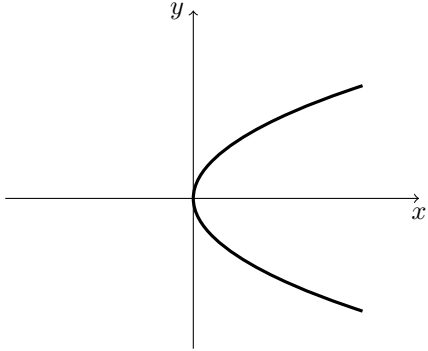
\begin{tikzpicture}
\draw[->] (-2.5,0)--(3,0) node[anchor=north] {$x$}; 
\draw[->] (0,-2)--(0,2.5) node[anchor=east] {$y$}; 
%\draw   (0,0) parabola 
\draw [rotate=-90][very thick] (-1.5,2.25) parabola bend (0,0)(1.5,2.25);
%\shade[top color=gray,bottom color=white!50] 
       [rotate=-90] (-1.5,2.25) parabola bend (0,0)(1.5,2.25);
\put(16,-0.9){$\bullet$};
\put(-17.5,-0.9){$\bullet$};
\put(-0.9,-0.9){$\bullet$};
\put(12.5,2){$(1{:}0{:}1)$};
\put(-10,-6){$(0{:}0{:}1)$};
\put(-22,2){$({-}1{:}0{:}1)$};
\put(14,20){$x{=}y^2$};
\put(30,10){$\{(\rho^2{:}\mu\rho{:}\mu^2)\} \subset {\mathbb R}P^2$};
\put(45,5){$\mu \neq 0$};
%\put(15,30){${\rm Aut}_{gr}({\mathcal L}(2,3)){=}GL(2,{\mathbb R})$} 
\end{tikzpicture}

The point $(0:0:1) \in {\mathbb R}P^2$ corresponds to the cocycle $b^1\wedge b^3$,which defines a one-dimensional central extension ${\mathfrak m}_0^3(5)$. Orbits represented by the points $({\pm}1:0:1)$,  (i.e. cocycles
$\pm a^1\wedge a^3+b^1\wedge b^3$) correspond to exetensions ${\mathfrak n}_1^{\pm}(4)$.
\begin{remark}
The action of the group $GL(2,{\mathbb R})$ on the Grassmannian of two-dimensional subspaces in $H^2_{(4)}({\mathcal L}(2,3),{\mathbb R})$ also has three orbits in the real case and two orbits in the complex.
\end{remark}
\end{proof}

In turn, the Lie algebra $ {\ mathfrak m} _0 (3) $ has only one central extension of Carnot ${\mathfrak m}_0(4)$.  Hence, we have a list of Carnot algebras of length $4$ satisfying (\ref{3/2})
$$
{\mathfrak m}_0(4),  {\mathfrak m}_0^3(4), {\mathfrak n}_{1}^{\pm}(4).
$$
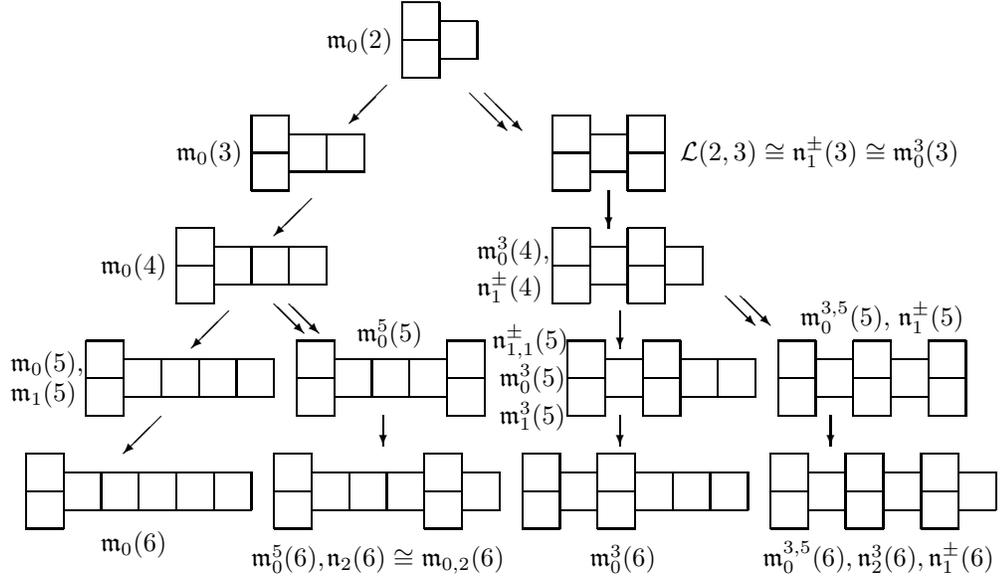
\begin{figure}[tb]
\begin{picture}(112,100)(-10,20)
\put(30, 89){${\mathfrak m}_0(2)$}
\put(10, 74){${\mathfrak m}_0(3)$}
\put(0, 59){${\mathfrak m}_0(4)$}
\put(-12, 46){${\mathfrak m}_0(5),$}
\put(-12, 42){${\mathfrak m}_1(5)$}
\put(0, 22){${\mathfrak m}_0(6)$}
\put(77, 74){${\mathcal L}(2,3)\cong {\mathfrak n}_1^{\pm}(3)\cong {\mathfrak m}_0^3(3)$}
\put(50, 61){${\mathfrak m}_0^3(4),$}
\put(50, 56){${\mathfrak n}_1^{\pm}(4)$}
\put(34, 50){${\mathfrak m}_0^5(5)$}
\put(52, 49){${\mathfrak n}_{1,1}^{\pm}(5)$}
\put(53, 44){${\mathfrak m}_0^3(5)$}
\put(53, 39){${\mathfrak m}_1^3(5)$}
\put(93, 52){${\mathfrak m}_0^{3,5}(5),$}
\put(106, 52){${\mathfrak n}_1^{\pm}(5)$}
\put(88, 20){${\mathfrak m}_0^{3,5}(6), {\mathfrak n}_{2}^{3}(6),$}
\put(110, 20){${\mathfrak n}_1^{\pm}(6)$}
%\put(70, 20){${\mathfrak n}_{1,2}^{\pm}(6)$}
\put(65, 20){${\mathfrak m}_{0}^3(6)$}
\put(20, 20){${\mathfrak m}_0^5(6),$}
\put(30, 20){${\mathfrak n}_2(6)\cong {\mathfrak m}_{0,2}(6)$}

\multiput(40,85)(5,0){2}{\line(0,1){10}}
  \multiput(40,85)(0,5){3}{\line(1,0){5}}
  \multiput(50,87.5)(5,0){1}{\line(0,1){5}}
  \multiput(45,87.5)(0,5){2}{\line(1,0){5}}
  \put(38,84){\vector(-1,-1){5}}
\put(49,83){\vector(1,-1){5}}
\put(51,83){\vector(1,-1){5}}

\put(28,69){\vector(-1,-1){5}}

\put(83,56){\vector(1,-1){4}}
\put(85,56){\vector(1,-1){4}}

\put(23,55){\vector(1,-1){4}}
\put(25,55){\vector(1,-1){4}}

\put(17,54){\vector(-1,-1){5}}

\put(8,40){\vector(-1,-1){5}}

\put(67.5,70){\vector(0,-1){5}}
\put(69,54){\vector(0,-1){5}}

%\put(-2,25){\vector(-1,-1){4}}
\put(69,40){\vector(0,-1){4}}
\put(97,40){\vector(0,-1){4}}
\put(37.5,40){\vector(0,-1){4}}
%\put(97.5,25){\vector(0,-1){5}}
%\put(110.5,25){\vector(1,-1){4}}
%\put(112.5,25){\vector(1,-1){4}}
%\put(20.5,25){\vector(1,-1){4}}
%\put(22.5,25){\vector(1,-1){4}}
%\put(58,38){\vector(-1,-3){6}}
%\put(59.5,38){\vector(-1,-3){6}}

%\put{32,54}{${\mathfrak m}_0(3)$}

\multiput(20,70)(5,0){2}{\line(0,1){10}}
  \multiput(20,70)(0,5){3}{\line(1,0){5}}
  \multiput(30,72.5)(5,0){1}{\line(0,1){5}}
  \multiput(25,72.5)(0,5){2}{\line(1,0){5}}
\multiput(35,72.5)(5,0){1}{\line(0,1){5}}
  \multiput(30,72.5)(0,5){2}{\line(1,0){5}}

\multiput(60,70)(5,0){2}{\line(0,1){10}}
  \multiput(60,70)(0,5){3}{\line(1,0){5}}
  \multiput(65,72.5)(0,5){2}{\line(1,0){5}}
\multiput(70,70)(5,0){2}{\line(0,1){10}}
  \multiput(70,70)(0,5){3}{\line(1,0){5}}

 \multiput(10,55)(5,0){2}{\line(0,1){10}}
  \multiput(10,55)(0,5){3}{\line(1,0){5}}
  \multiput(20,57.5)(5,0){1}{\line(0,1){5}}
  \multiput(15,57.5)(0,5){2}{\line(1,0){5}}
\multiput(25,57.5)(5,0){1}{\line(0,1){5}}
  \multiput(20,57.5)(0,5){2}{\line(1,0){5}}
\multiput(30,57.5)(5,0){1}{\line(0,1){5}}
  \multiput(25,57.5)(0,5){2}{\line(1,0){5}}

 \multiput(-2,40)(5,0){2}{\line(0,1){10}}
  \multiput(-2,40)(0,5){3}{\line(1,0){5}}
  \multiput(8,42.5)(5,0){1}{\line(0,1){5}}
  \multiput(3,42.5)(0,5){2}{\line(1,0){5}}
\multiput(13,42.5)(5,0){1}{\line(0,1){5}}
  \multiput(8,42.5)(0,5){2}{\line(1,0){5}}
\multiput(18,42.5)(5,0){1}{\line(0,1){5}}
  \multiput(13,42.5)(0,5){2}{\line(1,0){5}}
\multiput(23,42.5)(5,0){1}{\line(0,1){5}}
  \multiput(18,42.5)(0,5){2}{\line(1,0){5}}

\multiput(26,40)(5,0){2}{\line(0,1){10}}
  \multiput(26,40)(0,5){3}{\line(1,0){5}}
  \multiput(36,42.5)(5,0){1}{\line(0,1){5}}
  \multiput(31,42.5)(0,5){2}{\line(1,0){5}}
\multiput(41,42.5)(5,0){1}{\line(0,1){5}}
  \multiput(36,42.5)(0,5){2}{\line(1,0){5}}
\multiput(36,42.5)(5,0){1}{\line(0,1){5}}
  \multiput(41,42.5)(0,5){2}{\line(1,0){5}}
\multiput(46,40)(5,0){2}{\line(0,1){10}}
  \multiput(46,40)(0,5){3}{\line(1,0){5}}

\multiput(60,55)(5,0){2}{\line(0,1){10}}
  \multiput(60,55)(0,5){3}{\line(1,0){5}}
  \multiput(65,57.5)(0,5){2}{\line(1,0){5}}
\multiput(70,55)(5,0){2}{\line(0,1){10}}
  \multiput(70,55)(0,5){3}{\line(1,0){5}}
\multiput(75,57.5)(0,5){2}{\line(1,0){5}}
\multiput(80,57.5)(0,5){1}{\line(0,1){5}}

 \multiput(-10,25)(5,0){2}{\line(0,1){10}}
  \multiput(-10,25)(0,5){3}{\line(1,0){5}}
  \multiput(0,27.5)(5,0){1}{\line(0,1){5}}
  \multiput(-5,27.5)(0,5){2}{\line(1,0){5}}
\multiput(5,27.5)(5,0){1}{\line(0,1){5}}
  \multiput(0,27.5)(0,5){2}{\line(1,0){5}}
\multiput(10,27.5)(5,0){1}{\line(0,1){5}}
  \multiput(5,27.5)(0,5){2}{\line(1,0){5}}
\multiput(15,27.5)(5,0){1}{\line(0,1){5}}
  \multiput(10,27.5)(0,5){2}{\line(1,0){5}}
\multiput(20,27.5)(5,0){1}{\line(0,1){5}}
  \multiput(15,27.5)(0,5){2}{\line(1,0){5}}

\multiput(23,25)(5,0){2}{\line(0,1){10}}
  \multiput(23,25)(0,5){3}{\line(1,0){5}}
  \multiput(33,27.5)(5,0){1}{\line(0,1){5}}
  \multiput(28,27.5)(0,5){2}{\line(1,0){5}}
\multiput(38,27.5)(5,0){1}{\line(0,1){5}}
  \multiput(33,27.5)(0,5){2}{\line(1,0){5}}
\multiput(33,27.5)(5,0){1}{\line(0,1){5}}
  \multiput(38,27.5)(0,5){2}{\line(1,0){5}}
\multiput(43,25)(5,0){2}{\line(0,1){10}}
  \multiput(43,25)(0,5){3}{\line(1,0){5}}
\multiput(43,27.5)(5,0){1}{\line(0,1){5}}
  \multiput(48,27.5)(0,5){2}{\line(1,0){5}}
\multiput(53,27.5)(5,0){1}{\line(0,1){5}}

\multiput(62,40)(5,0){2}{\line(0,1){10}}
  \multiput(62,40)(0,5){3}{\line(1,0){5}}
  \multiput(67,42.5)(0,5){2}{\line(1,0){5}}
\multiput(72,40)(5,0){2}{\line(0,1){10}}
  \multiput(72,40)(0,5){3}{\line(1,0){5}}
\multiput(77,42.5)(0,5){2}{\line(1,0){5}}
\multiput(82,42.5)(0,5){1}{\line(0,1){5}}
\multiput(82,42.5)(0,5){2}{\line(1,0){5}}
\multiput(87,42.5)(0,5){1}{\line(0,1){5}}

\multiput(90,40)(5,0){2}{\line(0,1){10}}
  \multiput(90,40)(0,5){3}{\line(1,0){5}}
  \multiput(95,42.5)(0,5){2}{\line(1,0){5}}
\multiput(100,40)(5,0){2}{\line(0,1){10}}
  \multiput(100,40)(0,5){3}{\line(1,0){5}}
\multiput(105,42.5)(0,5){2}{\line(1,0){5}}
\multiput(110,42.5)(0,5){1}{\line(0,1){5}}
\multiput(110,40)(0,5){3}{\line(1,0){5}}
\multiput(105,40)(5,0){2}{\line(0,1){10}}
\multiput(110,40)(5,0){2}{\line(0,1){10}}

\multiput(56,25)(5,0){2}{\line(0,1){10}}
  \multiput(56,25)(0,5){3}{\line(1,0){5}}
  \multiput(61,27.5)(0,5){2}{\line(1,0){5}}
\multiput(66,25)(5,0){2}{\line(0,1){10}}
  \multiput(66,25)(0,5){3}{\line(1,0){5}}
\multiput(71,27.5)(0,5){2}{\line(1,0){5}}
\multiput(76,27.5)(0,5){1}{\line(0,1){5}}
\multiput(76,27.5)(0,5){2}{\line(1,0){5}}
\multiput(81,27.5)(0,5){1}{\line(0,1){5}}
\multiput(81,27.5)(0,5){2}{\line(1,0){5}}
\multiput(86,27.5)(0,5){1}{\line(0,1){5}}

\multiput(89,25)(5,0){2}{\line(0,1){10}}
  \multiput(89,25)(0,5){3}{\line(1,0){5}}
  \multiput(94,27.5)(0,5){2}{\line(1,0){5}}
\multiput(99,25)(5,0){2}{\line(0,1){10}}
  \multiput(99,25)(0,5){3}{\line(1,0){5}}
\multiput(104,27.5)(0,5){2}{\line(1,0){5}}
\multiput(109,27.5)(0,5){1}{\line(0,1){5}}
\multiput(109,25)(0,5){3}{\line(1,0){5}}
\multiput(104,25)(5,0){2}{\line(0,1){10}}
\multiput(109,25)(5,0){2}{\line(0,1){10}}
\multiput(114,27.5)(0,5){2}{\line(1,0){5}}
\multiput(119,27.5)(0,5){1}{\line(0,1){5}}
\end{picture}
\caption{The genealogical tree of the narrow Carnot algebras with $s({\mathfrak g})\le 6$.}
\label{tree}
\end{figure}

The Carnot algebra ${\mathfrak m}_0(4)$ extends to two nonisomorphic Carnot algebras ${\mathfrak m}_0(5)$, 
${\mathfrak m}_1(5)$ (one-dimensional extensions) and a single two-dimensional central extension  ${\mathfrak m}_0^5(5)$.  The Lie algebra 
${\mathfrak m}_0^3(4)$ has three nonisomorphic Carnot extensions: one-dimensional ${\mathfrak m}_0^3(5)$, ${\mathfrak m}_1^3(5)$ and two-dimensional extension ${\mathfrak m}_0^{3,5}(5)$. The Lie quotient-algebra ${\mathfrak n}_1^{\pm}(4)$ has two non-isomorphic Carnot extensions: a one-dimensional  ${\mathfrak n}_{1,1}^{\pm}(5)$ and two-dimensional
${\mathfrak n}_{1}^{\pm}(5)$. Summarizing the preliminary results, we compile the following list of Carnot algebras of height $5$ satisfying condition (\ref{3/2})
\begin{equation}
{\mathfrak m}_0(5), {\mathfrak m}_1(5), {\mathfrak m}_1^3(5), {\mathfrak m}_0^3(5),  {\mathfrak m}_0^5(5), {\mathfrak m}_0^{3,5}(5), {\mathfrak n}_{1,1}^{\pm}(5), {\mathfrak n}_{1}^{\pm}(5).
\end{equation}

Next step. We must now classify all possible Carnot extensions for each Lie algebra from this list.

In the filiform Lie algebra  ${\mathfrak m}_0(5)$ has a unique Carnot extension ${\mathfrak m}_0(6)$.  In the same way, the Carnot algebra ${\mathfrak m}_0^3(5)$ can be extended to   ${\mathfrak m}_0^3(6)$. The Lie algebras ${\mathfrak m}_1(5),{\mathfrak m}_1^3(5)$ have no Carnot extensions.

The first interesting case is the investigation of the Carnot extensions of the Lie algebra ${\mathfrak m}_0^5(5)$. The differential $d$ of the cochain complex $\Lambda^* ({\mathfrak m}_0^5(5))$ is defined by formulas
$$
\begin{array}{c}
da^1=db^1=0, \;\;\; da^2=a^1\wedge b^1, \;\;\; da^3=a^1\wedge a^2, \\ da^4=a^1\wedge a^3, \; da^5=a^1\wedge a^4,\;
 db^5={-}b^1\wedge b^4+a^2\wedge a^3.
\end{array}
$$
The cohomology subspace $H^2_{(6)}({\mathfrak m}_0^5(5))$ of grading $6$ is a linear span of two basic cocycles
$$
a^1{\wedge}a^5, \quad a^1{\wedge}b^5-2b^1{\wedge}a^5+a^2{\wedge}a^4.
$$
A graded automorphism $\varphi$ of the Lie algebra ${\mathfrak m}_0^5(5)$ acts on its cochain complex as follows
$$
\begin{array}{c}
\varphi(a^1)=\alpha a^1, \varphi(b^1)=\rho a^1+\mu b^1, \varphi(a^2)=\alpha \mu a^2, \varphi(a^3)=\alpha^2 \mu a^3,  \varphi(a^4)=\alpha^3 \mu a^4,\\
  \varphi(a^5)=\alpha^4 \mu a^5, \varphi(b^5)=\alpha^3\mu \left( \rho a^5+ \mu b^5\right), \; \alpha\mu \ne 0.
\end{array}
$$
Thus, the automorphism $\varphi$ acts on the cohomology $H^2_{(6)}({\mathfrak m}_0^5(5))$  as a linear operator with an upper triangular matrix
$$
\alpha^4\mu \begin{pmatrix} \alpha & 3\rho\\ 0 & \mu \end{pmatrix}, \alpha\mu \ne 0.
$$
Obviously, the action of $\varphi$ on projectivization ${\mathbb P}H^2_{6}({\mathfrak m}_0^5(5))$
has two orbits defined by their representatives $(1:0)$ and $(0:1)$. Hence, we have two nonisomorphic one-dimensional Carnot extensions, one of them is isomorphic ${\mathfrak m}_0^5(6)$  (corresponds to the orbit of the point $ (1: 0) $), another extension is isomorphic to  ${\mathfrak m}_{0,2}(6)$ (the orbit of the point $(0:1)$). 

The case of the Lie algebra ${\mathfrak m}_0^{3,5}(5)$ is considered similarly to the previous one. We obtain two non-isomorphic Carnot extensions ${\mathfrak m}_0^{3,5}(5)$ this are the Lie algebras ${\mathfrak m}_0^{3,5}(6)$ and ${\mathfrak m}_{0,2}^3(6)$. The length of these Carnot algebras is six.

\begin{propos}
There exist the following isomorphisms
$$
{\mathfrak m}_{0,2}(6) \cong {\mathfrak n}_{2}(6), {\mathfrak m}_{0,2}^3(6) \cong {\mathfrak n}_{2}^3(6).
$$
\end{propos}
\begin{proof}
We define a linear map $\psi: {\mathfrak m}_{0,2}(6) \to {\mathfrak n}_{2}(6)$ as
$$
\begin{array}{c}
\psi(a_1)= f_1,\psi(b_1)=f_2, \psi(a_2)=f_3, \psi(a_3)=f_4, \psi(a_4)={-}3f_5,\\
\psi(a_5)=6f_6, \psi(b_5)= 3 f_7, \psi(a_6)=3f_8,
\end{array}
$$
where $a_1,b_1,\dots,a_6$ и $f_1, f_2,\dots, f_8$ are standard bases of Lie algebras ${\mathfrak m}_{0,2}(6)$ and ${\mathfrak n}_{2}(6)$ respectively.
It is easy to verify that the mapping $\psi$ is compatible with the commutation relations of both Lie algebras. The isomorphism $\psi$ constructed is extended to an isomorphism $\Psi: {\mathfrak m}_{0,2}^3(6) \to {\mathfrak n}_{2}^3(6)$. We set  $\Psi(b_3)=z$. 
\end{proof}

According to Lemma \ref{n_1_ext} the Lie algebras ${\mathfrak n}_1^{\pm}(5)$ can be extended only to ${\mathfrak n}_1^{\pm}(6)$ and the Lie algebra ${\mathfrak n}_{1,1}^{\pm}(5)$ does not admit Carnot extensions..

Now we can sum up the results in the list of Carnot algebras of length $6$
\begin{equation}
{\mathfrak m}_0(6), {\mathfrak m}_0^3(6), {\mathfrak m}_0^5(6), {\mathfrak m}_0^{3,5}(6), {\mathfrak n}_{1}^{\pm}(6),  {\mathfrak n}_{2}(6), {\mathfrak n}_{2}^3(6).
\end{equation}
We have depicted the whole scheme of the classification of Carnot algebras of length $n\le 6$ in Fig. \ref{tree}. Double arrows denote two-dimensional extensions of Carnot. In an informal sense, this scheme is similar to the genealogical tree with the "ancestor" ${\mathfrak m}_0(2)$ of all Carnot algebras that satisfy the condition (\ref{3/2}).

\begin{figure}[tbbb]
\begin{picture}(110,40)(-10,20)

\put(31, 59){${\mathfrak n}_2(6)\cong {\mathfrak m}_{0,2}(6)$}
\put(0, 25){${\mathfrak n}_{2,1}(7), {\mathfrak n}_{2,2}(7)\cong {\mathfrak m}_{0,3}(7), {\mathfrak n}_{2,3}(7)$}
\put(70, 25){${\mathfrak n}_{2}(7)$}

\put(60,45){\vector(1,-1){4}}
\put(62,45){\vector(1,-1){4}}

\put(30,45){\vector(-1,-1){4}}

\multiput(30,47)(5,0){2}{\line(0,1){10}}
  \multiput(30,47)(0,5){3}{\line(1,0){5}}
  \multiput(35,49.5)(0,5){2}{\line(1,0){10}}
\put(40,49.5){\line(0,1){5}}
\put(45,49.5){\line(0,1){5}}
 \multiput(45,49.5)(0,5){2}{\line(1,0){5}}
\multiput(50,47)(5,0){2}{\line(0,1){10}}
  \multiput(50,47)(0,5){3}{\line(1,0){5}}
\multiput(55,49.5)(0,5){2}{\line(1,0){5}}
\multiput(60,49.5)(0,5){1}{\line(0,1){5}}

\multiput(10,30)(5,0){2}{\line(0,1){10}}
  \multiput(10,30)(0,5){3}{\line(1,0){5}}
  \multiput(15,32.5)(0,5){2}{\line(1,0){10}}
\put(20,32.5){\line(0,1){5}}
\put(25,32.5){\line(0,1){5}}
 \multiput(25,32.5)(0,5){2}{\line(1,0){5}}
\multiput(30,30)(5,0){2}{\line(0,1){10}}
  \multiput(30,30)(0,5){3}{\line(1,0){5}}
\multiput(35,32.5)(0,5){2}{\line(1,0){10}}
\multiput(40,32.5)(0,5){1}{\line(0,1){5}} 
\multiput(45,32.5)(0,5){1}{\line(0,1){5}} 

\multiput(50,30)(5,0){2}{\line(0,1){10}}
  \multiput(50,30)(0,5){3}{\line(1,0){5}}
  \multiput(55,32.5)(0,5){2}{\line(1,0){10}}
\put(60,32.5){\line(0,1){5}}
\put(65,32.5){\line(0,1){5}}
 \multiput(65,32.5)(0,5){2}{\line(1,0){5}}
\multiput(70,30)(5,0){2}{\line(0,1){10}}
  \multiput(70,30)(0,5){3}{\line(1,0){5}}
\multiput(75,32.5)(0,5){2}{\line(1,0){5}}
\multiput(80,32.5)(0,5){1}{\line(0,1){5}}
  \multiput(80,30)(0,5){3}{\line(1,0){5}}
\multiput(80,30)(5,0){2}{\line(0,1){10}}
\end{picture}
\caption{Carnot extensions ${\mathfrak n}_2(6) \cong {\mathfrak m}_{0,2}(6)$.}
\label{}
\end{figure}
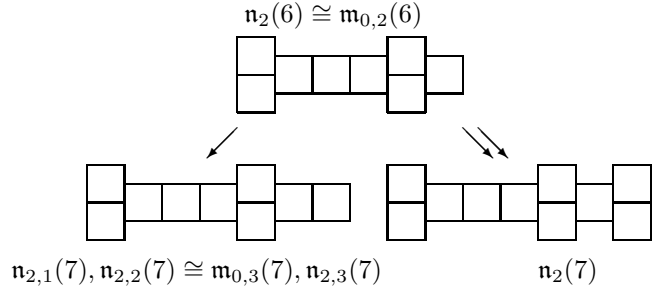

Now we turn to the study of Carnot algebras of length $7$. From Lemma \ref{lemma_m_0} it follows that
the Lie algebra ${\mathfrak m}_{0}(6)$ has two extensions ${\mathfrak m}_{0}(7)$ and
${\mathfrak m}_{1}(7)$.  The same holds for the Carnot algebras 
$ {\mathfrak m}_0^3(6), {\mathfrak m}_0^5(6), {\mathfrak m}_0^{3,5}(6)$.They are extendable to ${\mathfrak m}_0^3(7), {\mathfrak m}_0^5(7), {\mathfrak m}_0^{3,5}(7)$ or to
$ {\mathfrak m}_1^3(7), {\mathfrak m}_1^5(7), {\mathfrak m}_1^{3,5}(7)$ respectively. Also each of them
there is a unique two-dimensional Carnot extension: all together they are Carnot algebras ${\mathfrak m}_{0}^{3,7}(7), {\mathfrak m}_{0}^{5,7}(7)$ and ${\mathfrak m}_{0}^{3,5,7}(7)$.

According to the Lemma \ref{n_1_ext} the Lie algebra ${\mathfrak n}_{1}^{\pm}(6)$ extends to ${\mathfrak n}_{1}^{\pm}(7)$ and to ${\mathfrak n}_{1,1}^{\pm}(7)$.

\begin{propos}
The Lie algebra  ${\mathfrak n}_2(6)$ has three nonisomorphic one-dimensional Carnot extensions ${\mathfrak n}_{2,1}(7)$, ${\mathfrak n}_{2,2}(7)$, ${\mathfrak n}_{2,3}(7)$ and unique two-dimensional extension
${\mathfrak n}_{2}(7)$.
\end{propos}
\begin{proof}
We have already noted above, in the proof of Proposition \ref{prop_m_0,2} that
the homogeneous subspace $H^2_{(7)}({\mathfrak n}_2(6))$ is the linear span of the basic cocycles
$$
 {-}a^1{\wedge}a^6{+}a^2{\wedge}a^5{+}3a^3{\wedge}a^4, \; b^1{\wedge}a^6{-}a^2{\wedge}b^5.
$$
The group $Aut_{gr}({\mathfrak n}_2(6))$ is a two-dimensional algebraic torus that acts on a two-dimensional plane $H^2_{(7)}({\mathfrak n}_{2}(6))$ by diagonal matrices
$$
(\alpha, \mu) \to  \alpha^4\mu^2  \begin{pmatrix} \alpha & 0 \\ 0& \mu \end{pmatrix}, \; (\alpha,\mu) \in {\mathbb K}^*\times {\mathbb K}^*.
$$
The corresponding projective action has three orbits represented by the points $(1: 0), (0: 1), (1: 1)$ of the projective line ${\mathbb K}P^1$. Hence, we have three nonisomorphic one-dimensional Carnot extensions that are given by the following cocycles
$$
\begin{array}{l}
\omega_1={-}a^1\wedge a^6+a^2 \wedge a^5+3a^3 \wedge a^4, \\
\omega_2=
b^1 \wedge a^6-a^1 \wedge a^6-a^2 \wedge b^5+
a^2 \wedge a^5+3a^3 \wedge a^4,\\
\omega_3=2b^1\wedge a^6-2a^2\wedge b^5.
\end{array}
$$
We denote the corresponding Carnot extensions by ${\mathfrak n}_{2,1}(7)$, ${\mathfrak n}_{2,2}(7)$, ${\mathfrak n}_{2,3}(7)$.
\end{proof}
\begin{corollary}
The Lie algebra  ${\mathfrak n}_2^3(6)$ has three non isomorphic one-dimensional central extensions  ${\mathfrak n}_{2,1}^3(7)$, ${\mathfrak n}_{2,2}^3(7)$, ${\mathfrak n}_{2,3}^3(7)$ and one two-dimensional one
${\mathfrak n}_{2}^3(7)$.
\end{corollary}
\begin{propos}
There exist isomorphisms
\begin{equation}
\label{isom_m_0,3}
{\mathfrak n}_{2,1}(7) \cong {\mathfrak m}_{0,3}(7), \;{\mathfrak n}_{2,1}^3 \cong {\mathfrak m}_{0,3}^3(7).
\end{equation}
\end{propos}
\begin{proof}
In fact, the isomorphism considered above $\psi: {\mathfrak m}_{0,2}(6) \to {\mathfrak n}_{2}(6)$  can be extended to an isomorphism $\psi: {\mathfrak m}_{0,3}(7) \to {\mathfrak n}_{2,1}(7)$, if we define $\psi( a_7)=3f_9$. The second isomorphism $\psi: {\mathfrak m}_{0,3}^3(7) \to {\mathfrak n}_{2,1}^3(7)$ is defined in the following way: $\psi(b_3)=z$.
\end{proof}

We complete the first stage of the proof of the theorem, as its result we represent the list of Carnot algebras of length $7$ satisfying the condition 
(\ref{3/2})
\begin{equation}
\label{Carnot_dlina_7}
\begin{array}{c}
{\mathfrak m}_0(7), {\mathfrak m}_0^3(7), {\mathfrak m}_0^5(7),  {\mathfrak m}_0^7(7), {\mathfrak m}_0^{3,5}(7),{\mathfrak m}_0^{3,7}(7), {\mathfrak m}_0^{5,7}(7),  {\mathfrak m}_0^{3,5,7}(7), \\ {\mathfrak m}_1(7), {\mathfrak m}_1^3(7), {\mathfrak m}_1^5(7), {\mathfrak m}_1^{3,5}(7),\quad {\mathfrak n}_{1}^{\pm}(7), {\mathfrak n}_{1,1}^{\pm}(7), \\ {\mathfrak n}_{2}(7), {\mathfrak n}_{2,1}(7),
{\mathfrak n}_{2,2}(7), {\mathfrak n}_{2,3}(7), \quad {\mathfrak n}_{2}^3(7), {\mathfrak n}_{2,1}^3(7),
{\mathfrak n}_{2,2}^3(7), {\mathfrak n}_{2,3}^3(7).
\end{array}
\end{equation}

The second stage of the proof of the main theorem has actually been carried out in the previous two sections. Thanks to Lemma \ref{lemma_m_0}, Propositions \ref{prop_m_1},  \ref{prop_m_0,2}, \ref{prop_m_0,3}, Lemmas \ref{n_1_ext}, \ref{n_2_ext} we have absolutely
A deterministic process of Carnot extensions for Lie algebras from families
$$
\begin{array}{c}
\mathfrak{m}_{0}^{S_n}(n), 
\mathfrak{m}_{1}^{S_{n-2}}(2m{-}1), 
\mathfrak{m}_{0,2}^{S_{n-3}}(2m),
\mathfrak{m}_{0,3}^{S_{n-4}}(2m{+}1),
\mathfrak{n}_1^{\pm}(n), 
\mathfrak{n}_{1,1}^{\pm}(2m{+}1),\\
\mathfrak{n}_2(n), \mathfrak{n}_2^3(n), \;
\mathfrak{n}_{2,s}(2m{+}1), \mathfrak{n}_{2,s}^3(2m{+}1), \mathfrak{n}_{2,s}(2m+5), \mathfrak{n}_{2,s}^3(2m+5),s=1,2,3.
\end{array}
$$
and the list (\ref {Carnot_dlina_7}) of Carnot algebras of length $7$ contains Lie algebras only from these semigroups and we can continue to algorithmically continue to construct an infinite tree of classifications of such Lie algebras. This inductive argument completes the second, the same and last, stage of the proof of the main theorem.

 As an example, we construct a list of Carnot algebras of length $8$. First we note that the Lie algebras
${\mathfrak m}_1(7), {\mathfrak m}_1^3(7), {\mathfrak m}_1^5(7), {\mathfrak m}_1^{3,5}(7)$ are non extendable according to Propositions \ref{prop_m_1},  \ref{prop_m_0,2}, \ref{prop_m_0,3}. The Lie algebras $\mathfrak{n}_{1,1}^{\pm}(7)$ do not admit Carnot extensions according to
Lemma \ref{n_1_ext}. Similarly, this is the case of the Carnot algebras
${\mathfrak n}_{2,s}(7), {\mathfrak n}_{2,s}^3(7), s=1,2,3$. The last assertion follows from Lemma \ref{n_2_ext}.

The remaining possibilities of Lemmas and Propositions cited by us leave such a list of narrow Carnot algebras of length $8$
$$
\begin{array}{c}
{\mathfrak m}_0(8), {\mathfrak m}_0^3(8), {\mathfrak m}_0^5(8),   {\mathfrak m}_{0}^7(8), {\mathfrak m}_0^{3,5}(8),  {\mathfrak m}_{0}^{3,7}(8),   {\mathfrak m}_{0}^{5,7}(8), {\mathfrak m}_{0}^{3,5,7}(8), \\ {\mathfrak m}_{0,2}(8),{\mathfrak m}_{0,2}^{3}(8), {\mathfrak m}_{0,2}^{5}(8),  {\mathfrak m}_{0,2}^{3,5}(8), {\mathfrak n}_{1}^{\pm}(8), {\mathfrak n}_{2}(8), {\mathfrak n}_{2}^3(8).
\end{array}
$$
\end{proof}
We note only at the very end that of all our series of Carnot algebras in small dimensions no Lie algebras of the form  ${\mathfrak m}_{0,3}^{S_{2m-3}}(2m+1)$. The first of them  ${\mathfrak m}_{0,3}(9),{\mathfrak m}_{0,3}^{3}(9), {\mathfrak m}_{0,3}^{5}(9),  {\mathfrak m}_{0,3}^{3,5}(9)$ will appear in the next step of our classification. The point is that up to the level of $9$, algebras of this type were isomorphic to other Carnot algebras, recall, for example, isomorphisms (\ref{isom_m_0,3}).

\section{Classification in the infinite-dimensional case}\label{s8}

\begin{theorem}
\label{second_osnovn}
Let $\mathfrak{g}=\bigoplus_{i=1}^{+\infty}\mathfrak{g}_i$ be
an infinite-dimensional real positive-graded Lie algebra such that
\begin{equation}
\label{3/2}
[e_1,e_i]=e_{i+1}, \; \dim{\mathfrak{g}_i}+\dim{\mathfrak{g}_{i{+}1}}\le 3, \; i \in {\mathbb N}.
\end{equation}
Then $\mathfrak{g}$ is isomorphic to one and only one Lie algebra from the following list
$$\mathfrak{n}_1^{\pm}, \mathfrak{n}_2, \mathfrak{n}_2^3, \mathfrak{m}_{0}, \{ \mathfrak{m}_{0}^{S}, S \subset \{3,5,7,\dots,2m{+}1,\dots \}.$$
The complex classification of ${\mathbb N}$-graded Lie algebras with properties (\ref{3/2}) differs from the real one only in one point: Lie algebras $\mathfrak{n}_1^{\pm}$, nonisomorphic over ${\mathbb R}$, are isomorphic over the field of complex numbers ${\mathbb C}$.
\end{theorem}
\begin{proof}
Each infinite-dimensional ${\mathbb N}$-graded Lie algebra $\mathfrak {g} = \bigoplus_{i=1}^{+\infty}\mathfrak{g}_i$ of finite width can be associated with an infinite
spectrum of finite-dimensional nilpotent Lie algebras
$$
 \dots \stackrel{p_{k{+}2,k{+}1}}{\longrightarrow}{\mathfrak g}/{\mathfrak g}^{k{+}1} \stackrel{p_{k{+}1,k}}{\longrightarrow}{\mathfrak g}/
{\mathfrak g}^k \stackrel{p_{k,k{-}1}}{\longrightarrow}
\dots \stackrel{p_{3,2}}{\longrightarrow} {\mathfrak g}/{\mathfrak g}^2 \stackrel{p_{2,1}}{\longrightarrow}{\mathfrak g}/{\mathfrak g}^1,
$$
It is easy to see that an arbitrary quotient Lie algebra ${\mathfrak g}/{\mathfrak g}^{k+1}$ of this spectrum is a Carnot algebra of length $k$; in particular, it has the property to be infinitely extended by means of Carnot extensions.

From the proof of the main Theorem \ref{osnovn} it follows that the following finite finite-dimensional Carnot algebras ${\mathfrak g}$ do not have Carnot extensions
\begin{eqnarray*}
{\mathfrak m}_1^{r_1,\dots,r_k}(2m{+}1),   {\mathfrak m}_{0,3}^{r_1,\dots,r_k}(2m{+}1),
{\mathfrak n}_{1,1}^{\pm}(2m{+}1),\\
{\mathfrak n}_{2,s}(6m{+}q), {\mathfrak n}_{2,s}^3(6m{+}q), s=1,2,3,  q=1,5,
\end{eqnarray*}

It should be noted that we discuss precisely the Carnot extensions and nothing else.
The Lie algebra ${\mathfrak n}_{1,1}^{\pm}(2m{+}1)$ has one-dimensional central extension  ${\mathfrak n}_{1}^{\pm}(2m{+}1)$, but it is not a Carnot extension, because  the Lie algebra ${\mathfrak n}_{1}^{\pm}(2m{+}1)$ has exactly the same nil-index
$s=2m{+}1$ as the initial Lie algebra ${\mathfrak n}_{1,1}^{\pm}(2m{+}1)$. 

The Carnot algebra ${\mathfrak n}_{2,s}(6m{+}q), s=1,2,3, q=1,5,$ is also extendable to the Carnot algebra  ${\mathfrak n}_{2}(6m{+}q)$, $q=1,5$, but this central extension will not be an Carnot extension. There are no Carnot extensions of this Lie algebra. The same is true for Carnot algebras ${\mathfrak n}_{2,s}^3(6m{+}q), s=1,2,3, q=1,5$. 

We recall that the Carnot algebra
${\mathfrak m}_{0,2}^{r_1, \dots, r_k}(2m)$ is the unique (up to isomorphism) Carnot expansion ${\mathfrak m}_{0,3}^{r_1,\dots,r_k}(2m{+}1)$. And this Carnot algebra, as we have already noted, does not allow the extension of Carnot.

As a result, we have four infinite spectra of Carnot algebras corresponding to infinite-dimensional Lie algebras ${\mathfrak n}_1^{\pm}, {\mathfrak n}_2, {\mathfrak n}_2^3$ and another spectrum connected with the Carnot algebras
 ${\mathfrak m}_{0}^{r_1,\dots,r_k}(n)={\mathfrak m}_{0}^{S_n}(n)$
\begin{equation}
 \dots \stackrel{p_{n{+}2,n{+}1}}{\longrightarrow}{\mathfrak m}_0^{S_{n+1}}(n{+}1) \stackrel{p_{n{+}1,n}}{\longrightarrow}{\mathfrak m}_0^{S_{n}}(n)  \stackrel{p_{n,n{-}1}}{\longrightarrow}
\dots \stackrel{p_{4,3}}{\longrightarrow}{\mathfrak m}_0^{S_{3}}(3) \stackrel{p_{3,2}}{\longrightarrow}{\mathfrak m}_0(2),
\end{equation}
where the chain of nested subsets of the set of odd numbers
$$
S_3 \subset S_4 \subset \dots \subset S_n \subset S_{n+1} \subset \dots
$$
is such that for all $m \ge 1$
$
S_{2m+1}, S_{2m+2} \subset \left\{ 3, 5, \dots, 2m{+}1\right\}.
$
We denote by $S=\cup_{n \to + \infty}S_n$ the limit of this chain of subsets of the set of odd numbers. This limit $S$ will itself be a subset, possibly finite, infinite or even empty, of the set of odd natural numbers of greater unity.
$$
S \subset  \left\{ 3, 5, \dots, 2m{-}1, \dots \right\}.
$$
We conclude that in the latter case an arbitrary subset $S$ of the set of odd numbers larger than one determines a naturally graded Lie algebra ${\mathfrak m}_0^S$. This means in particular that the Lie algebras of the form ${\mathfrak m}_0^S$ will be an entire continuum.
\end{proof}
\begin{remark}
We have already noted earlier that the growth function $F_{\mathfrak g}(n)$ of an arbitrary ${\mathbb N}$-graded Lie algebra ${\mathfrak g}=\oplus_{i=1}^{+\infty}{\mathfrak g}_i$, generally speaking, is not always completely determined by the dimensions of $\dim{\mathfrak g}_i$ of its homogeneous components. But it will be so, if its ${\mathbb N}$-grading is natural $[{\mathfrak g}_1,{\mathfrak g}_i]={\mathfrak g}_{i+1}$.

Another important observation is that we constructed a continuum of pairwise nonisomorphic linearly increasing Lie algebras ${\mathfrak m}_0^S$. While, according to Mathieu's theorem \cite {Mat}, there is only a countable number of pairwise nonisomorphic simple ${\mathbb Z}$-graded Lie algebras of finite growth.
\end{remark}
\section*{The appendix 1}
A0. Naturally graded filiform Lie algebras.
$$
\begin{tabular}{|c|c|c|}
\hline
&&\\[-10pt]
\begin{tabular}{c} algebra, \\ parameters \\ \end{tabular} &  \begin{tabular}{c}  dimension,  \\ 
basis  \end{tabular} & relations  \\
&&\\[-10pt]
\hline
&&\\[-10pt]
\begin{tabular}{c}
$\mathfrak{m}_{0}(n)$, \\ 
$ n \ge 2$\\ \end{tabular}
&
\begin{tabular}{c}
$n{+}1$ \\
\hline
$a_1,{\dots},a_{n},b_1$\\
[2pt]
\end{tabular} 
&
$[a_1, b_1]=a_2, \; [a_1, a_i]=a_{i{+}1}, \; i{=}2, \dots, n{-}1$;
\\
&&\\[-10pt]
\hline
&&\\[-10pt]
\begin{tabular}{c}
$\mathfrak{m}_{1}(2m{-}1)$, \\
$m \ge 3$\\ \end{tabular}
&
\begin{tabular}{c}
$2m$ \\
\hline
$a_1,{\dots},a_{2m{-}1},b_1$
\end{tabular} 
&
\begin{tabular}{c}
$[a_1, b_1]=a_2, \; [a_1, a_i]=a_{i{+}1}, \; i{=}2, \dots, 2m{-}3$;\\
$[b_1, a_{{2m}{-}2}]{=}{-}a_{2m{-}1},$ \\ $[a_q,a_{2m{-}1{-}q}]=({-}1)^{q}a_{2m{-}1}, q{=}2,{\dots}, m{-}1;$\\
\end{tabular}
\\
&&\\
\hline
\end{tabular}
$$

A1. Central extensions of naturally graded filiform Lie algebras.
$$
\begin{tabular}{|c|c|c|}
\hline
&&\\[-10pt]
\begin{tabular}{c} algebra, \\ parameters \\ \end{tabular} &  \begin{tabular}{c} dimension,  \\ 
basis  \end{tabular} & relations  \\
&&\\[-10pt]
\hline
&&\\[-10pt]
\begin{tabular}{c}
$\mathfrak{m}_{0}^{r_1,\dots,r_k}(n)$, \\ 
$3{\le}r_1{<}{\dots}{<}{r_k}{\le}n,$\\
$r_j{=}2s_j{+}1, 1{\le}j{\le}k;$\\
$1 {\le} k{\le} \left[\frac{n}{2}\right],\; n \ge 2$\\ \end{tabular}
&
\begin{tabular}{c}
$n{+}k{+}1$ \\
\hline
$a_1,{\dots},a_{n},b_1,$\\
$b_{r_1},\dots, b_{r_k}$\\
[2pt]
\end{tabular} 
&
\begin{tabular}{c}
$[a_1, b_1]=a_2, \; [a_1, a_i]=a_{i{+}1}, \; i{=}2, \dots, n{-}1$;\\
  j{=}1,{\dots},k: \;
$[b_1, a_{{r_j}{-}1}]{=}{-}b_{r_j}$  {\rm and} \\$[a_q,a_{r_j{-}q}]=({-}1)^{q}b_{r_j}, q{=}2,{\dots}, \frac{r_j{-}1}{2}$.\\
\end{tabular}
\\
&&\\[-10pt]
\hline
&&\\[-10pt]
\begin{tabular}{c}
$\mathfrak{m}_{1}^{r_1,\dots,r_k}(2m{-}1)$, \\
$3{\le}r_1{<}{\dots}{<}{r_k}{\le}2m{-}3,$\\
$r_j{=}2s_j{+}1, 1{\le}j{\le}k;$\\
$0 {\le} k{\le} m{-}2,\; m \ge 3$\\ \end{tabular}
&
\begin{tabular}{c}
$2m{+}k$ \\
\hline
$a_1,{\dots},a_{2m{-}1},b_1,$\\
$b_{r_1},\dots, b_{r_k}$\\
\end{tabular} 
&
\begin{tabular}{c}
$[a_1, b_1]=a_2, \; [a_1, a_i]=a_{i{+}1}, \quad i{=}2, \dots, 2m{-}3$;\\
$[b_1, a_{{2m}{-}2}]{=}{-}b_{2m{-}1},$ \\ $[a_q,a_{2m{-}1{-}q}]=({-}1)^{q}b_{2m{-}1}, q{=}2,{\dots}, m{-}1;$\\
  j{=}1,{\dots},k: \;
$[b_1, a_{{r_j}{-}1}]{=}{-}b_{r_j}$ \; {\rm and} \\$[a_q,a_{r_j{-}q}]=({-}1)^{q}b_{r_j}, q{=}2,{\dots}, \frac{r_j{-}1}{2}$.\\
\end{tabular}
\\
&&\\[-10pt]
\hline
&&\\[-10pt]
\begin{tabular}{c}
$\mathfrak{m}_{0,2}^{r_1,\dots,r_k}(2m)$, \\
$3{\le}r_1{<}{\dots}{<}{r_k}{\le}2m{-}3,$\\
$r_j{=}2s_j{+}1, 0{\le}j{\le}k;$\\
$0 {\le} k{\le} m{-}2,\;m \ge 4$\\ \end{tabular}
&  
\begin{tabular}{c}
$2m{+}k{+}2$ \\
\hline
$a_1,{\dots},a_{2m},b_1,$\\
$b_{r_1},\dots, b_{r_k}, b_{2m{-}1}$\\
\end{tabular} 
&
\begin{tabular}{c}
$[a_1, b_1]=a_2, \; [a_1, a_i]=a_{i{+}1}, \quad i{=}2, \dots, 2m{-}2$;\\
$[b_1, a_{{2m}{-}2}]{=}{-}b_{2m{-}1},$ \\ 
$[a_q,a_{2m{-}1{-}q}]=({-}1)^{q}b_{2m{-}1}, q{=}2,{\dots}, m{-}1;$\\
$[a_1,b_{2m{-}1}]{=}a_{2m},$  \; $[b_1, a_{{2m}{-}1}]{=}{-}(m{-}1)a_{2m},$\\
$[a_q,a_{2m{-}q}]=({-}1)^{q}(m{-}q)a_{2m}, \; q{=}2,
\dots,m{-}1;$\\
  j{=}1,{\dots},k: \;
$[b_1, a_{{r_j}{-}1}]{=}{-}b_{r_j}$ \; {\rm and} \\$[a_q,a_{r_j{-}q}]=({-}1)^{q}b_{r_j}, q{=}2,{\dots}, \frac{r_j{-}1}{2}$.\\
\end{tabular}
\\
&&\\[-10pt]
\hline
&&\\[-10pt]
\begin{tabular}{c}
$\mathfrak{m}_{0,3}^{r_1,\dots,r_k}(2m{+}1)$, \\
$3{\le}r_1{<}{\dots}{<}{r_k}{\le}2m{-}3,$\\
$r_j{=}2s_j{+}1, 0{\le}j{\le}k;$\\
$0 {\le} k{\le} m{-}2,\;m \ge 4$\\ \end{tabular}
&  
\begin{tabular}{c}
$2m{+}k{+}3$ \\
\hline
$a_1,{\dots},a_{2m},a_{2m{+}1},$\\
$b_1,b_{r_1},\dots, b_{r_k}, b_{2m{-}1}$\\
\end{tabular} 
&
\begin{tabular}{c}
$[a_1, b_1]=a_2, \; [a_1, a_i]=a_{i{+}1}, \quad i{=}2, \dots, 2m{-}2$;\\
$[b_1, a_{{2m}{-}2}]{=}{-}b_{2m{-}1},$ \; \\ 
$[a_q,a_{2m{-}1{-}q}]=({-}1)^{q}b_{2m{-}1}, q{=}2,{\dots}, m{-}1;$\\
$[a_1,b_{2m{-}1}]{=}a_{2m},$  \;$[b_1, a_{{2m}{-}1}]{=}{-}(m{-}1)a_{2m},$ \\
$[a_q,a_{2m{-}q}]=({-}1)^{q}(m{-}q)a_{2m}, \; q{=}2,
\dots,m{-}1;$\\
$[a_1,a_{2m}]=a_{2m{+}1}$,\\
$[a_p,a_{2m{-}p{+}1}]=
({-}1)^{p{+}1}(p{-}1)\left(m{-}\frac{p}{2}\right)a_{2m{+}1}$,\\
$p{=}2,\dots,m$;\\
  j{=}1,{\dots},k: \;
$[b_1, a_{{r_j}{-}1}]{=}{-}b_{r_j}$ \; {\rm and} \\$[a_q,a_{r_j{-}q}]=({-}1)^{q}b_{r_j}, q{=}2,{\dots}, \frac{r_j{-}1}{2}$.\\
\\[-10pt]
\end{tabular}
\\
\hline
\end{tabular}
$$

B. Finite-dimensional Lie quotient-algebras of  $\mathfrak{n}_{1}^{\pm}$.
$$
\begin{tabular}{|c|c|c|}
\hline
&&\\[-10pt]
\begin{tabular}{c} algebra, \\ parameters \\ \end{tabular} &  \begin{tabular}{c} dimension,  \\ 
basis  \end{tabular} & relations  \\
&&\\[-10pt]
\hline
&&\\[-10pt]
\begin{tabular}{c}
$\mathfrak{n}_{1}^{\pm}(2m)$, \\
$m \ge 2$\\ \end{tabular}
&  
\begin{tabular}{c}
$3m$ \\
\hline
$a_1,a_3,{\dots},a_{2m{-}1},$\\
$b_1,b_3,{\dots},b_{2m{-}1},$\\
$a_2,a_4,{\dots},a_{2m}$\\
\end{tabular} 
&
\begin{tabular}{c}
$[b_{2p{+}1}, a_{2q} ]=a_{2(p{+}q){+}1}, \; p{+}q {\le}  m{-}1; $\\
$[a_{2p}, b_{2q{+}1} ]=\pm b_{2(p{+}q){+}1},   \; p{+}q {\le} m{-}1;$\\
$[a_{2p{+}1}, b_{2q{+}1} ]= a_{2(p{+}q{+}1)},  \; p{+}q{+}1{\le} m.$
\end{tabular}
\\
\hline
&&\\[-10pt]
\begin{tabular}{c}
$\mathfrak{n}_{1}^{\pm}(2m{+}1)$, \\
$m \ge 2$\\ \end{tabular}
&  
\begin{tabular}{c}
$3m{+}2$ \\
\hline
$a_1,a_3,{\dots},a_{2m{+}1},$\\
$b_1,b_3,{\dots},b_{2m{+}1},$\\
$a_2,a_4,{\dots},a_{2m}$\\
\end{tabular} 
&
\begin{tabular}{c}
$[b_{2p{+}1}, a_{2q} ]=a_{2(p{+}q){+}1},  \; p{+}q {\le}  m; $\\
$[a_{2p}, a_{2q{+}1} ]=\pm b_{2(p{+}q){+}1},  \; p{+}q {\le} m;$\\
$[a_{2p{+}1}, b_{2q{+}1} ]= a_{2(p{+}q{+}1)},  \; p{+}q{+}1{\le} m.$
\end{tabular}
\\
\hline
&&\\[-10pt]
\begin{tabular}{c}
$\mathfrak{n}_{1,1}^{\pm}(2m{+}1)$, \\
$m \ge 2$\\ \end{tabular}
&  
\begin{tabular}{c}
$3m{+}1$ \\
\hline
$a_1,{\dots},a_{2m{-}1},a_{2m{+}1},$\\
$b_1,b_3,{\dots},b_{2m{-}1},$\\
$a_2,a_4,{\dots},a_{2m}$\\
\end{tabular} 
&
\begin{tabular}{c}
$[b_{2p{+}1}, a_{2q} ]=a_{2(p{+}q){+}1},  \; p{+}q {\le}m; $\\
$[a_{2p}, a_{2q{+}1} ]=\pm b_{2(p{+}q){+}1},  \; p{+}q {\le}m{-}1;$\\
$[a_{2p{+}1}, b_{2q{+}1} ]= a_{2(p{+}q{+}1)}, \; p{+}q{+}1{\le} m.$
\end{tabular}
\\
%\hline
%&&\\[-10pt]
%\begin{tabular}{c}
%$\mathfrak{n}_{1,2}^{\pm}(2m{+}2)$, \\
%$m \ge 2$\\ \end{tabular}
%&  
%\begin{tabular}{c}
%$3m{+}2$ \\
%\hline
%$a_1,{\dots},a_{2m{-}1},a_{2m{+1}},$\\
%$b_1,b_3,{\dots},b_{2m{-}1},$\\
%$a_2,{\dots},a_{2m},z_{2m{+}2}$\\
%\end{tabular} 
%&
%\begin{tabular}{c}
%$[b_{2p{+}1}, a_{2q} ]=a_{2(p{+}q){+}1}, {\rm if} \; p{+}q {\le} m; $\\
%$[a_{2p}, a_{2q{+}1} ]={\pm} b_{2(p{+}q){+}1},  {\rm if} \; p{+}q {\le} m{-}1;$\\
%$[a_{2p{+}1}, b_{2q{+}1} ]= a_{2(p{+}q{+}1)},  {\rm if} \; p{+}q{+}1{\le} m.$\\
%$????[,]=z_{2m{+}2}$
%\end{tabular}
%\\
\hline
\end{tabular}
$$

C. Finite-dimensional Lie quotient-algebras of $\mathfrak{n}_{2}$.
$$
\begin{tabular}{|c|c|c|}
\hline
&&\\[-10pt]
\begin{tabular}{c} алгебра, \\ параметры \\ \end{tabular} &  \begin{tabular}{c} размерность,  \\ 
базис \end{tabular} & \begin{tabular}{c} коммутационные \\ соотношения,\\ $d_{ij}$ из Таблицы \ref{structure_const_n_2}. \end{tabular} \\
&&\\[-10pt]
\hline
&&\\[-10pt]
\begin{tabular}{c}
$\mathfrak{n}_{2}(6m{+}q)$, \\
$m \ge 1,$\\$1 \le q \le 4.$\\ 
\end{tabular}
&  
\begin{tabular}{c}
$8m+q+1$ \\
\hline
$f_{1},f_2,\dots,f_{8m{+}q{+}1}.$\\
\end{tabular} 
&
\begin{tabular}{c}
$\begin{array}{c}[f_i, f_j ]{=}d_{i,j}f_{i{+}j},  \\ i{+}j \le  8m{+}q{+}1.\end{array} $\\
\end{tabular}
\\
\hline
&&\\[-10pt]
\begin{tabular}{c}
$\mathfrak{n}_{2}(6m{+}q)$, \\
$m \ge 1,$\\$5 \le q \le 6.$\\ 
\end{tabular}
&  
\begin{tabular}{c}
$8m+q+2$ \\
\hline
$f_{1},f_2,\dots,f_{8m{+}q{+}2}.$\\
\end{tabular} 
&
\begin{tabular}{c}
$\begin{array}{c}[f_i, f_j ]{=}d_{i,j}f_{i{+}j},\\  i{+}j \le  8m{+}q{+}2 .\end{array} $\\
\end{tabular}
\\
\hline
&&\\[-10pt]
\begin{tabular}{c}
$\mathfrak{n}_{2,s}(6m{+}1)$, \\
$m \ge 1,$\\
$s=1,2,3.$ \end{tabular}
&  
\begin{tabular}{c}
$8m+1$ \\
\hline
$f_{1},\dots,f_{8m},c,$\\
$c{=}f_{8m{+}s}, s{=}1,2; c{=}f_{8m{+}1}{=}f_{8m{+}2}, s{=}3.$
\end{tabular} 
&
\begin{tabular}{c}
$\begin{array}{c}[f_i, f_j ]{=}d_{i,j}f_{i{+}j},\\  i{+}j \le  8m{+}s .\end{array} $\\
\end{tabular}
\\
\hline
&&\\[-10pt]
\begin{tabular}{c}
$\mathfrak{n}_{2,s}(6m{+}5)$, \\
$m \ge 1$,\\
$s=1,2,3$ \end{tabular}
&  
\begin{tabular}{c}
$8m+6$ \\
\hline
$f_{1},\dots,f_{8m{+}5},c,$\\
$c{=}f_{8m{+}5{+}s}, s{=}1,2; c{=}f_{8m{+}6}{=}f_{8m{+}7}, s{=}3.$
\end{tabular} 
&
\begin{tabular}{c}
$\begin{array}{c}[f_i, f_j ]{=}d_{i,j}f_{i{+}j},\\ i{+}j \le 8m{+}5{+}s .\end{array} $\\
\end{tabular}
\\
\hline
\end{tabular}
$$

D. Finite-dimensional Lie quotient-algebras of $\mathfrak{n}_{2}^3$.
$$
\begin{tabular}{|c|c|c|}
\hline
&&\\[-10pt]
\begin{tabular}{c} алгебра, \\ параметры \\ \end{tabular} &  \begin{tabular}{c} размерность,  \\ 
базис \end{tabular} & \begin{tabular}{c} коммутационные \\ соотношения,\\ $d_{ij}$ из Таблицы  \ref{structure_const_n_2}. \end{tabular} \\
&&\\[-10pt]
\hline
&&\\[-10pt]
\begin{tabular}{c}
$\mathfrak{n}_{2}^3(6m{+}q)$, \\
$m \ge 1,$\\$1\le q \le4$\\ 
\end{tabular}
&  
\begin{tabular}{c}
$8m+q+2$ \\
\hline
$f_{1},f_2,\dots,f_{8m{+}q{+}1},z$\\
\end{tabular} 
&
\begin{tabular}{c}
$\begin{array}{c}[f_i, f_j ]{=}d_{i,j}f_{i{+}j},  \\ i{+}j \le  8m{+}q{+}1, \\
\;[f_2,f_3]=z.\end{array} $\\
\end{tabular}
\\
\hline
&&\\[-10pt]
\begin{tabular}{c}
$\mathfrak{n}_{2}^3(6m{+}q)$, \\
$m \ge 1,$\\$5\le q \le 6$\\ 
\end{tabular}
&  
\begin{tabular}{c}
$8m+q+3$ \\
\hline
$f_{1},f_2,\dots,f_{8m{+}q{+}2},z.$\\
\end{tabular} 
&
\begin{tabular}{c}
$\begin{array}{c}[f_i, f_j ]{=}d_{i,j}f_{i{+}j},\\  i{+}j \le  8m{+}q{+}2 , \\
\;[f_2,f_3]=z.\end{array} $\\
\end{tabular}
\\
\hline
&&\\[-10pt]
\begin{tabular}{c}
$\mathfrak{n}_{2,s}^3(6m{+}1)$, \\
$m \ge 1,$\\
$s=1,2,3$ \end{tabular}
&  
\begin{tabular}{c}
$8m+2$ \\
\hline
$f_1,\dots,f_{8m},c,z,$\\
$c{=}f_{8m{+}s}, s{=}1,2; c{=}f_{8m{+}1}{=}f_{8m{+}2}, s{=}3.$\\
\end{tabular} 
&
\begin{tabular}{c}
$\begin{array}{c}[f_i, f_j ]{=}d_{i,j}f_{i{+}j},\\  i{+}j \le  8m{+}s,\\
\;[f_2,f_3]=z.\end{array} .$\\
\end{tabular}
\\
\hline
&&\\[-10pt]
\begin{tabular}{c}
$\mathfrak{n}_{2,s}^3(6m{+}5)$, \\
$m \ge 1$,\\
$s=1,2,3$ \end{tabular}
&  
\begin{tabular}{c}
$8m+7$ \\
\hline
$f_1,\dots,f_{8m{+}5},c,z,$\\
$c{=}f_{8m{+}5{+}s}, s{=}1,2; c{=}f_{8m{+}6}{=}f_{8m{+}7}, s{=}3.$\\
\end{tabular} 
&
\begin{tabular}{c}
$\begin{array}{c}[f_i, f_j ]{=}d_{i,j}f_{i{+}j},\\ i{+}j \le 8m{+}5{+}s \\
\;[f_2,f_3]=z.\end{array}. $\\
\end{tabular}
\\
\hline
\end{tabular}
$$

\section*{Appendix 2. Quasifiliform Lie algebras.}

Recall that for a filiform Lie algebra ${\mathfrak g}$, the length of its lower central series $s({\mathfrak g})$ is maximal for its dimension: $s({\mathfrak g})=\dim{\mathfrak g}-1$. It was quite natural to start studying the next class of nilpotent algebras satisfying the property $s({\mathfrak g})=\dim{\mathfrak g}-2$. They are called quasifiliform Lie algebras \cite{GJR}. The starting point of this research was the classification of naturally graded quasifiliform Lie algebras \cite{Go}. Naturally graded quasifiliform Lie algebras, according to \cite{Go}, can be divided into two subclasses: 

1) direct sums
${\mathfrak g}\oplus {\mathbb K}$ naturally graded filiform Lie algebras with one-dimensional abelian Lie algebras;   

2) naturally graded Lie algebras ${\mathfrak g}=\oplus_{i{=}1}^{n{-}2}{\mathfrak g}_i, \dim{{\mathfrak g}_1}=\dim{{\mathfrak g}_k}=2$ for some $k,  3 \le k \le n{-}2$, in other cases $\dim{{\mathfrak g}_k}=1$. Lie algebras of the second type also exist in our classification
$$
\begin{array}{c}
\mathfrak{m}_{0}^{r}(n), r=2l{+}1,  3\le r \le n, n \ge 3;\;\;\;  \mathfrak{m}_{1}^{r}(2m{-}1), r=2l{+}1,  3 \le r \le 2m{-}1, m\ge 3; \\ \mathfrak{m}_{0,2}(2m), m \ge 4;\;\; \mathfrak{m}_{0,3}(2m{+}1), m \ge 4; \;\; {\mathfrak n}^{\pm}_{1,1}(4); \;\; 
{\mathfrak n}_{2,s}(7), s=1,2,3. 
\end{array}
$$
We give below a table of correspondence of the notation for naturally graded quasifiliform Lie algebras from \cite{Go} with Lie algebras from our list.
$$
\begin{tabular}{|c|c|c|c|c|c|c|}
\hline
&&&&&&\\[-10pt]
&$\mathfrak{m}_{0}^{r}(n{-}1)$  &$\mathfrak{m}_{1}^{r}(2m{-}1)$ &$\mathfrak{m}_{0,2}(2m)$&$\mathfrak{m}_{0,3}(2m{+}1)$&$\begin{array}{c}{\mathfrak n}_{2,s}(7),\\ s{=}2,3.\end{array}$&${\mathfrak n}^{\pm}_{1,1}(5)$\\
&&&&&&\\[-10pt]
\hline
&&&&&&\\[-10pt]\cite{Go} &${\mathcal L}_{(n, r)}$ &$\begin{array}{c} {\mathcal Q}_{(n, r)},\\ n{=}2m{+}1\end{array}$&$\begin{array}{c}{\mathcal T}_{(n, n{-}3)},\\ n{=}2m{+}2.\end{array}$&$\begin{array}{c}{\mathcal T}_{(n, n{-}4)},\\ n{=}2m{+}3\end{array}$&$\begin{array}{c}{\mathcal E}_{(9,5)}^s,\\ s{=}2,1.\end{array}$&${\mathcal E}_{(7, 3)}$ \\[2pt]
\hline
\end{tabular}
$$

\begin{remark}
In \cite{Ga}, as a correction of the inaccuracy in \cite {Go}, the Lie algebra  ${\mathfrak E}_{9,5}^3\cong {\mathfrak n}_{2,1}(7)$ was added  to the list of naturally graded quasifiliform Lie algebras. However, as we have shown earlier (the isomorphisms (\ref{isom_m_0,3}) the Lie algebra ${\mathfrak E}_{9,5}^3\cong {\mathfrak n}_{2,1}(7)$ is isomorphic to ${\mathfrak T}_{9, 5}\cong \mathfrak{m}_{0,3}(7)$ (${\mathcal T}_{(9, 5)}$ in notations from \cite{Go}). Thus, the classification list of \cite {Go} does not require this additional Lie algebra.
\end{remark}

\end{document}